\documentclass[noinfoline
]{imsart}
\usepackage{fullpage}
\usepackage[ruled,vlined]{algorithm2e}

\usepackage[open]{bookmark}
\usepackage{comment}
\usepackage{amsthm}
\usepackage{amsmath}
\usepackage{amsfonts}
\usepackage{amssymb}
\usepackage{thmtools}

\usepackage{hyperref}

\usepackage[mathscr]{euscript}

\usepackage{enumerate}
\usepackage{multirow}
\usepackage{multicol}
\usepackage{caption} 
\usepackage{graphicx}
\usepackage{tikz}
\usetikzlibrary{arrows.meta}
\usepackage{xcolor}
\usepackage{subcaption}

\newtheorem{definition}{Definition}[section]
\newtheorem{lemma}[definition]{Lemma}
\newtheorem{theorem}[definition]{Theorem}
\newtheorem{corollary}[definition]{Corollary}
\newtheorem{remark}[definition]{Remark}
\newtheorem{example}[definition]{Example}

\allowdisplaybreaks

\newcommand\Tstrut{\rule{0pt}{2.6ex}}         
\newcommand\Bstrut{\rule[-0.9ex]{0pt}{0pt}}   
\newcommand\mr{\multirow{2}{*}}

\begin{document}
	
\title{{\large \bfseries The Completion of Covariance Kernels}}

\runtitle{The Completion of Covariance Kernels}

\author{\fnms{Kartik G.} \snm{Waghmare}\ead[label=e1]{kartik.waghmare@epfl.ch}}
\and
\author{\fnms{Victor M.} \snm{Panaretos}\ead[label=e2]{victor.panaretos@epfl.ch}}

\runauthor{K. Waghmare \& V.~M. Panaretos}

\affiliation{Ecole Polytechnique F\'ed\'erale de Lausanne}


\begin{abstract}
We consider the problem of positive-semidefinite continuation: extending a partially specified covariance kernel from a subdomain $\Omega$ of a rectangular domain $I\times I$ to a covariance kernel on the entire domain $I\times I$. For a broad class of domains $\Omega$ called \emph{serrated domains}, we are able to present a complete theory. Namely,  we demonstrate that a canonical completion always exists and can be explicitly constructed. We characterise all possible completions as suitable perturbations of the canonical completion, and determine necessary and sufficient conditions for a unique completion to exist. We interpret the canonical completion via the graphical model structure it induces on the associated Gaussian process. Furthermore, we show how the estimation of the canonical completion reduces to the solution of a system of linear statistical inverse problems in the space of Hilbert-Schmidt operators, and derive rates of convergence. We conclude by providing extensions of our theory to more general forms of domains, and by demonstrating how our results can be used to construct covariance estimators from sample path fragments of the associated stochastic process. Our results are illustrated numerically by way of a simulation study and a real example.

\end{abstract}



\maketitle

\tableofcontents

\section{Introduction}\label{sec:intro}


Consider a bivariate function $K_{\Omega}: \Omega \to \mathbb{R}$ where $\Omega$ is a subset of $I \times I$ for some interval $I$. An extension of $K_{\Omega}$ to a covariance kernel on $I$ is called a \textit{completion}. Under appropriate conditions on $\Omega$ and $K_{\Omega}$, we would like to answer the following questions: Does a completion always exist? Is there a canonical completion and can we construct it explicitly? Can we characterise the set of all completions? Can we give necessary and sufficient conditions for a unique completion to exist? Is a unique completion necessarily canonical? 

Such questions are arguably very natural from a mathematical point of view, with connections to the classical moment problem and the continuation of characteristic functions, via Bochner's theorem. They are well understood for covariance matrices and for stationary/isotropic kernel. It appears that the study of positive-definite completions was initiated by Carath\'{e}odory \cite{caratheodory1907}, who showed that every positive-definite function on a subset $\{j \in \mathbb{Z}: |j| \leq n\}$ of $\mathbb{Z}$ extends to a positive-definite function on $\mathbb{Z}$. A continuous analogue of this result was proved by Krein in \cite{krein1940} for continuous positive-definite functions on a symmetric interval of the real line, and the problem of uniqueness as well as that of description of all extensions in case of non-uniqueness was considered in the same context by Krein \& Langer \cite{krein2014}. Higher-dimensional versions of the problem where considered by Calderon \cite{calderon1952} and Rudin \cite{rudin1979}, whose results were extended to discontinuous kernels by Artjomenko \cite{artjomeko1983} and  Gneiting \& Sasv\'{a}ri \cite{gneiting1999}. A short survey of these developments can be found in Sasv\'{a}ri \cite{sasvari2005}. The case of positive-definite completions of banded matrices was considered by  Gohberg \& Dym \cite{gohberg1981}. The general case was treated by Grone et al. \cite{johnson1984}. Many of these results have been further extended to matrices with operator entries in Gohberg, Kaashoek \& Woerdeman \cite{gohberg1989} and Paulsen \cite{paulsen1989}. An extensive survey concerning the importance of positive-definite functions and kernels can be found in Stewart\cite{stewart1976}.

Our interest is to obtain a theory for the case where $I$ is an interval and $K_{\Omega}$ is not constrained to satisfy invariance properties such as stationarity.  Besides the intrinsically mathematical motivation for developing such extensions, we are motivated by the  \textit{the problem of covariance estimation from functional fragments}. Namely, estimating the covariance of a (potentially non-stationary) second-order process $X = \{X_{t}: t \in I\}$ on the basis of i.i.d. sample paths $\{X_{j}\}$ censored outside subintervals $\{I_{j}\}$, i.e. on the basis of  \textit{fragments} $X_{j} |_{I_{j}}$ drawn from $X_{I_{j}} = \{ X_{t}: t \in I_{j}\} $ for a collection of subintervals $\{I_j\}$ of $I$. Because of the fragmented nature of the observations, one is only able to estimate a restriction $K_{\Omega}$ of $K$ to a symmetric region, say $\Omega \subset I \times I$ centered around the diagonal. Nevertheless, one needs an estimator of the full covariance $K$, as this is necessary for further statistical analysis -- tasks like dimension reduction, regression, testing, and classification require the complete covariance. The problem thus reduces to ascertaining how and under what conditions one can estimate $K$ from an estimate $\hat{K}_{\Omega}$ of $K_{\Omega}$. This problem arises in a range of contexts, as documented in the references in the next paragraph. For instance, in longitudinal studies where a continuously varying random quantity (e.g. bone mineral density or systolic blood pressure is measured on each study subject over a short time interval (see Section \ref{sec:illustration} for a presentation and analysis of such an example) or in the modeling of hourly electricity pricing, where price functions are only partially observed. 

 Kraus \cite{krauss2015} originally introduced and studied a simpler version of this problem, where some samples were observable over the entire domain, hence resulting in reduced rather than no information outside $\Omega$.  Delaigle \& Hall \cite{delaigle2016} were the first to attack the genuinely fragmented problem, by imposing a (discrete) Markov assumption. Though their approach also yielded a completion, it was more focussed on predicting the missing segments. Similarly, Kneip \& Leibl \cite{kneip2020} focussed on how to optimally reconstruct the missing segments using linear prediction. The problem has been recently revisited with a firm focus on the identifiability and estimation of the complete covariance itself, see Descary \& Panaretos \cite{descarypan2019}, Delaigle et al. \cite{delaigle2020} and Lin et al. \cite{lin2019}. At a high level, they all proceed by (differently) penalized least squares fitting a finite-rank tensor product expansion over the region $\Omega$ and use it to extrapolate the covariance beyond $\Omega$. While there are substantial differences in their set up and technique, common to all three approaches is the pursuit of sufficient conditions on the process $X$ for identifiability to hold, i.e. a for uniquely existing completion. Imposing such conditions \emph{a priori} ensures that extrapolation is sensible. Starting with a strong condition in \cite{descarypan2019} (namely, analyticity), these sufficient conditions have progressively been weakened, albeit not to the point of attaining conditions that are also necessary.

We shall pursue a different approach to the problem, which we believe sheds more insight, and ultimately leads to necessary and sufficient conditions for uniqueness. Rather than start by focussing on uniqueness, we will aim at a comprehensive description of the set of all valid completions from a broad class of domains $\Omega$ called \emph{serrated domains}. Specifically, we will exhibit that a \emph{canonical completion} can always be explicitly and uniquely constructed (Section \ref{canonical-section}). Canonicity will be clearly interpreted by means of a graphical model structure on the associated Gaussian process (Section \ref{graphical-section}). We will then obtain necessary and sufficient conditions for a unique completion to exist, and discuss how these relate to the problem of identifiability (Section \ref{sec:unique-completion}). Furthermore, we will constructively characterise the set of completions as suitable perturbations of the canonical completion (i.e. show how any other valid completion can be built using the canonical completion; see Section \ref{sec:characterisation}) and parametrize this set it in terms of contractions between certain $L^2$ spaces (Theorem \ref{general-characterisation}). As for the statistical side of the question related to fragments, since a canonical solution always exists uniquely, and is equivalent to the unique completion when uniqueness holds, it is always an identifiable and interpretable target of estimation. We thus consider how to estimate it based on an estimator of the observed partial covariance, say $\hat K_{\Omega}$, and provide rates of convergence in Section \ref{sec:estimation}. We then show how our results can be adapted to more general domains $\Omega$ in Section \ref{sec:nearly-serrated} .This allows us to give a treatment of the statistical problem of covariance estimation from sample path fragments in Section \ref{sec:fragments}. The finite sample performance of statistical methodology developed is investigated by means of an extensive simulation study (Section \ref{sec:numerical_simn}) and a data analysis (Section \ref{sec:illustration}). The proofs of the proofs of our results are collected in the Supplementary Material.

 Our general perspective is inspired by the work of Grone et al. \cite{johnson1984} and Dym \& Gohberg \cite{gohberg1981} on matrices. Our methods, however, are very different, because algebraic tools such as determinants and matrix factorizations that are elemental to those works are unavailable in the kernel case. Instead, we generalize the concept of canonical extension \cite{gohberg1981} (or determinant-maximizing completion \cite{johnson1984}) to a general kernel version by demonstrating and exploiting its intimate connection to Reproducing Kernel Hilbert Spavces (RKHS) and graphical models for random processes (Theorem \ref{thm:canonical-graphical2}). An apparent consequence is that our necessary and sufficient conditions for a partial covariance to complete uniquely (Theorem \ref{uniqueness}) seem to be novel even in the context of matrices.


\section{Background and Notation}\label{sec2}

To set the context of the problem, we delineate the functions $K_{\Omega}$ that are admissible as partial covariances, and the types of domains $\Omega$ under consideration. Recall that $K: I \times I \to \mathbb{R}$ is a covariance kernel on $I$ if
	\begin{enumerate}
		\item $K(s,t) = K(t,s)$ for $s,t \in I$, and
		\item $\sum_{i,j=1}^{n} \alpha_{i} \alpha_{j}K(t_{i}, t_{j}) \geq 0$ for $n \geq 1$, $\{t_{i}\}_{i=1}^{n} \subset I$ and $\{ \alpha_{i}\}_{i=1}^{n} \subset \mathbb{R}$.
	\end{enumerate}
We shall denote the set of covariances on $I$ by $\mathscr{C}$. We shall say that $\Omega \subset I \times I$ is a \textit{symmetric domain} if $(s,t) \in \Omega$ if $(t,s) \in \Omega$ for $s,t \in I$ and $\{(t,t): t \in I\} \subset \Omega$. Since a covariance is always defined over square domains, it is natural to define partial covariances as follows:
\begin{definition}[Partial Covariance]
	Let $I$ be a set and $\Omega \subset I \times I$ be a symmetric domain. A function $K_{\Omega}: \Omega \to \mathbb{R}$ is called a partial covariance on $\Omega$ if for every $J \subset I$ such that $J \times J \subset \Omega$, the restriction $K_{J} = K_{\Omega}|_{J \times J}$ is a covariance on $J$.
\end{definition}
In the above definition, the set $J$ need not be an interval.

\begin{remark}[On Notation]
Whenever we write $K_J$ for some $J\subset I$, we will always understand that this refers to the restriction $K_{\Omega}|_{J \times J}$ of the partial covariance $K_{\Omega}$ to the square $J\times J\subseteq \Omega$.\end{remark}

A completion of the partial covariance $K_{\Omega}$ will be a function $K:I\times I\rightarrow\mathbb{R}$ such that
\begin{equation}\label{eq:compln_problem}
K \in \mathscr{C} \mbox{ and } K|_{\Omega} = K_{\Omega},
\end{equation}
The set of all possible completions of $K_{\Omega}$ will be denoted by
$$\mathscr{C}(K_{\Omega}) = \{ K \in \mathscr{C}: K|_{\Omega} = K_{\Omega}\}.$$ 
Note that our definition of partial covariance does not \emph{a priori} assume that $K_{\Omega}$ arises as the restriction of a covariance $K$ on $I$. Rather, it defines $K_{\Omega}$ intrinsically on $\Omega$. In this sense, our setting is more general than the functional fragment setting. Consequently, $\mathscr{C}(K_{\Omega})$ is not automatically non-empty. Notice however that if $K_{1}, K_{2} \in \mathscr{C}(K_{\Omega})$ then $\alpha K_{1} + (1- \alpha)K_{2} \in \mathscr{C}(K_{\Omega})$ for every $\alpha \in (0,1)$. $\mathscr{C}(K_{\Omega})$ is thus \textit{convex}. It is also bounded so long as $\sup_{t\in I} K_{\Omega}(t,t) < \infty$ because for every $K \in \mathscr{C}(K_{\Omega})$ we have $|K(s,t)| \leq \sqrt{K(s,s)K(t,t)} \leq \sup_{t\in I} K_{\Omega}(t,t)$. It follows from convexity that $\mathscr{C}(K_{\Omega})$ can either be an empty set, a singleton or have an (uncountably) infinite number of elements. Finally, the elements of $\mathscr{C}(K_\Omega)$ inherit the regularity properties of $K_\Omega$. In particular, if $K_{\Omega} \in C^{k,k}(\Omega)$, then  $K \in C^{k,k}(I \times I)$, where $C^{k,k}(\Delta)$ for a domain $\Delta \subset I \times I$ denotes the set of functions $F: \Delta \to \mathbb{R}$ such that the partial derivatives $\partial^{j}_{y}\partial^{i}_{x}F(x,y)$ and  $\partial^{i}_{x}\partial^{j}_{y}F(x,y)$ exist for $0 \leq i,j \leq k$. This is a direct consequence of the fact that the $X$ is $k$-differentiable in quadratic mean if and only if for the covariance $K$, the partial derivatives $\partial^{j}_{y}\partial^{i}_{x}K$ and  $\partial^{i}_{x}\partial^{j}_{y}K$ exist for $0 \leq i,j \leq k$ at the diagonal $\{(x,x): x \in I\}$ (see \cite{loeve1963} or \cite{saitoh2016}). Indeed, if $K_{\Omega} \in C^{k,k}(\Omega)$, then the process $X$ is $k$-differentiable in quadratic mean and hence $K \in C^{k,k}(I \times I)$.

In some cases, we will need to work with the covariance \emph{operators} associated with the corresponding covariance kernels. For $\mu$ a measure on the Borel sets of $I$, and $S\subseteq I$ we will define the Hilbert space $L^2(S)$ to be the set of all $f:S\rightarrow \mathbb{R}$ such that
$\int_S f^2(x) ~s\mu(x)<\infty$
with associated inner product
\begin{equation*}
\langle f, g \rangle_{2} = \int_{S} f(x) g(x) \mu(dx),\qquad f,g \in L^{2}(S).
\end{equation*}
Since continuity of $K_{\Omega}$ implies continuity of any completion thereof, any completion $K$ induces a Hilbert-Schmidt integral operator $\mathbf{K}: L^{2}(I) \to L^{2}(I)$ given by
\begin{equation*}
\mathbf{K}f(x) = \int_{I} K(x,y) f(y) \mu(dy),\qquad \mu-\mathrm{a.e.},
\end{equation*}
i.e. an operator with $K$ as its integral kernel. Similarly, any restriction $K|_{S\times S}$ on a square domain induces an integral operator $\mathbf{K}|_S: L^{2}(S) \to L^{2}(S)$ by way of
$$\mathbf{K}|_S\,g(x) = \int_{S} K(x,y) g(y) ~d\mu(y),\qquad \mu|_{S}-\mathrm{a.e.}$$
The operator norm $\|\cdot\|_\infty$ and Hilbert-Schmidt norm $\|\cdot\|_2$ of an operator $\mathbf{K}: L^2(S_1)\rightarrow L^2(S_2)$, $S_1,S_2\subseteq  I$, with continuous kernel $K:S_1\times S_2\rightarrow\mathbb{R}$ will be defined via
$$\| \mathbf{K} \|^2_\infty = \sup_{f\in L^2(S_1)\setminus\{0\}} \frac{\int_{S_2}\left(\int_{S_1}K(u,v)f(v)\mu(dv)\right)^2\mu(du)}{\int_{S_1}f^2(u)\mu(du)}$$
and
$$\| \mathbf{K} \|^2_2=\int_{S_1}\int_{S_2}K^2(u,v)\mu(du)\mu(dv).$$
The positive root of an operator $\mathbf{A}$ will be denoted by $|\mathbf{A}|=(\mathbf{A}\mathbf{A}^*)^{1/2}$. We denote the space of Hilbert-Schmidt operators from $L^2(S_1)$ to $L^2(S_2)$ as $\mathcal{S}_{2}(S_{1},S_{2})$. The image of a subset $\mathcal{S}_1\subseteq L^2(S_1)$ via the operator $\mathbf{K}: L^2(S_1)\rightarrow L^2(S_2)$ will be simply denoted as $\mathbf{K}\mathcal{S}_1=\{\mathbf{K}f: f\in \mathcal{S}_1\}$. We shall use the same convention for operator multiplication, for example, we denote $\mathbf{K}\mathcal{S}_{2}(S_{3},S_{1}) = \{ \mathbf{K}\mathbf{A}: \mathbf{A} \in \mathcal{S}_{2}(S_{3},S_{1})\}$.

Given a Hilbert-Schmidt operator $\mathbf{K}:L^2(S_1)\rightarrow L^2(S_1)$ with integral kernel $K:S_1\times S_1\rightarrow\mathbb{R}$, we define the Reproducing Kernel Hilbert Space (RKHS) of $K$ (equivalently of $\mathbf{K}$) as the Hilbert space $\mathscr{H}({K})=\mathbf{K}^{1/2}L^2(S_1)$, endowed with the inner product 
$$\langle f,g\rangle_{\mathscr{H}({K})}:=\langle \mathbf{K}^{-1/2}f,\mathbf{K}^{-1/2}g\rangle_{L^2(S_1)},\qquad f,g\in\mathscr{H}({K}).$$

As for the types of symmetric domains $\Omega$ under consideration, our main focus will be on serrated domains:

\begin{definition}[Serrated Domain]\label{def:serrated}
A domain $\Omega\subseteq{I\times I}$ is called serrated if it can be written as a union $\Omega=\cup_j (I_j\times I_j)$
for $\{I_j\}$ a finite collection of subintervals covering of $I$. 
\end{definition}

\begin{figure}
\includegraphics[scale=0.11625]{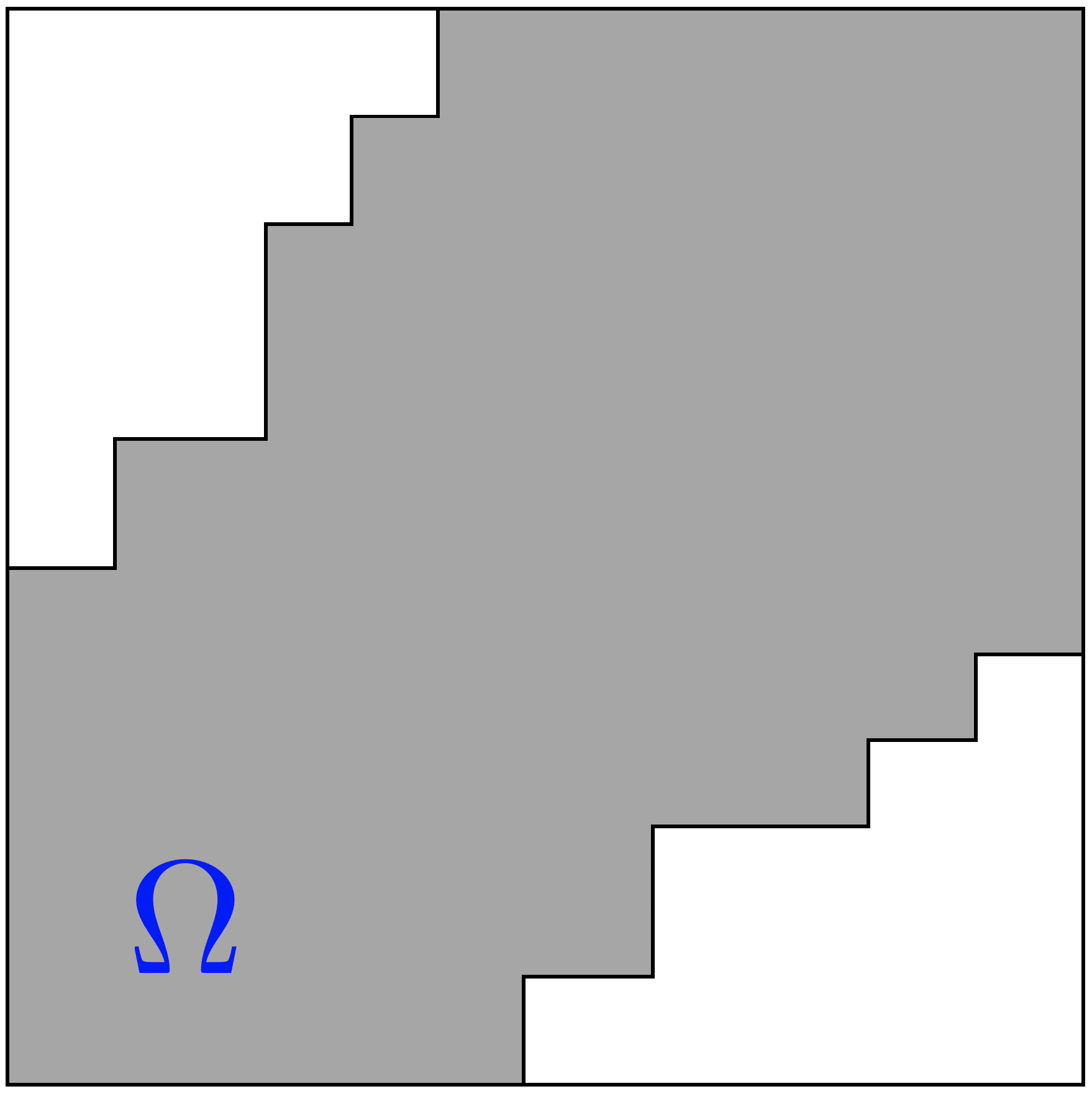}$\qquad$\includegraphics[scale=0.12]{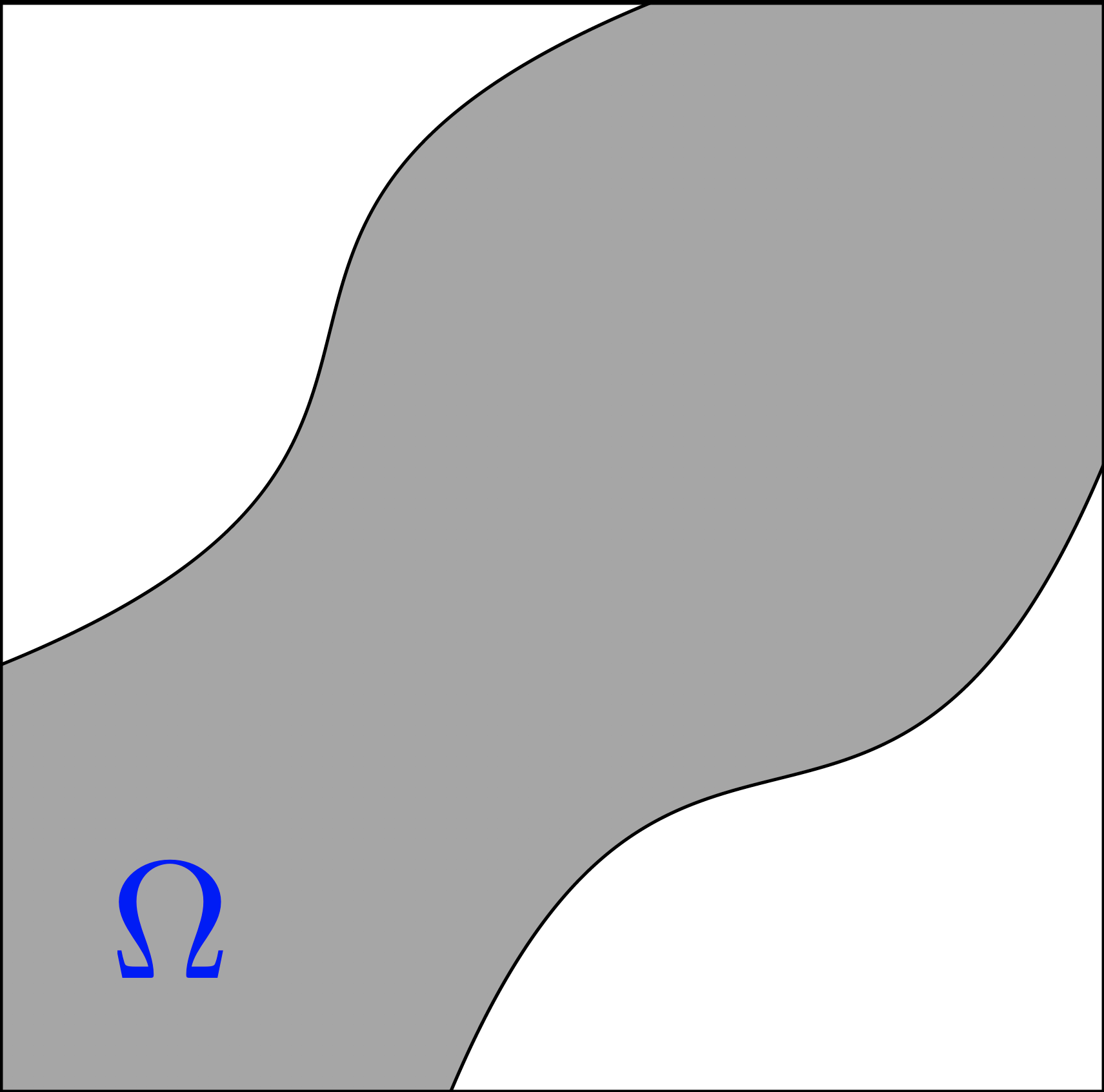}
\caption{A serrated domain (left) and a nearly serrated domain (right).}
\label{fig:serrated}
\end{figure}

Informally, a serrated domain consists of a collection of squares of varying sizes, strung symmetrically along the diagonal in a manner that covers it (see Figure \ref{fig:serrated}). When restricting attention to matrices or stationary kernels, serrated domains reduce to the types of domains on which the problem has been previously studied. In the functional fragments problem, the observation of a finite collection of path fragments $X_j|_{I_j}$ leads to partial covariance information on the serrated domain $\Omega=\cup_j (I_j\times I_j)$. By taking sequences of covers consisting of progressively more squares, serrated domains can approximate a very rich class of symmetric domains that we call \emph{nearly serrated} (see Figure \ref{fig:serrated} and Definition \ref{def:nearly_serr_domain}  for a rigorous definition). In the next sections, we develop an essentially complete theory of completion for serrated domains. Then, Section \ref{sec:nearly-serrated}  demonstrates how our results on serrated domains can be used to obtain results for nearly serrated domains.

\section{The Canonical Completion}\label{canonical-section}

Recall that the set of completions $\mathscr{C}(K_{\Omega})$ of a partial covariance $K_{\Omega}$ can be empty, a singleton, or uncountably infinite. We will now show that for $\Omega$ a serrated domain, $\mathscr{C}(K_{\Omega})$ is never empty. We will do so by explicitly constructing a completion $K_{\star}$, that will be subsequently argued to be canonical.

It is instructive to commence with the \emph{2-serrated case}, i.e. when $\Omega=(I_1\times I_1)\cup (I_2\times I_2)$ for two intervals $\{I_j\}_{i=1}^{2}$ such that $I_1\cup I_2 \supseteq I$, depicted in Figure \ref{fig:2serrated} (left). Define a function $K_{\star}:I\times I\rightarrow \mathbb{R}$ as follows:
\begin{equation}\label{2-serrated-completion}
K_{\star}(s,t)=\begin{cases}
K_{\Omega}(s,t), &(s,t)\in\Omega\\
\big\langle K_{\Omega}(s,\cdot), K_{\Omega}(\cdot, t) \big\rangle_{\mathscr{H}(K_{I_1\cap I_2})}, &(s,t) \notin \Omega.
\end{cases}
\end{equation}
Here, $K_{I_1\cap I_2}= K_{\Omega}|_{(I_1\cap I_2)^2}$ is the restriction of the partial covariance $K_{\Omega}$ to the square $(I_1\cap I_2)\times (I_1\cap I_2)$ and $\mathscr{H}(K_{I_1\cap I_2})$ is the RKHS of $K_{I_1\cap I_2}$. It is implicit in the notation $\langle K_{\Omega}(s,\cdot), K_{\Omega}(\cdot, t) \rangle_{\mathscr{H}(K_{I_1\cap I_2})}$ that the domain of $K_{\Omega}(s,\cdot)$ and $K_{\Omega}(\cdot, t)$ is automatically restricted to $I_1\cap I_2$ within that inner product, as depicted in Figure \ref{fig:2serrated} (right). 

\begin{remark} The reproducing kernel inner product in Equation \ref{2-serrated-completion} can be seen as the infinite-dimensional equivalent of matrix multiplication formulas appearing in maximum entropy matrix completion \cite{johnson-matrix} and low rank matrix completion \cite{descarypan2019b}.
\end{remark}

Our first result is now:
\begin{theorem}[Canonical Completion from a 2-Serrated Domain]\label{2-serrated-theorem}
Given any partial covariance $K_{\Omega}$ on a 2-serrated domain $\Omega\subseteq I\times I$, the function $K_{\star}:I\times I\rightarrow \mathbb{R}$ defined in \eqref{2-serrated-completion} is a well-defined covariance that always constitutes a valid completion, i.e. 
$$K_{\star}\in\mathscr{C}(K_{\Omega}).$$
In particular, if $K_{\Omega}$ admits a unique completion, then this must equal $K_{\star}$.
\end{theorem}

The second part of the theorem hints at why we refer to the completion $K_{\star}$ as the \emph{canonical completion} of $K_{\Omega}$. We will provide a more definitive reason in Section \ref{graphical-section},  but first we will use the formula from the 2-serrated case in order to extend our result to a general serrated domain.

\begin{figure}[h]
\includegraphics[scale=0.25]{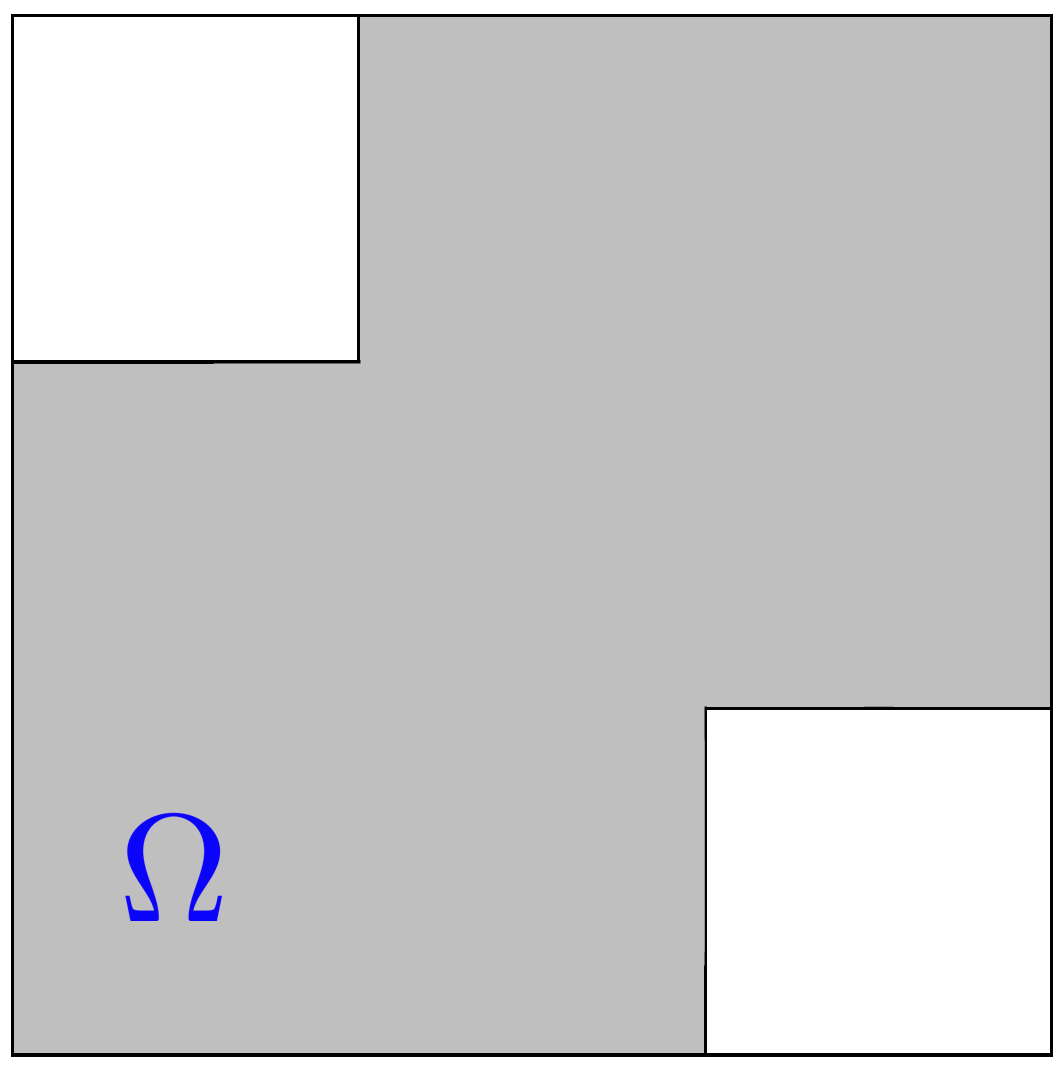} $\qquad$\includegraphics[scale=0.228]{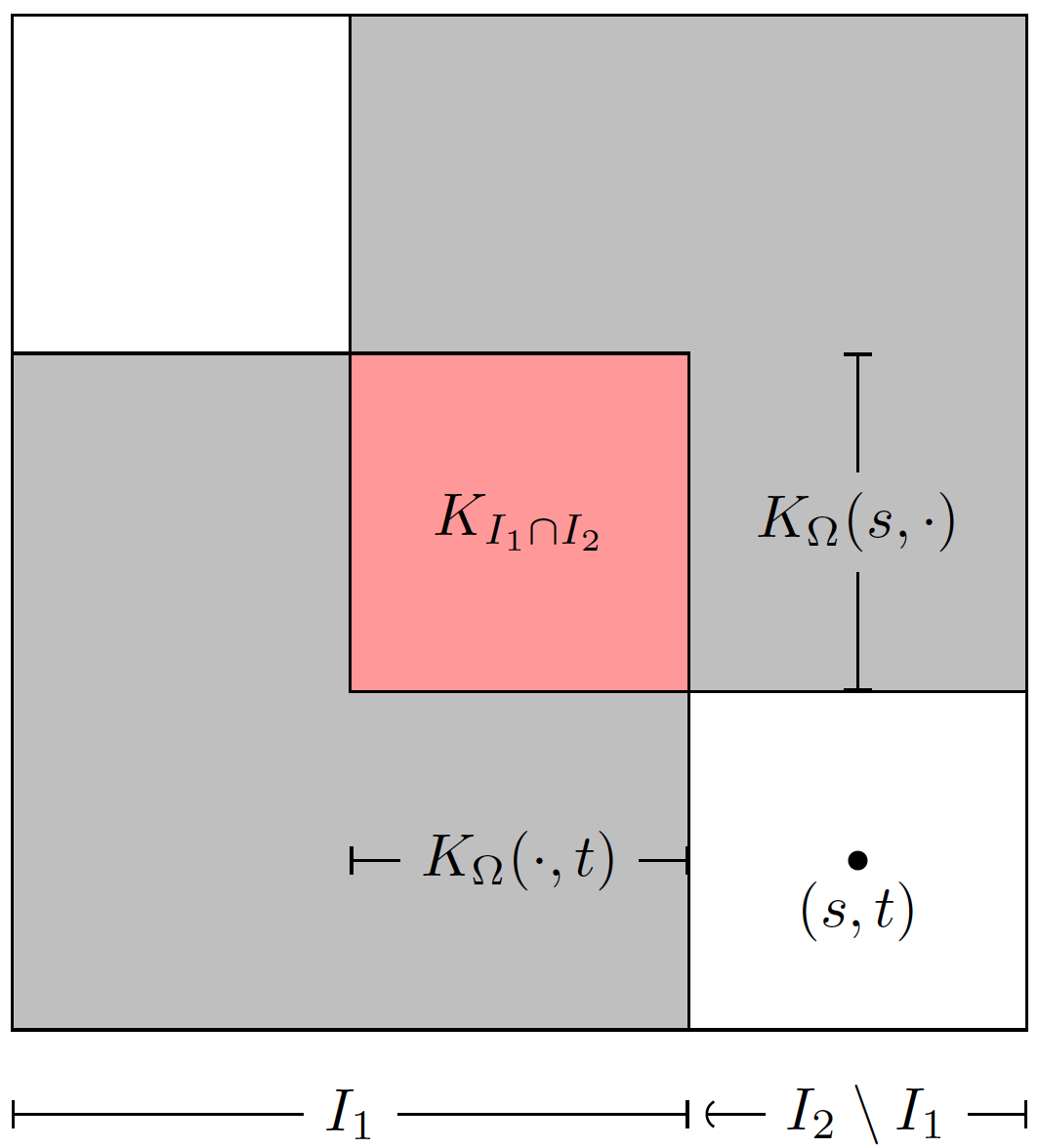} 
\caption{A two serrated domain (left) and a heuristic illustration of the formula for $K_{\star}$.}
\label{fig:2serrated}
\end{figure}

We will do this iteratively. Intuitively, if we have a general serrated domain generated by a cover of $m$ subintervals $\{I_1,...,I_m\}$, we can apply the 2-serrated formula to any pair of successive squares $\{I_p,I_{p+1}\}$, to reduce to the problem to one of completion from a serrated domain generated the reduced set of $m-1$ subintervals $\{I_1,...,I_{p-1},I_{p}\cup I_{p+1}, I_{p+2},...,I_m\}$ (see Figure \ref{fig:iteration}). Repeating the same prescription, we can eventually complete $K_{\Omega} $ to a covariance on $I$. To be more precise, let $\Omega=\cup_{j=1}^{m} (I_j\times I_j)$ be an $m$-serrated domain and for notational ease assume that the indices of the $\{I_j\}$ correspond to their natural partial ordering as intervals. Define the intersection of any two successive squares as
$$J_p = (I_p\times I_{p})\cap (I_{p+1}\times I_{p+1})$$
and the corresponding restriction of $K_{\Omega}$ as
$K_{J_p}=K_{\Omega}|_{J_p\times J_p}$. Next define the square of the union of the intervals $\{I_1,...,I_p\}$ as 
$$U_p=(I_1\cup\hdots \cup I_p)\times (I_1\cup\hdots \cup I_p).$$
Finally, define the serrated domain generated by the cover $\{\cup_{j=1}^{p}I_j,I_{p+1},...,I_m\}$ as
$$\Omega_{p}=U_p\bigcup \left\{ \cup_{j=p+1}^{m} (I_j\times I_j) \right\}$$
noting that $\Omega_1=\Omega$. 

The following algorithm uses the formula from the 2-serrated case to extend $K_{\Omega}$ to a partial covariance on $\Omega_2$, then $\Omega_3$, and so on, until completion to covariance on $I=\Omega_m$:

\begin{algorithm}[H]
\SetAlgoLined
\begin{enumerate}

\medskip
\item Initialise with the partial covariance $K_1 = K_{\Omega}$ on $\Omega_1=\Omega$.

\medskip
\item For $p\in \{1,...,m-1\}$ define the partial covariance $K_{p+1}$ on $\Omega_{p+1}$ as
$$
K_{p+1}(s,t)=\begin{cases}
K_p(s,t), &(s,t)\in \Omega_p\\
\big\langle K_p(s,\cdot), K_p(\cdot, t) \big\rangle_{\mathscr{H}(K_{J_{p}})}, &(s,t) \in \Omega_{p+1}\setminus\Omega_p.
\end{cases}
$$

\medskip
\item Output the covariance $K_{\star}=K_{m}$ on $I\times I = \Omega_m$.
\end{enumerate}
 \caption{$m$-Serrated Completion by Successive $2$-Serrated Completions}
 \label{algorithm}
\end{algorithm}

Of course, there is nothing special about the application of the iterative completion in ascending order. We could have set up our notation and algorithm using a descending order starting with $\{I_m,I_{m-1}\}$, or indeed using an arbitrary order, starting from any pair of successive squares $\{I_p,I_{p+1}\}$ and moving up and down to neighbouring squares. Our second result shows that, no matter the chosen order, the algorithm will output a valid completion $K_{\star}\in \mathscr{C}(K_{\Omega})$:

\begin{theorem}[Canonical Completion from a General Serrated Domain]\label{serrated-theorem}
The recursive application of the 2-serrated formula as described in Algorithm \eqref{algorithm} to a partial covariance $K_\Omega$ on a serrated domain $\Omega$ yields the same valid completion $K_{\star}\in \mathscr{C}(K_\Omega)$, irrespective of the order it is applied in.  In particular, if $K_{\Omega}$ admits a unique completion, then this must equal $K_{\star}$.

\end{theorem}

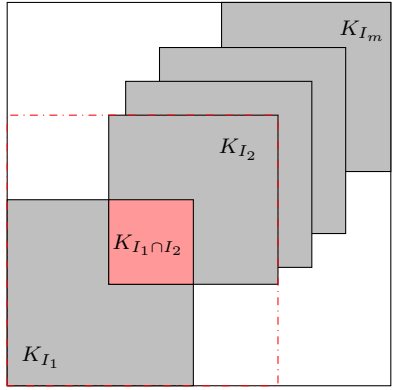
\begin{figure}
	\begin{tikzpicture}[scale=0.75]
	\draw [shift={(0.5,0.5)}] (1.2,1.2) -- (1.2,8) -- (8,8) -- (8,1.2) -- cycle;
	\draw [shift={(1.7,1.7)}, fill=gray!50] (0,0) -- (0,3.3) -- (3.3,3.3) -- (3.3,0) -- cycle;
	
	\draw [shift={(5.5,5.5)}, fill=gray!50] (0,0) -- (0,3) -- (3,3) -- (3,0) -- cycle; 
	\draw [shift={(4.4,4.4)}, fill=gray!50] (0,0) -- (0,3.3) -- (3.3,3.3) -- (3.3,0) -- cycle;
	\draw [shift={(3.8,3.8)}, fill=gray!50] (0,0) -- (0,3.3) -- (3.3,3.3) -- (3.3,0) -- cycle;
	\draw [shift={(3.5,3.5)}, fill=gray!50] (0,0) -- (0,3) -- (3,3) -- (3,0) -- cycle;
	
	\draw [dash dot, shift={(0.5,0.5)}, red] (1.2,1.2) -- (1.2,6) -- (6,6) -- (6,1.2) -- cycle;
	
	\draw [shift={(3.5,3.5)}, fill=red!40] (0,0) -- (0,1.5) -- (1.5,1.5) -- (1.5,0) -- cycle;
	\draw [shift={(0.7,0.7)}] (1.6,1.5) node [fill=gray!50] {$K_{I_{1}}$};
	\draw [shift={(0.5,0.5)}] (3.7,3.7) node {$K_{I_{1} \cap I_{2}}$};
	\draw [shift={(0.5,0.5)}] (5.3,5.4) node {$K_{I_{2}}$};
	\draw [shift={(0.5,0.5)}] (7.5,7.5) node  {$K_{I_{m}}$};

	\end{tikzpicture}
\caption{Illustration of the iterative completion procedure for a general serrated domain.}
\label{fig:iteration}
\end{figure}

Notice that Theorems \ref{2-serrated-theorem} and \ref{serrated-theorem} make no assumption on $K_{\Omega}$ except that it be a partial covariance. In particular, $K_{\Omega}$ need not be continuous or even bounded.  

\begin{example}[Brownian Motion]
As an example, consider the following partial covariance on a 2-serrated subdomain of $[0,1]^2$:
$$K_{\Omega}(s,t) = s \wedge t, \quad (s,t)\in \Omega=\underset{I_1}{\underbrace{([0,2/3] \times [0,2/3])}} \cup \underset{I_2}{\underbrace{([1/3,1] \times [1/3,1])}}.$$
Clearly, this can be completed to the covariance of standard Brownian motion on $[0,1]^2$,
$$K(s,t)=s\wedge t,\qquad (s,t)\in[0,1]^2.$$
To see what our completion algorithm yields, we note that the restriction $K_{[1/3,2/3]}$ yields the RKHS inner product
$$\langle f,g\rangle_{\mathscr{H}(K_{[1/3,2/3]})}=\frac{1}{(1/3)}\int_{1/3}^{2/3}f'(u)g'(u)du.$$
Thus, for $s \in (2/3,1]$ and $t\in [0,1/3)$,
$$K_{\star}(s,t)=\frac{1}{(1/3)}\int_{1/3}^{2/3}\underset{=1}{\underbrace{\frac{\partial}{\partial u}K_{\Omega}(s,u)}}\underset{=t}{\underbrace{ \frac{\partial}{\partial u}K_{\Omega}(u, t)}} du=t=s\wedge t,\quad\mbox{ since }t<s.$$
Iterating, we can directly see that the extension of a partial covariance that has the form $s\wedge t$ on an arbitrary domain by means of Algorithm \ref{algorithm} will also yield the covariance of Brownian motion.
\end{example}

The example illustrates that Algorithm \ref{algorithm} yields the ``right'' answer in an important special case. The next section demonstrates that this is no accident, and that the completions given in Theorems \ref{2-serrated-theorem} and \ref{serrated-theorem} are indeed canonical in a strong sense.

\section{Canonicity and Graphical Models}\label{graphical-section}

We will now interpret the canonical completion via the conditional independence structure it induces on the associated Gaussian process. Recall that an undirected graph $G$ on a set $I$ is an ordered pair $G=(I, \Omega)$ where $I$ is called the vertex set and $\Omega \subset I \times I$ is the edge set such that $(s,t) \in \Omega$ if and only if $(t,s) \in \Omega$. We shall often refer to the graph $(I, \Omega)$ as $\Omega$. Notice that if $I$ is an interval of the real line, then a symmetric domain $\Omega$ induces an uncountable graph on $I$ with $\Omega$ serving as the edge set. 

We shall say that $S \subset I$ \textit{separates} $s, t \in I$ with respect to the graph $(I,\Omega)$ if every path from $s$ to $t$ comprised of edges in $\Omega$ is intercepted by $S$, that is, for every $\{t_{i}\}_{i=1}^{r} \subset I$ with $r \geq 1$ such that $t_{1} = s$, $(t_{i}, t_{i+1}) \in \Omega$ for $1 \leq i < r$ and $t_{r} = t$, we have that $t_{j} \in S$ for some $1 < j < r$. 

A graph $(I,\Omega)$ induces a conditional independence structure on a Gaussian process $X_I:=\{X_t:t\in I\}$ much in the same way as in the finite dimensional case.
\begin{definition}[Graphical Models on Gaussian Processes]\label{graphical-definition}
	The Gaussian process $X = \{X_{t}: t \in I\}$ is said to form an undirected graphical model over the graph $\Omega \subset I \times I$, if for every $s,t \in I$ separated by $J \subset I$, we have 
\begin{equation}\label{citoc}
\mathrm{Cov}(X_{s}, X_{t}| X_{J}) \equiv \mathbb{E}\left[ (X_{s} - \mathbb{E}\left[X_{s}|X_{J}\right]) (X_{t} - \mathbb{E}\left[X_{t}|X_{J}\right]) | X_{J}\right] = 0\quad \mbox{ a.s.}
\end{equation}

\end{definition}
Equation \eqref{citoc} implies that $\mathbb{E}\left[ X_{s}X_{t}|X_{J} \right] = \mathbb{E}\left[X_{s}|X_{J}\right] \mathbb{E}\left[X_{t}|X_{J}\right]$ almost surely. Taking the expectation gives 
\begin{equation}\label{ciwithce}
	\mathbb{E}\left[X_{s}X_{t}\right] = \mathbb{E}\left[ \mathbb{E}\left[X_{s}|X_{J}\right] \mathbb{E}\left[X_{t}|X_{J}\right] \right],
\end{equation} 
i.e. the covariance of $X_{s}$ and $X_{t}$ coincides with that of their best predictors given $X_{J}$ when $J$ separates $s$ and $t$. Notice that from (\ref{ciwithce}), it follows that
\begin{equation*}
	\mathbb{E}\left[ \mathbb{E}\left[X_{s}|X_{J}\right] \mathbb{E}\left[X_{t}|X_{J}\right] \right] 
	= \mathbb{E}\left[ \mathbb{E}\left[ X_{s}\mathbb{E}\left[X_{t}|X_{J}\right] |X_{J} \right]\right]=  \mathbb{E}\left[X_{s}\mathbb{E}\left[X_{t}|X_{J}\right]\right]
\end{equation*}
and thus,
$\mathbb{E}\left[(X_{t} - \mathbb{E}[X_{t}|X_{J}]) X_{s}\right] = 0$
which implies that
$\mathbb{E}[X_{t}|X_{J}] = \mathbb{E}[X_{t}|X_{J}, X_{s}]$.
Similar reasoning yields,
\begin{equation}\label{markov}
	\mathbb{E}\left[ f(X_{t}) | X_{J}, X_{s} \right] = \mathbb{E}\left[ f(X_{t}) | X_{J} \right]
\end{equation}
which  is reminiscent of Markov process, where
\begin{equation}\label{markovp}
\mathbb{E}\left[ f(X_{t}) | \{X_{u}: u \leq v\} \right] = \mathbb{E}\left[ f(X_{t}) | X_{v} \right]
\end{equation}
Indeed, the undirected graphical model structure induced by $\Omega$ is a natural generalization of Markov processes structure, but with a notion of separation stemming from the graph structure rather than simple time ordering. In the terminology of Markov random fields, Definition \ref{graphical-definition} is equivalent to the global Markov property with respect to $\Omega$.

\begin{theorem}\label{thm:canonical-graphical}
Let $K_\Omega$ be a partial covariance on a {serrated} domain $\Omega\subset I$. The canonical completion $K_{\star}$ is the only completion of $K_{\Omega}$ such that the associated Gaussian process $\{X_t:t\in I\}$ forms an undirected graphical model with respect to the graph $G=([0,1],\Omega)$.
\end{theorem}

Said differently, $K_{\star}$ is the only completion of $K_\Omega$ that possesses the global Markov property with respect to the edge said $\Omega$. Intuitively, the canonical completion is the unique completion to rely exclusively on correlations intrinsic to the ``observed'' set $\Omega$: it propagates the ``observable'' correlations of $K_\Omega$ to the rest of $I$ via the Markov property, without introducing an extrinsic ``unobserved'' correlations. By contrast, any other completion will introduce correlations extrinsic to those observed via $K_{\Omega}$. This last statement is considerably refined in Section \ref{sec:characterisation}, where we characterise all possible completions as perturbations of the canonical completion. 

In closing this section, we give a result going in the opposite direction: namely we show that a Gaussian process admits a graphical model structure w.r.t. a serrated $\Omega$ if and only if it has a covariance that satisfies the defining equations \eqref{2-serrated-completion}  of a canonical completion. This is a result that is of interest in its own right, since it characterises the set of all Gaussian process graphical models compatible with $\Omega$ and in doing so provides a more convenient way of expressing conditional independence relations in a Gaussian process than say, cross-covariance operators defined by Baker \cite{baker1973}. To state it rigorously, define the set of covariances
	\begin{equation*}
\mathscr{G}_{\Omega} = \left\lbrace K \in \mathscr{C}: K(s,t) = \big\langle K(s,\cdot), K(\cdot, t) \big\rangle_{\mathscr{H}(K_{J})} \mbox{ for all } J \subset I \mbox{ separating } s,t \in I \mbox{ in } \Omega \right\rbrace. 
	\end{equation*}
We can now state:
\begin{theorem}\label{thm:canonical-graphical2}
	Let $\{X_t:t\in I\}$ be a Gaussian process with covariance $K$. Then, $X$ forms an undirected graphical model with respect to a serrated $\Omega$ if and only if $K \in \mathscr{G}_{\Omega}$.
\end{theorem}
There is actually no reason to restrict attention to Gaussian processes, and we did this solely for interpretability: for a Gaussian process, the condition $K\in \mathscr{G}_\Omega$ can be interpreted in terms of  conditional independence. But we can more generally define a ``second order graphical model'' as long as we focus solely on conditional uncorrelatedness rather than conditional independence -- just take Definition \ref{graphical-definition} and drop the word ``Gaussian'', while replacing ``graphical model'' by ``second order graphical model''. 

\section{Necessary and Sufficient Conditions for Unique Completion}\label{sec:unique-completion}

We will now state necessary and sufficient conditions guaranteeing unique completion from a serrated domain $\Omega\subset I$. And we will argue that identifiability can occur even without enforcing the existence of a unique extension. For this, we need some additional notation. Given $A\subset B \subset \Omega$, let $K_{B}/ K_A$ be the Schur complement of $K_B$ with respect to $K_A$, 
$$(K_{B}/ K_A)(s,t)=K_B(s,t)-\big\langle K_B(s,\cdot), K_B(\cdot,t) \big\rangle_{\mathscr{H}(K_A)}$$
i.e. the covariance of the \emph{residuals} $\{X_{t} - \Pi(X_{t}|X_{A}): t \in B \setminus A\}$, where $\Pi(W|Z)$ is the best linear predictor of $W$ given $Z$. We now have: 

\begin{theorem}[Unique Completion from a Serrated Domain]\label{uniqueness}
	Let $K_{\Omega}$ be a partial covariance kernel on a serrated domain $\Omega\subset I$ of $m$ intervals. The following two statements are equivalent:
	\begin{enumerate}
	
		\item[(I)] $K_{\Omega}$ admits a unique completion on $I$, i.e. $\mathscr{C}(K_{\Omega})$ is a singleton.
		
		
		\item[(II)] there exists an $ r \in\{1,\ldots, m\}$, such that 
			$$K_{I_{p}}/K_{I_{p}\cap I_{p+1}}=0, \mbox{ for }1 \leq p < r\quad\mbox{ and }\quad K_{I_{q+1}}/K_{I_{q}\cap I_{q+1}}=0, \mbox{ for }r \leq q < m.$$  
			\end{enumerate}
\end{theorem}

Condition (II) is strictly weaker than any of the sufficient conditions that have previously been stated in the literature on functional fragments (e.g. \cite[Theorem 1]{delaigle2020} and \cite[Proposition 2]{descarypan2019}). Consequently, none of those conditions is necessary in the context of a serrated domain (for a discussion of more general domains, see Section \ref{sec:nearly-serrated}). Furthermore, an appealing feature of (II) is that it is checkable at the level of $K_{\Omega}$ in a concrete manner by constructing a finite number of Schur complements (in fact the number is linear in $m$).

Theorem \ref{uniqueness} elucidates just how restrictive it is to a priori assume that a unique completion exists. When the Schur complements involved in (II) vanish, one can start with the associated process $\{X_t:t\in I_r\}$ restricted to $I_r$, and iteratively \emph{perfectly} predict each segment $\{X_t:t\in I_j\}$ by means of best linear prediction. Consequently, the entire process $\{X_t:t\in I\}$ is generated as the image of its restriction $\{X_t: t\in I_r\}$ via a \emph{deterministic linear operator}. Indeed which interval(s) $\{I_j\}_{j=1}^{m}$ generate(s) the process can be discovered by checking the equations given in (II).

Note, however, that being able to identify $K$ from $K|_\Omega$ does not require assuming that $K|_\Omega$ completes uniquely -- all we need is a way to select one element from $\mathscr{C}(K|_\Omega)$. For example, to obtain identifiability, it would be much less restrictive to assume the admittance of a (second order) graphical model with respect to $(I,\Omega)$. The set of covariances $\mathscr{K}_\Omega$ corresponding to such processes is potentially very large, and can be highly ``non-deterministic'' in its dependence structure. Assuming that $K\in \mathscr{G}_\Omega$ will then yield identifiability given $K|_\Omega$ via Theorem \ref{thm:canonical-graphical}, which can be re-interpreted in this notation as stating
$$\mathscr{C}(K|_\Omega)\cap \mathscr{G}_\Omega = \{K_{\star}\}$$
Since a unique completion is automatically canonical, it must also lie in $\mathscr{G}_\Omega$. Therefore, the assumption $K\in \mathscr{G}_\Omega$ is \emph{strictly weaker} than the uniqueness assumption, while still guaranteeing identifiability. As noted earlier, in the last paragraph of Section \ref{graphical-section}, one can easily define a ``second-order graphical model'' structure with conditional uncorrelatedness replacing conditional independence, so imposing the assumption $K\in \mathscr{G}_\Omega$ in no way entails assuming Gaussianity.  The family $\mathscr{G}_{\Omega}$ can also be thought of as a covariance selection model of the kind first proposed by Dempster \cite{dempster1972} for multivariate normal distributions, so that imposing the condition $K \in \mathscr{G}_{\Omega}$ amounts to doing continuous-domain parameter reduction.

\section{Characterisation of All Completions}\label{sec:characterisation}

We will now show how all the elements of $\mathscr{C}(K_\Omega)$ can be spanned by suitable perturbations of the canonical completion, when $\Omega$ is serrated. Again, it is instructive to commence with the the 2-serrated case (see the left plot in Figure \ref{fig:2serrated}, p. \pageref{fig:2serrated}).

\begin{theorem}[Characterisation of Completions in the 2-Serrated Case]\label{2serrated-characteriation}
Let $\Omega=(I_1\times I_1)\cup (I_2\times I_2)$ be a two serrated subdomain in $I\times I$. The function $K:I\times I\rightarrow\mathbb{R}$ is a completion of $K_{\Omega}:\Omega\rightarrow\mathbb{R}$ {if and only if} for some covariance 
	\begin{equation*}
		K = K_{\star} + C
	\end{equation*}
	where $C: [I_{1} \cup I_{2}]^{2} \to \mathbb{R}$ satisfies $C(s,t) = 0$ for $(s,t) \in I_{1}^{2} \cup I_{2}^{2}$ such that the function $L: [(I_{1} \setminus I_{2}) \cup (I_{2} \setminus I_{1})]^{2} \to \mathbb{R}$ given by
	\begin{alignat*}{3}
		L|_{(I_{1} \setminus I_{2})^{2}} &= K_{I_{1}}/K_{I_1\cap I_2}, & \quad L|_{(I_{1} \setminus I_{2}) \times (I_{2} \setminus I_{1})} &= C|_{(I_{1} \setminus I_{2}) \times (I_{2} \setminus I_{1})},\\
		L|_{(I_{2} \setminus I_{1})^{2}} &= K_{I_{2}}/K_{I_1\cap I_2}, & \quad L|_{(I_{2} \setminus I_{1}) \times (I_{1} \setminus I_{2})} &= C|_{(I_{2} \setminus I_{1}) \times (I_{1} \setminus I_{2})}
	\end{alignat*}
	is a covariance.
\end{theorem}

\begin{figure}
\includegraphics[scale=0.3]{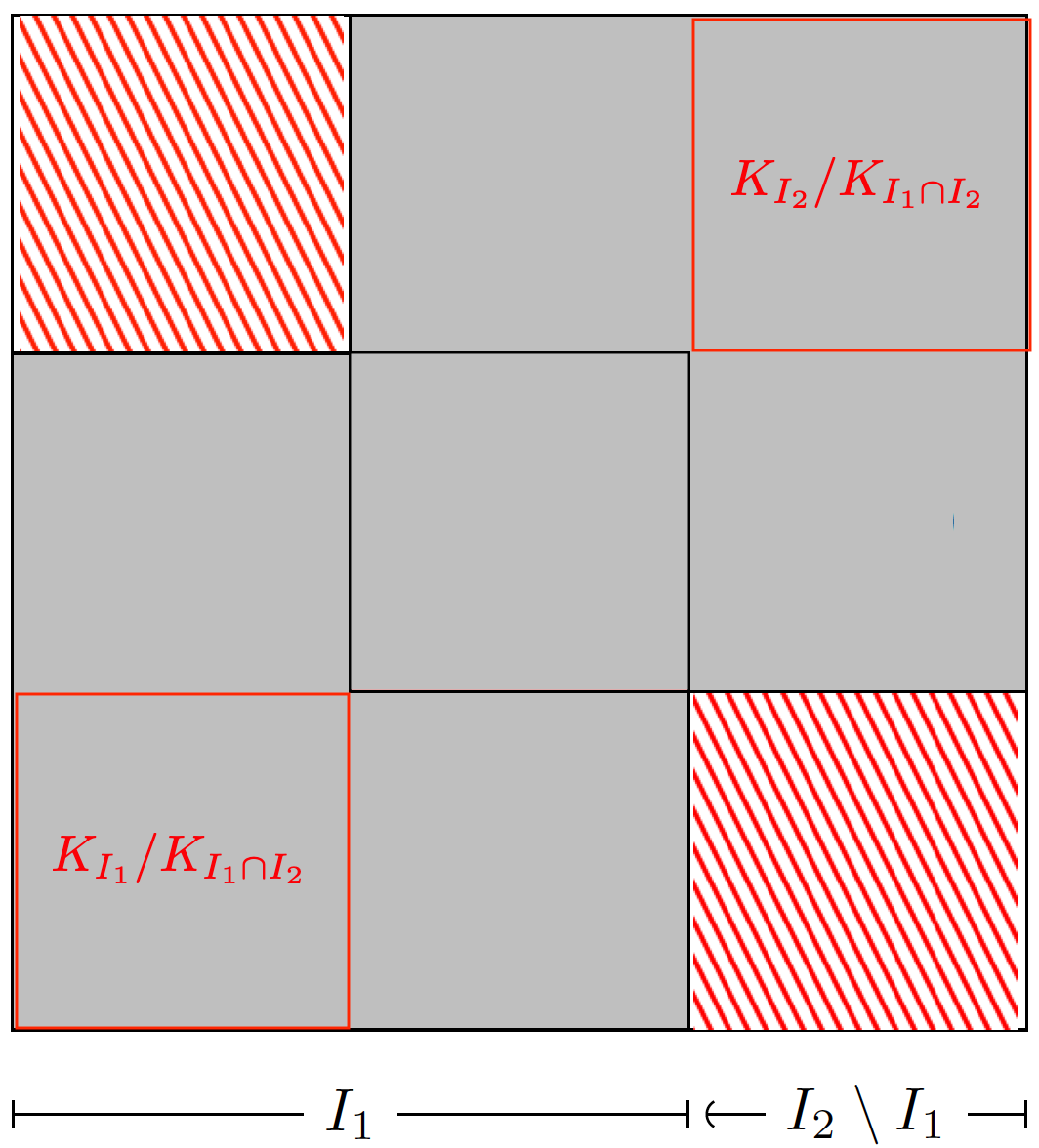}
\caption{Illustration of Theorem \ref{2serrated-characteriation}. The 2-serrated domain $\Omega$ is shaded in grey, and the central square is $(I_1\cap I_2)^2$. The set $\mathscr{C}(K_\Omega)$ is spanned as $K_{\star}+C$, where $C$ ranges over cross covariances supported on the union of the two squares shaded in red, and compatible with the covariances $K_{I_1}\setminus K_{I_1\cap I_2}$ and $K_{I_2}\setminus K_{I_1\cap I_2}$.}
\end{figure}

\noindent Said differently, in the 2-serrated case $\Omega=(I_1\times I_1)\cup (I_2\times I_2)$ one has
$$K\in \mathscr{C}(K_\Omega)\iff K = K_{\star}+C $$
where $K_{\star}$ is the canonical completion and $C$ is a valid perturbation. The set of all valid perturbations $C$ is given by the cross-covariances $C|_{(I_1\setminus I_2)\times (I_2\setminus I_1)}$ (with $C|_{(I_2\setminus I_1)\times (I_1\setminus I_2)}$ determined by symmetry, i.e. $C(s,t) = C(t,s)$) corresponding to all possible couplings $(Y_t,W_t)$ of the Gaussian processes 
\begin{eqnarray*}
\{Y_t:t\in I_{1} \setminus I_2\},\quad Y&\sim&N (0,K_{I_{1}}/K_{I_1\cap I_2})\\
\{W_t:t\in I_{2} \setminus I_1\},\quad W&\sim&N (0,K_{I_{2}}/K_{I_1\cap I_2})
\end{eqnarray*}
over the indicated region and is zero elsewhere.

Selecting valid perturbations $C$ is thus very easy: it basically amounts to the functional analogue of ``assigning a correlation to two variances''. 
At the same time, notice that any non-zero perturbation $C$ introduces arbitrary correlations that were never observed (i.e. are entirely extrinsic to the partial covariance $K_{\Omega}$). This observation crystallises some of the remarks made in the closing of Section \ref{graphical-section}, i.e. that the canonical completion is unique in not introducing any arbitrary correlations extrinsic to $K_{\Omega}$.


We will now re-interpret the last result through the lens of operator theory -- this perspective will allow a fruitful extension of our characterisation to general serrated domains. First, we note that if $K_{\Omega}$ is continuous on $\Omega$, then so are all elements of $\mathscr{C}(K_\Omega)$ on $I$ and $L$ on $[(I_{1} \setminus I_{2}) \cup (I_{2} \setminus I_{1})]^{2}$ (for the latter, see Remark \ref{L_is_C} in Section \ref{sec:proofs}). As a result, we can think of $L$ as the kernel of a covariance operator. Let 
$$\mathbf{L}_{1}: L^{2}(I_{1} \setminus I_2) \to L^{2}(I_{1} \setminus I_2)\quad \mbox{ and }\quad\mathbf{L}_{2}: L^{2}(I_{2} \setminus I_1) \to L^{2}(I_{2} \setminus I_1)$$ denote the integral operators induced by the covariance kernels 
$$L|_{(I_{1} \setminus I_2) \times (I_{1} \setminus I_2)} = K_{I_{1}} / K_{I_1\cap I_2}\quad \mbox{ and }\quad L|_{(I_{2} \setminus I_1) \times (I_{2} \setminus I_1)} = K_{I_{2}} / K_{I_1\cap I_2}.$$
Moreover, let 
$$\mathbf{L}_{12}: L^{2}(I_{2} \setminus I_1) \to L^{2}(I_{1} \setminus I_2)$$ 
denote the integral operator corresponding to the kernel 
$$L|_{(I_{2} \setminus I_1) \times (I_{1} \setminus I_2)}.$$ 
Finally, define 
$$\mathbf{L}: L^{2}(I_{1} \setminus I_2) \times L^{2}(I_{2} \setminus I_1) \to L^{2}(I_{1} \setminus I_2) \times L^{2}(I_{2} \setminus I_1)$$ 
to be a linear operator defined via its action:
\begin{equation*}
\mathbf{L}(f,g) = (\mathbf{L}_{1} f + \mathbf{L}_{12} g, \mathbf{L}_{12}^{\ast} f + \mathbf{L}_{2} g)
\end{equation*}
Clearly, $L$ is a completion if and only if $\mathbf{L}$ is positive semidefinite. Notice that $\mathbf{L}_{1}$ and $\mathbf{L}_{2}$ are trace-class and positive semidefinite, and as a result $L$ is trace-class if it is positive-semidefinite. Now, $\mathbf{L}$ is positive semidefinite if and only if there is Gaussian measure $\mu_{12}$ on the Hilbert space $L^{2}(I_{1} \setminus I_2) \times L^{2}(I_{2} \setminus I_1)$ with zero mean and covariance operator $\mathbf{L}$ which has two Gaussian measures $\mu_{1}$ and $\mu_{2}$ with zero mean and covariance operators $\mathbf{L}_{1}$ and $\mathbf{L}_{2}$ as marginals. According to Baker \cite[Theorem 2]{baker1973}, the possible values of $\mathbf{L}$ are precisely the ones given when setting 
\begin{equation*}
\mathbf{L}_{12} = \mathbf{L}^{1/2}_{1} \Psi \mathbf{L}^{1/2}_{2}
\end{equation*}
for $\Psi: L^{2}(I_{1} \setminus I_2) \to L^{2}(I_{2} \setminus I_1)$ a bounded linear map with operator norm  $\left\| \Psi \right\|_{\infty} \leq 1$. In summary, if $K_\Omega$ is continuous, Theorem \ref{2serrated-characteriation} can be re-interpreted at the level of operators. Namely, in block notation, the operator $\mathbf{K}$ has a kernel in $\mathscr{C}(K_\Omega)$ if and only if
\begin{equation}\label{block-equation}
\mathbf{K}f = \mathbf{K}_{\star}f +\underset{\mathbf{C}}{\underbrace{\left(\begin{array}{c|c|c} 0 & 0 & \quad(\mathbf{L}^{1/2}_{1} \Psi \mathbf{L}^{1/2}_{2})^{*} \Bstrut\quad \\
\hline
0 & 0 & 0 \\
\hline
\quad\mathbf{L}^{1/2}_{1} \Psi \mathbf{L}^{1/2}_{2} \Tstrut \quad& \quad 0 \quad & 0\end{array}\right)}} \left(\begin{array}{c}f|_{I_1\setminus I_2} \\f|_{I_1\cap I_2} \\f|_{I_2\setminus I_1}\end{array}\right)
\end{equation}
Here $\mathbf{K}_{\star}$ is the operator with the canonical completion $K_{\star}$ as its kernel, and as $\Psi$ ranges over the ball $\|\Psi\|_\infty\leq 1$, the expression above generates all possible operator completions. Choosing $\Psi=0$ obviously yields the canonical completion. Note that the operator $\mathbf{C}$ in Equation \eqref{block-equation} is precisely the operator corresponding to the (cross-covariance) integral kernel $C$ as described earlier.

We will use this operator perspective to obtain a characterisation in the general case, where $\Omega=\cup_{j=1}^{m} (I_j\times I_j)$ is an $m$-serrated domain. This will require some additional notation to avoid excessively cumbersome expressions. For $1 \leq p < m$ define the following sets:
$$J_{p}=I_{p} \cap I_{p+1},\quad D_{p}=I_{p+1} \setminus I_{p},\quad S_{p}=\left[ \cup_{j=1}^{p}I_{j} \right] \setminus I_{p+1},\quad R_{p}=D_{p} \times S_{p}, \quad R^{\prime}_{p}=S_{p} \times D_{p}.$$ 
See Figure \ref{fig:cov_op} for a visual interpretation.\\

\begin{figure}[h]	
	\centering
	\begin{center}
		\begin{minipage}[l]{0.49\textwidth}
			\begin{tikzpicture}[scale=0.75]
			\draw [shift={(0.5,0.5)}] (0,0) -- (0,8) -- (8,8) -- (8,0) -- cycle;
			
			\draw [shift={(0.5,0.5)}, fill=gray!50] (0,0) -- (0,2.7) -- (2.7,2.7) -- (2.7,0) -- cycle;
			\draw [shift={(1.3,1.3)}, fill=gray!50] (0,0) -- (0,3) -- (3,3) -- (3,0) -- cycle;
			\draw [shift={(1.7,1.7)}, fill=gray!50] (0,0) -- (0,3.3) -- (3.3,3.3) -- (3.3,0) -- cycle;
			
			\draw [shift={(5.5,5.5)}, fill=gray!50] (0,0) -- (0,3) -- (3,3) -- (3,0) -- cycle; 
			\draw [shift={(4.1,4.1)}, fill=gray!50] (0,0) -- (0,3.3) -- (3.3,3.3) -- (3.3,0) -- cycle;
			\draw [shift={(3.5,3.5)}, fill=gray!50] (0,0) -- (0,3) -- (3,3) -- (3,0) -- cycle;
			
			\draw [shift={(3.5,3.5)}, fill=red!40] (0,0) -- (0,1.5) -- (1.5,1.5) -- (1.5,0) -- cycle;
			\draw [fill=blue!30] (5,0.5) -- (5, 3.5) -- (6.5,3.5) -- (6.5,0.5) -- cycle;
			
			\draw [shift = {(0.5,0.5)}] [dotted] (6,3.5) -- (6,0);
			\draw [shift = {(0.5,0.5)}] [dotted] (4.5,1.5) -- (4.5,0);
			\draw [shift = {(0.5,0.5)}] [dotted] (3.5,6) -- (0,6);
			\draw [shift = {(0.5,0.5)}] [dotted] (1.5,4.5) -- (0,4.5);
			
			\draw [shift={(0.5,0.5)}] (0.5,0.4) node {$K_{I_{1}}$};
			\draw [shift={(0.7,0.7)}] (1.6,1.5) node [fill=gray!50] {$K_{I_{p}}$};
			\draw [shift={(0.5,0.5)}] (3.7,3.7) node [fill=red!40] {$K_{J_{p}}$};
			\draw [shift={(0.5,0.5)}] (5.3,5.4) node {$K_{I_{p+1}}$};
			\draw [shift={(0.5,0.5)}] (7.5,7.5) node  {$K_{I_{m}}$};
			\draw [shift={(0.5,0.5)}] (5.3,1.5) node [fill=blue!30] {$R_{p}$};
			\draw [shift={(0.5,0.5)}] (1.5,5.3) node {$R^{\prime}_{p}$};
			
			\draw [{Parenthesis}-|,shift = {(0.9,0.5)}] (6,2.95) -- (6,0.05) node [midway,fill=white] {$S_{p}$};
			\draw [{Parenthesis}-|,shift = {(0.5,0.4)}] (4.5,-0.3) -- (6,-0.3) node [midway, fill=white] {$D_{p}$};
			
			\draw (0.5,0.5) node {{\scriptsize $|$}};
			\draw (8.5,0.5) node {{\scriptsize $|$}};
			\end{tikzpicture}
		\end{minipage}
		\begin{minipage}[r]{0.49\textwidth}
			\begin{tikzpicture}[scale=0.75]
			\draw [shift={(0.5,0.5)}] (0,0) -- (0,8) -- (8,8) -- (8,0) -- cycle;			
			\draw [shift={(0.5,0.5)}, fill=gray!50] (0,0) -- (0,2.7) -- (2.7,2.7) -- (2.7,0) -- cycle;
			\draw [shift={(1.3,1.3)}, fill=gray!50] (0,0) -- (0,3) -- (3,3) -- (3,0) -- cycle;
			\draw [shift={(1.7,1.7)}, fill=gray!50] (0,0) -- (0,3.3) -- (3.3,3.3) -- (3.3,0) -- cycle;
			
			\draw [shift={(5.5,5.5)}, fill=gray!50] (0,0) -- (0,3) -- (3,3) -- (3,0) -- cycle; 
			\draw [shift={(4.1,4.1)}, fill=gray!50] (0,0) -- (0,3.3) -- (3.3,3.3) -- (3.3,0) -- cycle;		
			\draw [shift={(3.5,3.5)}, fill=gray!50] (0,0) -- (0,3) -- (3,3) -- (3,0) -- cycle; 

			\draw [shift={(3.5,3.5)}, fill=red!40] (0,0) -- (0,1.5) -- (1.5,1.5) -- (1.5,0) -- cycle;
			
			\draw [shift={(3.5,3.5)}, fill=green!50, fill opacity=0.4] (1.5,0) -- (3,0) -- (3,1.5) -- (1.5,1.5) -- cycle;
			\draw [shift={(3.5,3.5)}, fill=yellow!50, fill opacity=0.4] (0,0) -- (1.5,0) -- (1.5,-3) -- (0,-3) -- cycle;
			\draw [fill=blue!30] (5,0.5) -- (5, 3.5) -- (6.5,3.5) -- (6.5,0.5) -- cycle;
			
			\draw [shift = {(0.5,0.5)}] [dotted] (6,3.5) -- (6,0);
			\draw [shift = {(0.5,0.5)}] [dotted] (4.5,1.5) -- (4.5,0);
			\draw [shift = {(0.5,0.5)}] [dotted] (3.5,6) -- (0,6);
			\draw [shift = {(0.5,0.5)}] [dotted] (1.5,4.5) -- (0,4.5);
			
			\draw [shift={(0.5,0.5)}] (3.7,3.7) node [fill=red!40] {$\mathbf{J}_{p}$};
			\draw [shift={(0.5,0.5)}] (5.3,3.7) node  {$\mathbf{D}_{p}$};
			\draw [shift={(0.5,0.5)}] (3.7,1.5) node  {$\mathbf{S}_{p}$};
			
			\draw [shift={(0.5,0.5)}] (0.5,0.4) node [fill=gray!50] {$\mathbf{K}_{1}$};
			\draw [shift={(0.7,0.7)}] (1.5,1.5) node [fill=gray!50] {$\mathbf{K}_{p}$};
			\draw [shift={(0.5,0.5)}] (5.3,5.4) node [fill=gray!50] {$\mathbf{K}_{p+1}$};
			\draw [shift={(0.5,0.5)}] (7.4,7.5) node [fill=gray!50] {$\mathbf{K}_{m}$};
			\draw [shift={(0.5,0.5)}] (5.3,1.5) node  {$\mathbf{R}_{p}$};
			
			\draw [{Parenthesis}-|,shift = {(0.9,0.5)}] (6,2.95) -- (6,0.05) node [midway, fill=white] {$S_{p}$};
			\draw [{Parenthesis}-|,shift = {(0.5,0.4)}] (4.5,-0.3) -- (6,-0.3) node [midway, fill=white] {$D_{p}$};
			\end{tikzpicture}
		\end{minipage}
	
		\caption{Illustration of the sets $D_{p}$, $S_{p}$, $R_{p}$ and $R^{\prime}_{p}$ (left) and the corresponding operators (right).} 
		\label{fig:cov_op}
	\end{center}
	
\end{figure}
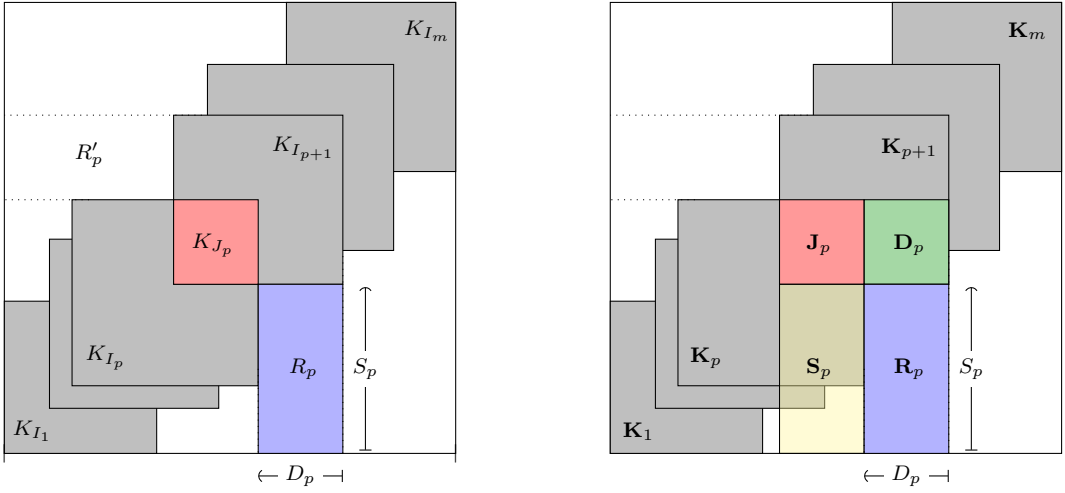

Let $K$ be a covariance on $I$ with associated operator $\mathbf{K}$. For every $1 \leq j \leq m$, let $\mathbf{K}_{j}: L^{2}(I_{j}) \to L^{2}(I_{j})$ be the Hilbert-Schmidt operator induced by the integral kernel $K_{I_{j}} = K|_{I_{j} \times I_{j}}$. And for $1 \leq p < m$, let $\mathbf{J}_{p} \in \mathcal{S}_{2}(J_{p}, J_{p})$ and $\mathbf{R}_{p}\in \mathcal{S}_{2}(D_{p}, S_{p})$ be Hilbert-Schmidt operators induced by the integral kernels $K_{J_{p}} = K|_{J_{p} \times J_{p}}$ and $K_{R_{p}} = K|_{S_{p} \times D_{p}}$ respectively.

We will now show that $\mathbf{K}$ can always be written in a sort of ``block notation'', i.e. in terms of $\{\mathbf{K}_{j}\}_{j}$, $\{\mathbf{J}_{p}\}_{p}$ and $\{\mathbf{R}_{p}\}_{p}$. This will allow us to generalise the type of expression Equation \eqref{block-equation} to the $m$-serrated case. 
\begin{restatable}{lemma}{operK} \label{lemma:operK}
	Given any $f \in L^{2}(I)$ and continuous kernel $K:I\times I\rightarrow \mathbb{R}$ with associated operator $\mathbf{K}$, the mapping $f\mapsto \mathbf{K}f$ can be represented blockwise as
	\begin{equation*}
	\mathbf{K}f(t) = \sum_{j: t \in I_{j}} \mathbf{K}_{j} f|_{I_{j}}(t)
	+ \sum_{p: t \in S_{p}} \mathbf{R}_{p} f|_{D_{p}}(t) 
	+ \sum_{p: t \in D_{p}} \mathbf{R}_{p}^{\ast}f|_{S_{p}}(t)
	- \sum_{p: t \in J_{p}} \mathbf{J}_{p} f|_{J_{p}}(t) \mathrm{~a.e.}
	\end{equation*}
\end{restatable}

Consequently, in order to characterize any integral operator corresponding to a completion of $K_{\Omega}$, it suffices to characterize the operators $\{\mathbf{R}_{p}\}_{p}$. These are the only ``missing pieces'', as the rest is known from $K_\Omega$ (see Figure \ref{fig:cov_op}).

To this end, for $1 \leq p < m$, we define $\mathbf{D}_{p} \in \mathcal{S}_{2}(J_{p},D_{p})$ and $\mathbf{S}_{p} \in \mathcal{S}_{2}(S_{p}, J_{p})$ to be the Hilbert-Schmidt operators induced by the integral kernels $K_{\Omega}|_{D_{p} \times J_{p}}$ and $K_{\star}|_{J_{p} \times S_{p}}$ respectively {(with $K_{\star}$ the canonical completion, as always)}. 

\smallskip
\noindent Now we have all the ingredients to characterise all completions from a serrated domain:

\begin{theorem}[Characterisation of Completions from a General Serrated Domain]\label{general-characterisation}
	Let $K_{\Omega}$ be a continuous partial covariance on a serrated domain $\Omega$ of $m$ intervals. Then $K: I \times I \to \mathbb{R}$ with $K|_{\Omega} = K_{\Omega}$ is a completion of $K_{\Omega}$ if and only if the integral operator $\mathbf{K}$ corresponding to $K$ is of the form
	\begin{equation}\label{general-formula}
	\mathbf{K}f(t) = \sum_{j: t \in I_{j}} \mathbf{K}_{j} f_{I_{j}}(t)
	+ \sum_{p: t \in S_{p}} \mathbf{R}_{p} f_{D_{p}}(t) 
	+ \sum_{p: t \in D_{p}} \mathbf{R}_{p}^{\ast}f_{S_{p}}(t)
	- \sum_{p: t \in J_{p}} \mathbf{J}_{p} f_{J_{p}}(t) \mathrm{~a.e.}
	\end{equation}	
	where for $1 \leq p < m$, 
	\begin{equation}\label{eq:para}
	\mathbf{R}_{p} = \left[ \mathbf{J}_{p}^{-1/2} \mathbf{S}^{\ast}_{p} \right]^{\ast}\left[\mathbf{J}_{p}^{-1/2} \mathbf{D}_{p}\right] + \mathbf{U}_{p}^{1/2} \Psi_{p} \mathbf{V}_{p}^{1/2}
	\end{equation}
	with
	$$
	\mathbf{U}_{p} = \mathbf{K}_{{S_{p}}} -  \left[\mathbf{J}_{p}^{-1/2}\mathbf{S}_{p}^{\ast}\right]^{\ast}\left[\mathbf{J}_{p}^{-1/2}\mathbf{S}_{p}^{\ast}\right],\qquad
	\mathbf{V}_{p} = \mathbf{K}_{D_{p}} - \left[\mathbf{J}_{p}^{-1/2}\mathbf{D}_{p}^{\ast}\right]^{\ast}\left[\mathbf{J}_{p}^{-1/2}\mathbf{D}_{p}^{\ast}\right]
	$$
	and $\Psi_{p}: L^{2}(D_{p}) \to L^{2}(S_{p})$ is a bounded linear map with $\left\| \Psi_{p} \right\| \leq 1$.
\end{theorem}
The only degrees of freedom in  Equation (\ref{eq:para}) stem from the $m$ contractions $\{\Psi\}_{p=1}^{m}$. All other operators involved in Equation \eqref{eq:para} (and in the right hand side of Equation \eqref{general-formula}) are uniquely defined via $K_\Omega$ (or equivalently via $K^*$). Allowing these to range over the unit balls
$$\|\Psi_{p}\|_{\infty}\leq 1,\quad  \Psi_{p}: L^{2}(D_{p}) \to L^{2}(S_{p}),\qquad p=1,\ldots,m$$ 
we trace out the set $\mathscr{C}(K_{\Omega})$ and get an idea of what the different possibilities of the actual covariance may look like. Substituting $\Psi_{p} = 0$ for all $1 \leq p < m$ returns the integral operator corresponding to the canonical completion $K_{\star}$ of $K_{\Omega}$. Since all other elements of Equation (\ref{eq:para}) are fully determined by $K_\Omega$, it is clear that the choice of $\{\Psi_p\}$ is arbitrary, and any non-zero choice will introduce information extrinsic to observed correlation patterns -- extending the intuition build in the 2-serrated case relating to the canonicity of $K_{\star}$.

Theorem \ref{general-characterisation} also complements Theorem \ref{serrated-theorem}, in that it yields expresses the canonical completion as the solution of a system of equations rather than the output of an algorithm. This manner of specification is slightly weaker, in that it assumes continuity of $K_\Omega$, whereas Theorem \ref{serrated-theorem} makes no such assumption. On the other hand, it provides a characterisation of the canonical solution in a form that lends itself for the problem of \emph{estimation}, treated in the next Section.

\section{Estimation of the Canonical Completion} \label{sec:estimation}

In this section, we consider the problem of estimation the canonical completion $K_{\star}$ when we only have access to an estimator of the partial covariance $K_{\Omega}$. From the purely analytical sense, we are studying the \emph{stability} of the canonical completion $K_{\star}$ of $K_\Omega$ with respect to perturbations of the partial covariance $K_\Omega$. From the statistical point of view, this relates to the problem that arises in the context of covariance recovery from functional fragments: when we observe fragments $X_{j} |_{I_{j}}$ of i.i.d. realisations of a second-order process $\{X_t: t\in I\}$ for a collection of intervals $\{I_j\}$ covering $I$. Because of the fragmented nature of the observations, we only have covariance information on the serrated domain $\Omega=\cup_j I_j\times I_j$, or equivalently we only can identify the partial covariance $K_{\Omega}$ corresponding to the restriction of $\mathrm{Cov}\{X_s,X_t\}=K(s,t)$ to the serrated domain $\Omega$. 


In this context, we posit that it makes good sense to choose the canonical completion $K_{\star}$ as the target of estimation. This because the canonical completion is always an \emph{identifiable} and \emph{interpretable} target of estimation:

\begin{enumerate}
\item When a unique completion exists, it must be the canonical one. So choosing the canonical completion allows us to \emph{adapt} to uniqueness.

\item When multiple completions exist, the canonical completion remains identifiable, and is least presumptuous -- it relies solely on the available data.
\end{enumerate}

Targeting the canonical completion without any attempt to enforce uniqueness is qualitatively very different from  previous approaches to covariance recovery from fragmented paths. Those approaches imposed uniqueness by way of assumptions (indeed assumptions implying very rigid consequences, as demonstrated in Section \ref{sec:unique-completion}). Once uniqueness is a priori guaranteed, any estimator $\hat{K}$ whose restriction $\hat{K}|_\Omega$ is consistent for $K_\Omega$ will be valid -- so, for instance, one can safely extrapolate an estimator $\hat{K}_\Omega$ of $K_\Omega$ by means of a basis expansion or matrix completion. But when the strong assumptions guaranteeing uniqueness fail to hold, such ``extrapolation'' estimators will yield arbitrary completions, indeed completions that will not even convergence asymptotically, but rather oscillate in some open neighbourhood of $\mathscr{C}(K_\Omega)$. On the other hand, the adaptivity (to uniqueness) and stability/interpretability (under non-uniqueness) of the canonical completion comes at a price: to be able to target the canonical completion $K_{\star}$ we need an estimator that is not merely consistent for $K_\Omega$ on $\Omega$, but one that  (asymptotically) also satisfies the system of equations in Theorem \ref{general-characterisation} (with $\Psi_p$ identically zero). This has consequences on the rates of convergence, which can no longer be as fast as the rates of estimating the partial covariance $K_\Omega$.


\begin{figure}[h]\label{fig:part_cov1}
	\centering
	\includegraphics[scale=0.45, trim= 4.3cm 0cm 3.5cm 0cm, clip]{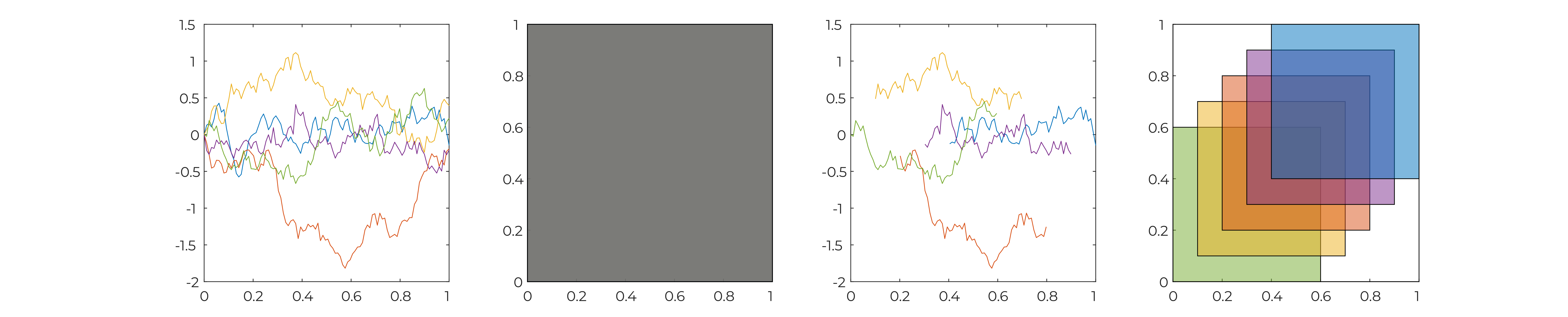}
	\caption{Illustration of the problem of covariance recovery from fragments: (from the left) fully observed sample paths of a process on the unit interval $I = [0,1]$, the region $I^{2}$ on which the covariance can be estimated in the fully observed case, partially observed versions of the same sample paths and the region on which the covariance can be estimated from the sample paths of the corresponding colour.}
	\label{fig:fragments}
\end{figure}

\subsection{Definition of the Estimator} 

Courtesy of Theorem \ref{general-characterisation}, the specification of $K_{\star}$ reduces to that the solution the following system of linear equations:
\begin{equation}\label{eq:system_linear}
\begin{split}
\mathbf{J}^{1/2}_{p}\mathbf{X}_{p} &= \mathbf{S}^{\ast}_{p}\\
\mathbf{J}^{1/2}_{p}\mathbf{Y}_{p} &= \mathbf{D}_{p}\\ 
\end{split}
\qquad \mbox{ for } 1 \leq p < m.
\end{equation}
Notice that the operator $\mathbf{J}^{1/2}_{p}$ is compact because $\mathbf{J}_{p}$ is. It follows that the canonical completion $K_{\star}$ does not depend continuously on the partial covariance $K_{\Omega}$. In practice we only have access to an estimator $\hat{K}_\Omega$ of $K_{\Omega}$. Therefore, the \textit{operator} of the inverse problem, i.e. $\mathbf{J}^{1/2}_{p}$, as well as the \textit{data} of the inverse problem, in the form of $\mathbf{D}_{p}$ and $\mathbf{S}_{p}$, are inexactly specified. 

We will thus define our estimator as the solution of a regularized empirical version of the system. Let $\hat K_\Omega$ be an estimator of $K_\Omega$. Let $\hat{\mathbf{K}}_{p}$,  $\hat{\mathbf{D}}_{p}$ and $\hat{\mathbf{J}}_{p}$ be the Hilbert-Schmidt operators with the kernels {$\hat{K}_\Omega|_{I_{p} \times I_{p}}$, $\hat{K}_\Omega|_{J_{p} \times D_{p}}$ and $\hat{K}_\Omega|_{J_{p} \times J_{p}}$}, respectively.\\

Finally, motivated by the definition
\begin{equation*}
	\mathbf{R}_{p} = \left[ \mathbf{J}_{p}^{-1/2} \mathbf{S}^{\ast}_{p} \right]^{\ast}\left[\mathbf{J}_{p}^{-1/2} \mathbf{D}_{p}\right]
\end{equation*} 
and using a truncated inverse of $\hat{\mathbf{J}}_{p}$, we define the regularised empirical version of $\mathbf{R}_p$ as
 \begin{equation} \label{eq:regularized_R}
\begin{split}
\hat{\mathbf{R}}_{p} 
&= \sum_{k = 1}^{N_{p}} \frac{1}{\hat{\lambda}_{p,k}} \cdot \hat{\mathbf{S}}_{p} \hat{e}_{p,k} \otimes \hat{\mathbf{D}}_{p}^{\ast} \hat{e}_{p,k}
\end{split} 
\end{equation} where $\hat{\lambda}_{p,k}$ and $\hat{e}_{p,k}$ denote the $k$th eigenvalue and eigenfunction of $\hat{\mathbf{J}}_{p}$, $N_{p}$ is the truncation or regularization parameter, and $\hat{\mathbf{S}}_{p}$ has kernel $\hat{K}_{\star}|_{S_{p} \times J_{p}}$. Notice that the definition is recursive:
\begin{itemize}
\item $\hat{\mathbf{R}}_{p}$ depends on $\hat{\mathbf{S}}_{p}$ and thus on $\hat{\mathbf{R}}_{i}$ for $i < p$. 

\item $\hat{\mathbf{S}}_{1}$ is fully determined by $\hat{K}_\Omega$, and $\hat{\mathbf{S}}_{p+1}$ is fully determined by $\hat{K}_\Omega$ and $\hat{\mathbf{R}}_p$.
\end{itemize}
In particular, though the kernel of $\hat{\mathbf{S}}_{p}$ can \emph{a posteriori} be seen to equal $\hat{K}_{\star}|_{S_{p} \times J_{p}}$, this does not mean that it depends \emph{a priori} on $\hat{K}_{\star}$ (i.e. there is no vicious circle in the definition).

\medskip
We can now define our estimator $\hat{K}_{\star}: I \times I \to \mathbb{R}$ of $K_{\star}$ to be the the integral kernel of the Hilbert-Schmidt operator $ \hat{\mathbf{K}}_{\star}: L^2 (I)\rightarrow L^2(I)$ defined via the action
\begin{equation}\label{estimator}
\hat{\mathbf{K}}_{\star}f(t) = \sum_{j: t \in I_{j}} \hat{\mathbf{K}}_j f|_{I_{j}}(t)
+ \sum_{p: t \in S_{p}} \hat{\mathbf{R}}_{p} f|_{D_{p}}(t) 
+ \sum_{p: t \in D_{p}} \hat{\mathbf{R}}_{p}^{\ast}f|_{S_{p}}(t)
- \sum_{p: t \in J_{p}} \hat{\mathbf{J}}_{p} f|_{J_{p}}(t). 
\end{equation}
Equivalently,  we can define $\hat{K}_{\star} \in L^{2}(I \times I)$ recursively as follows: $\hat{K}_{\star}|_{\Omega} = \hat{K}_{\Omega}$ and $\hat{K}_{\star}|_{R_{p}}$ is the kernel associated with the Hilbert-Schmidt operator $\hat{\mathbf{R}}_{p}$ defined recursively via (\ref{eq:regularized_R}).\\

\subsection{Rate of Convergence}

We will now characterize the rate of convergence of $\hat{K}_{\star}$ to $K_{\star}$ in terms of the spectral properties of the partial covariance $K_{\Omega}$, and the rate of convergence of the partial covariance estimator $\hat{K}_\Omega$ we have used as a basis, to the partial covariance $K_\Omega$ itself. Concerning the spectral properties of $K_{\Omega}$, let $\{ \lambda_{p,k} \}_{k=1}^{\infty}$ be the eigenvalues of $\mathbf{J}_{p}$ and define $A_{p,k}$ as:
\begin{equation*}
	A_{p,k} = \left\| \sum_{j = k+1}^{\infty} \frac{\mathbf{S}_{p}e_{p,k} \otimes \mathbf{D}_{p}^{\ast}e_{p,k}}{\lambda_{p,k}} \right\|_{2}^{2}
\end{equation*}
where $\{e_{p,k}\}_{k=1}^{\infty}$ are the eigenfunctions of $\mathbf{J}_{p}$.
Notice that $A_{p,0}$ is simply the Hilbert-Schmidt norm of $\mathbf{R}_{p}$ and $A_{p,k}$ simply represents the error in approximating $\mathbf{R}_{p}$ by using a rank-$k$ truncated inverse of $\mathbf{J}_{p}^{1/2}$ instead of $\mathbf{J}_{p}^{1/2}$ in the expression
\begin{equation*}
	\mathbf{R}_{p} = \left[ \mathbf{J}_{p}^{-1/2} \mathbf{S}^{\ast}_{p} \right]^{\ast}\left[\mathbf{J}_{p}^{-1/2} \mathbf{D}_{p}\right]
\end{equation*}  
so $A_{p,k}$ must necessarily converge to $0$ as $k \to \infty$. The following result gives the rate of convergence for the case when the eigenvalues and approximation errors decay at a polynomial rate. 

\begin{theorem}[Consistency and Rate of Convergence]\label{thm:rate_convergence}
	Let $K_{\Omega}$ be a partial covariance on a serrated domain $\Omega$ of $m$ intervals and $\hat{K}_{\Omega}$ be an estimator thereof. Let $\hat{K}$ be defined as in Equation \eqref{estimator}. Assume that for every $1 \leq p < m$, we have $\lambda_{p,k} \sim k^{-\alpha}$ and  $A_{p,k}  \sim k^{-\beta}$. If the error in the estimation of $K_{\Omega}$ satisfies
	\begin{equation*}
		\| \hat{K}_{\Omega} - K_{\Omega} \|_{L^{2}(\Omega)} = O_{\mathbb{P}}(1/n^{\zeta})
	\end{equation*}
	where $n$ is the number of fragments, then the error in the estimation of the canonical completion satisfies for every $\varepsilon > 0$,
	\begin{equation*}
		\| \hat{K}_{\star} - K_{\star} \|_{L^{2}(I\times I)} = O_{\mathbb{P}}(1/n^{\zeta\gamma_{m-1}-\varepsilon})
	\end{equation*}
	provided the truncation parameters $\mathbf{N} = (N_{p})_{p=1}^{m-1}$ scale according to the rule
	\begin{equation*}
		N_{p} \sim n^{\gamma_{p}/ \beta}
	\end{equation*}
	where $\gamma_{m-1} =	\tfrac{\beta}{\beta + 2 \alpha + 3/2}\left[ \tfrac{\beta}{\beta+\alpha+1/2}\right]^{m-2}.$
\end{theorem}

\begin{remark}[Plug-in Interpretation] The theorem can be seen as a plug-in rate of convergence theorem. We plug-in the ``baseline'' rate of convergence of $\hat{K}_{\Omega} \to K_{\Omega}$, and get a rate for $\hat{K}_{\star} \to K_{\star}$. Note that the tuning of the truncation parameters also depends on the baseline rate of convergence. Baseline rates are readily available for sparse, dense, and complete observation regimes.
\end{remark}

Notice that the rate of convergence $\gamma_{m-1}$ strictly decreases as a function of the number of intervals $m$, but can get arbitrarily close to $1$ for a large enough rate of decay of approximation errors $\beta$. Moreover, an increase in the rate of decay of eigenvalues $\alpha$ is accompanied by a decrease in the rate of convergence. If $K_{\Omega}$ is $r$-times differentiable then the same applies to the kernels $K_{\Omega}|_{J_{p} \times J_{p}}$ of $\mathbf{J}_{p}$ implying $\lambda_{p,k}$ is $o(1/n^{r+1})$ for every $1 \leq p < m$ and thus $\alpha = r+1$ . Thus, all other things being equal, an increase in the smoothness of $K_{\Omega}$ also tends to decrease the rate of convergence ---which is not surprising from an inverse problems perspective.

\section{Beyond Serrated Domains}\label{sec:nearly-serrated} 
Our theory has thus far concentrated on domains that are \emph{serrated} in the sense of Definition \ref{def:serrated}. We now turn our attention to a much larger class of domains, namely domains that can be approximated to arbitrary level of precision by serrated domains.  Recall that for subsets $X$ and $Y$ of a metric space $(M, d)$, the Hausdorff distance between is defined as 
\begin{equation*}
d_{H}(X, Y) = \big[ \sup\nolimits_{x \in X} \inf\nolimits_{y \in Y} d(x,y) \big] \vee \big[ \sup\nolimits_{y \in Y} \inf\nolimits_{x \in X} d(x,y) \big].
\end{equation*}
We define the class of \textit{nearly serrated domains} as the Hausdorff ``closure'' of the set of serrated subdomains of $I\times I$:
\begin{definition}[Nearly Serrated Domain]\label{def:nearly_serr_domain}
	We say that $\widetilde\Omega \subset I \times I$ is a nearly serrated domain if for every $\epsilon > 0$, there exists serrated domains $
	\Omega_{\epsilon}$ and $\Omega^{\epsilon}$ such that $\Omega_{\epsilon} \subset \Omega \subset \Omega^{\epsilon}$ and $d_{H}(\Omega, \Omega_{\epsilon}), d_{H}(\Omega, \Omega^{\epsilon}) < \epsilon$, where $d_{H}$ is the Hausdorff metric induced by the Euclidean metric on $I \times I \subset \mathbb{R}^{2}$.
\end{definition}

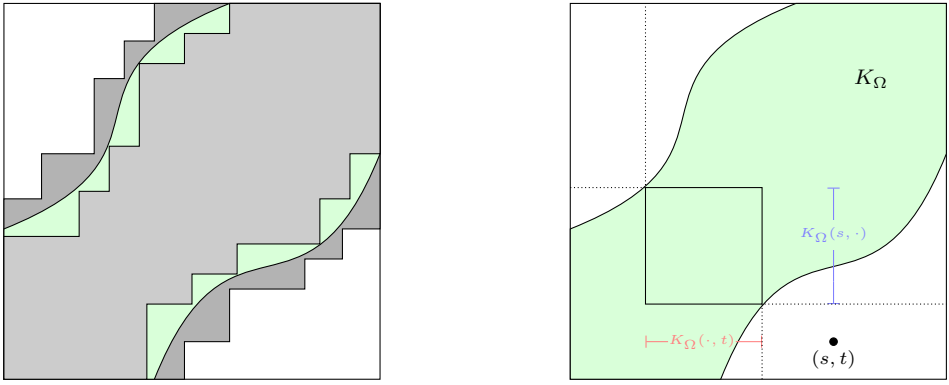
\begin{figure}
	\begin{minipage}[l]{0.45\textwidth}
		\begin{tikzpicture}
			\draw [shift={(0.5,0.5)}] (0,0) -- (0,5) -- (5,5) -- (5,0) -- cycle;			
			
			\draw [shift={(0.5,0.5)}, fill=gray!60] (0,0) -- (2.4,0) -- (2.4,0.5) -- (3,0.5) -- (3, 1.2) -- (4, 1.2) -- (4, 1.6) -- (4.5, 1.6) -- (4.5, 2) -- (5, 2) -- (5,5) -- (2,5) -- (2,4.5) -- (1.6,4.5) -- (1.6,4) -- (1.2,4) -- (1.2,3) -- (0.5,3) -- (0.5,2.4) -- (0,2.4) -- cycle;
			
			\draw [shift={(0.5,0.5)}, fill=green!15] (0,0) -- (2,0) .. controls (3,2.5) and (4,0.5) .. (5,3) -- (5,5) -- (3,5) .. controls (0.5,4) and (2.5,3) .. (0,2) -- cycle;
			
			\draw [shift={(0.5,0.5)}, fill=gray!40] (0,0) -- (1.9,0) -- (1.9,1) -- (2.5,1) -- (2.5, 1.4) -- (3.1,1.4) -- (3.1,1.8) -- (4.2,1.8) -- (4.2, 2.4) -- (4.6,2.4) -- (4.6, 3) -- (5, 3) -- (5,5) -- (3,5) -- (3,4.6) -- (2.4,4.6) -- (2.4,4.2) -- (1.8,4.2) -- (1.8,3.1) -- (1.4,3.1) -- (1.4,2.5) -- (1,2.5) -- (1,1.9) -- (0,1.9) -- cycle;
			
		\end{tikzpicture}
	\end{minipage}
	\begin{minipage}[r]{0.45\textwidth}
		\begin{tikzpicture}
		\draw [shift={(0.5,0.5)}] (0,0) -- (0,5) -- (5,5) -- (5,0) -- cycle;			
		\draw [shift={(0.5,0.5)}, fill=green!15] (0,0) -- (2,0) .. controls (3,2.5) and (4,0.5) .. (5,3) -- (5,5) -- (3,5) .. controls (0.5,4) and (2.5,3) .. (0,2) -- cycle;
		\draw [shift={(1.5,1.5)}] (0,0) -- (0,1.55) -- (1.55,1.55) -- (1.55,0) -- cycle;
		
		\draw [|-,shift = {(0.5,0.5)}, blue!50] (3.5,1) -- (3.5,1.3);
		\draw [|-,shift = {(0.5,0.5)}, blue!50] (3.5,2.55) -- (3.5,1.3) node [midway,fill=green!15] {\tiny{$K_{\Omega}(s,\cdot)$}};
		
		\draw [|-,shift = {(0.5,0.5)}, red!50] (1,0.5) -- (1.3,0.5);
		\draw [-|,shift = {(0.5,0.5)}, red!50] (2.2,0.5) -- (2.55,0.5);
		\draw [shift={(0.5,0.5)}, red!50] (1.75,0.5) node {\tiny{$K_{\Omega}(\cdot, t)$}};
		
		\draw [shift={(0.5,0.5)}, densely dotted] (2.55,0) -- (2.55,1) -- (5,1);
		\draw [shift={(0.5,0.5)}, densely dotted] (0,2.55) -- (1,2.55) -- (1,5);
		
		\draw [shift={(0.5,0.5)}] (4,4) node {$K_{\Omega}$};
		
		\draw [shift={(0.5,0.5)},fill] (3.5,0.5) circle [radius=0.05];
		\draw [shift={(0.5,0.5)}] (3.5,0.5) node [below] {{\scriptsize $(s,t)$}};
		\end{tikzpicture}
	\end{minipage}
\begin{center}

\caption{Left: Illustration of a nearly serrated domain (in green) as enveloped by two serrated domains (in light and dark grey). Right: a point $(s,t)$ escaping the scope of Equation (\ref{2-serrated-completion}).}
\label{fig:nearly-serrated}
\end{center}
\end{figure}

Notice that every serrated domain is nearly serrated according to the above definition. Of particular importance is the case when $\Omega$ is a strip of width $w > 0$ around the diagonal, that is,
$\Omega = \{(s,t) \in I\times I: |s-t| \leq w/\sqrt{2}\}$. This occurs asymptotically in the problem of covariance recovery from fragments, when each sample path is observable which over ``uniformly distributed'' intervals of constant length $w/\sqrt{2}$.\\
		
It should be clear from Figure \ref{fig:nearly-serrated} that we can not exploit Equation \ref{2-serrated-completion} to recover the canonical completion of partial covariance on nearly serrated but not serrated domain, as we did for serrated domains previously. This is because for such domains there are points $(s,t)$ for which the cross-covariances $K_{\Omega}(s,\cdot)$ and $K_{\Omega}(\cdot, t)$ are not available, nor can they be iteratively calculated from the part of the covariance that is known. Thus one can not evaluate their inner product $\big\langle K_{\Omega}(s,\cdot), K_{\Omega}(\cdot, t) \big\rangle_{\mathscr{H}(K_{I_1\cap I_2})}$ as required in Equation \ref{2-serrated-completion}. It was precisely because the domain was serrated that that we were able to recover the value of the canonical completion over successively larger regions as we did in Algorithm \eqref{algorithm}. 

Additionally, since we cannot apply Algorithm 1, it is unclear what it means for a covariance $K$ to be the canonical completion of a partial covariance $K_{\Omega}$ on a nearly serrated domian. Here we lean on our graphical models interpretation to define the canonical completion in this case. We shall say that a covariance $K$ is a canonical completion of $K_{\Omega}$ if it is a completion i.e. $K|_{\Omega} = K_{\Omega}$ and $K \in \mathscr{G}_{\Omega}$ as defined in Section \ref{graphical-section}.


Our focus will, therefore, be to obtain results pertaining to uniqueness/canonicity of completions from nearly serrated domains $\widetilde{\Omega}$ by means of serrated subdomains $\Omega\subset \widetilde\Omega$ or superdomains $\widetilde\Omega \subset \Omega$. Our first result gives a sufficient condition for unique completion from a nearly serrated domain:

\begin{theorem}[Checking Uniqueness via Serrated Subdomains]\label{thm:unique-nearly-serrated}
Let $K_{\widetilde\Omega}$ be a partial covariance on a nearly serrated domain $\widetilde\Omega$ and let  $\Omega \subset  \widetilde\Omega$ be a serrated domain. If the restriction $K_{\widetilde\Omega}\big|_{\Omega}$ admits a unique completion, so does $K_{\widetilde\Omega}$.
\end{theorem}

Theorem \ref{thm:unique-nearly-serrated}, via our necessary and sufficient conditions for uniqueness on serrated domains (Theorem {uniqueness}),  yields sufficient conditions for unique completion from a banded domain that are strictly weaker than any previously known set of sufficient conditions. In the serrated case, a unique completion is necessarily canonical. A natural question is whether this remains the case for nearly serrated domains. The answer is in the affirmative:

\begin{theorem}[Unique Completions are Canonical]\label{thm:unique-canonical-nearly}
If the partial covariance $K_{\widetilde\Omega}$ on a nearly serrated domain $\widetilde\Omega$ has a unique completion on $I\times I$, this completion is canonical.
\end{theorem}	

Theorem \ref{thm:unique-canonical-nearly} shows that targeting canonical completions remains a sensible strategy in the context of nearly serrated domains -- they remain interpretable and yield the ``correct answer'' in the presence of uniqueness. That is, of course, if we know how to construct them. Our last result 

\begin{theorem}[Construction of Canonical Completions]\label{thm:construction-nearly}
A covariance $K_{\star}$ on $I$ can be recovered  as the canonical completion of its restriction $K_{\star}|_\Omega$ on a serrated domain $\Omega$ if and only if it is the canonical completion of a partial covariance on some nearly serrated domain $\widetilde\Omega \subset \Omega$.
\end{theorem}

In particular, if a unique completion of $K|_{\widetilde\Omega}$ exists then it equals the canonical completion of $K|_{\Omega}$ for a (in fact any) serrated $\Omega\supset \widetilde{\Omega}$. Alternatively, if the process $X$ with covariance $K$ forms a second-order graphical model with respect to the nearly serrated $\widetilde\Omega$, then $K$ can be obtained by means of Algorithm \ref{algorithm} applied to $K|_{\Omega'}$, for any $\Omega'\supset \widetilde{\Omega}$. This is possible because if $\widetilde\Omega \subset \Omega$, then every separator of $\Omega$ also separates $\widetilde\Omega$. As a result, if the ``separation equation'' is satisfied by $(s,t) \in (\widetilde\Omega)^{c}$ for separators of $\widetilde\Omega$, then it is also satisfied for separators of $\Omega$. Thus $\widetilde\Omega \subset \Omega$ implies $\mathscr{G}_{\widetilde\Omega} \subset \mathscr{G}_{\Omega}$.

\section{Covariance Estimation from Sample Path Fragments}\label{sec:fragments}

We now elucidate how one can make combined use of our results from Section \ref{sec:estimation} and Section \ref{sec:nearly-serrated}, in order to address the problem of covariance estimation from sample path fragments in a general context. Let $X$ be a second-order process on the unit interval $I$ with the covariance $K$. Suppose that for $n$ intervals $I_i\subset I$ we observe $n$ sample path fragments $X_i=X_i|_{I_i}$, where $X_i {\sim}X$ independently. Now define the domain $$\Omega_{\infty} = \underset{k\to\infty}{\lim\sup}  ~(I_k\times I_k),$$
as the set of pairs $(x,y)\in I\times I$ such that $I_i\ni \{x,y\}$ infinitely often. The sequence $\{X_i\}_{i=1}^{n}$ enables us to consistently estimate the restriction $K_{\Omega}:=\left.K\right|_{\Omega}$ of $K$ on any $\Omega\subset \Omega_\infty$. Call such a consistent estimator $\hat{K}_{\Omega}$

However, we wish to estimate the complete $K$, not merely its restrictions to $\Omega\subset \Omega_\infty$. This requires $K$ to be identifiable from $\Omega_\infty$. One means to securing identifiability is to impose unique extendability, i.e. assume that $\mathscr{C}(K_\Omega)=\{K\}$ for some $\Omega$ whose elements are covered infinitely often by the sequence of rectangles $I_k\times I_k$ (i.e. we have the inclusion\footnote{Note that in general the inclusion $\Omega_\infty\supset\Omega$ is to be taken as strict. The ``critical'' equality case would be rather exceptional, because $\Omega_\infty$ is defined by the  censoring mechanism, whereas the $\Omega$ is a population quantity. Stipulating that the structure of $X$ varies with or is tailored to the censoring mechanism would be contrived.} $\Omega\subset\Omega_\infty$). But uniqueness was seen to be overly restrictive (see Theorem \ref{uniqueness}). We therefore wish to avoid this route to identifiability. Instead, following the development in Section \ref{graphical-section} we will secure identifiability by assuming that the underlying process $X$ is globally Markov with respect to some domain ${\Omega}$ (i.e. $K\in \mathscr{G}_{\Omega}$) whose elements are covered infinitely often by the sequence the rectangles $I_k\times I_k$ (i.e. $\Omega\subset\Omega_\infty$). This is a substantially weaker assumption (due to Theorem \ref{thm:unique-canonical-nearly}), and arguably much more intuitive. 

Proceeding thus, let ${\Omega}\subset\Omega_\infty$ be some nearly serrated domain with respect to which $X$ is global Markov. By Theorem \ref{thm:canonical-graphical2}, the true covariance $K$ is the canonical completion from $\Omega$, and by Theorem \ref{thm:construction-nearly}, it is also the canonical completion from any serrated domain containing ${\Omega}$. Thus we can identify $K$ directly from $K_\Omega$ as the canonical completion of $K_{\Omega_m}$ for some $m$-serrated $\Omega_m$ (with some $m<\infty$) satisfying $\Omega\subseteq\Omega_m\subseteq {\Omega}_\infty$. The inclusions will always be possible for some $m<\infty$ provided the boundaries $\partial \Omega$ and $\partial \Omega_\infty$ are everywhere distinct (i.e. $\|u-v\|>0$ for all $u\in \partial\Omega$ and $v\in \partial\Omega_\infty$).
  

\smallskip
\noindent Now we distinguish two cases:
\begin{itemize}

\item[(i)] \textbf{$\Omega_\infty$ is serrated.} If intervals $\{I_j\}_{j=1}^n$ are sampled from a finite cover of $I$, then $\Omega_\infty$ will be a serrated domain. This represents a fixed domain setting, in that for all sufficiently large $n$ the observation domain becomes almost surely fixed. 
\item[(ii)] \textbf{$\Omega_\infty$ is nearly serrated.} If the intervals $\{I_j\}_{j=1}^n$ are sampled from an infinite cover of $I$, then $\Omega_\infty$ will be a nearly serrated domain. This represents a variable domain setting, as our observation domain will continue evolving as $n$ grows.

\end{itemize}


\begin{figure}[h]
	\begin{tikzpicture}
	\draw [shift={(0.5,0.5)}] (0,0) -- (0,5) -- (5,5) -- (5,0) -- cycle;			
	

	\draw [shift={(0.5,0.5)}, fill=blue!50] (0,0) -- (2.7,0) .. controls (5,0.5) and (4.5,1) .. (5,2.5) -- (5,5) -- (2.5,5) .. controls (1,4.5) and (0.5,5) .. (0,2.7) -- cycle;
	
	\draw [shift={(0.5,0.5)}, fill=gray!40] (0,0) -- (2.7,0) -- (2.7,0.4) -- (3.5,0.4) -- (3.5, 1.2) -- (4.6, 1.2) -- (4.6, 2) -- (4.8,2) -- (4.8,2.5) -- (5, 2.5) -- (5,5) -- (2.5,5) -- (2.5, 4.8) -- (2,4.8) -- (2, 4.6) -- (1.2, 4.6) -- (1.2, 3.5) -- (0.4, 3.5) -- (0.4, 2.7) -- (0, 2.7) -- cycle;
	
	\draw [shift={(0.5,0.5)}, fill=red!50] (0,0) -- (2.4,0) .. controls (3,2.5) and (4,0.5) .. (5,3) -- (5,5) -- (3,5) .. controls (0.5,4) and (2.5,3) .. (0,2.4) -- cycle;
	
	
	\draw [shift={(0.5,0.5)}] (2,2) node {${\Omega}$};
	\draw [shift={(0.5,0.5)}] (3.1,0.8) node {$\Omega_{m}$};
	\draw [shift={(0.5,0.5)}] (4,0.8) node {$\Omega_\infty$};
	\end{tikzpicture}
	\begin{center}
		\caption{The regions ${\Omega}$ (red), $\Omega_{m}$ (gray), and $\Omega_\infty$ (blue).}
		\label{fig:cov_est}
	\end{center}
\end{figure}
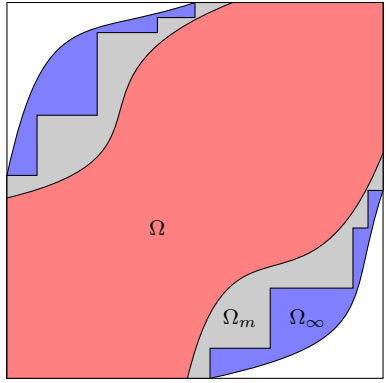
In case (i), we are squarely within the context of Section \ref{sec:estimation} and can use $\hat K_\Omega$ to directly define the estimator $K_{\star}$ used in Equation \ref{estimator}. 

In case (ii), we choose and fix an $m$-serrated approximation of the observable region $\Omega_{m} \subset \Omega$. We then construct the estimator $\hat K_{\star}$ in Equation \ref{estimator} based on $\hat K_{\Omega_m}$. So long as $\Omega_{m}$ contains ${\Omega}$ and $m$ is held fixed, the estimator thus constructed will converge to $K$ as per Theorem \ref{thm:rate_convergence}. A good choice of $\Omega_{m}$ involves a trade-off between covering a large subregion of $\Omega_\infty$ (to use as much of the observable domain as possible) and keeping $m$ small (to limit the number of inverse problems solved).  This approach is illustrated through a data analysis in Section \ref{sec:illustration}, and its performance is investigated via extensive simulations in Section \ref{sec:numerical_simn} (with special focus on the effect of the choice of $m$). 

\begin{remark}[On $m$ vs $n$ -- Practical Considerations]
From a practical point of view, the choice of a serrated approximation $\Omega_m$ to $\Omega_\infty$ does not entail any significant loss. In practice, $\Omega_\infty$ is in fact unknown, and the de facto domain of observation is the $n$-serrated domain $\cup_{j=1}^{n}(I_j\times I_j)$. Nevertheless, statistical considerations suggest that we should not use the full observation domain, namely:

\begin{enumerate}

\item The domain $\cup_{j=1}^{n}(I_j\times I_j)$ generally ``overfits'' $\Omega = {\lim\sup}_j  ~(I_j\times I_j)$. Regions of $I\times I$ that are more densely populated by observations are better proxies for $\Omega_\infty$ (meaning regions comprised of pairs $(x,y)\in I\times I$ such that $\#\{k\leq n: \{x,y\}\subset I_k\}$ is large). This suggests choosing $\Omega_m$ with $m$ distinctly smaller than $n$.

\item When the fragments are observed discretely and smoothing is used to construct $\hat{K}_{\Omega_m}$, we still use data in the the region $\cup_{j=1}^{n}(I_j\times I_j) \setminus \Omega_{m}$ as part of the local averaging, even though we desist from estimating outside $\Omega_m$. Hence we do not necessarily discard information, but rather focus on a smaller region on which we can estimate more efficiently: because this region is more densely populated by observations and furthermore because we avoid boundary effects.

\end{enumerate}
\end{remark}

\begin{remark}[On $m$ vs $n$ -- Asymptotic Considerations] 
As argued earlier, $m$ need not grow with $n$ for consistent estimation. We can nevertheless ask at what rate one might choose to let $m$ grow with $n$, in the spirit of trading off more error due to a higher number of inverse problems to solve in exchange for more data.  Theorem \ref{thm:rate_convergence} can partially inform heuristics on this. Suppose that the error $\hat{K}_{\Omega}$ is $n^{-\alpha}$--consistent for $K_\Omega$, $\alpha \in (0,1)$ -- which is certainly the case under complete observation. Take $\gamma_{m-1}=O(\eta^{m})$ where $\eta \in (0,1)$. Plugging these into Theorem \ref{thm:rate_convergence} would suggest that \begin{equation*}
	\| \hat{K}_{\star} - K_{\star} \|_{L^{2}(I\times I)} \preceq  n^{-\alpha\eta^{m}} = \exp \left[ -\alpha\eta^{m} \log n\right] \to 1
\end{equation*}
 This would indicate that $m$ should not grow any faster than $O(\log \log n)$, effectively leading to $m$ being practically constant: the increase in ill-posedness overwhelms any gain by adding more data.
\end{remark}

\section{Simulation Study}
\label{sec:numerical_simn}
We consider three covariances:
\begin{equation*}
	K_{1}(s,t) = \sum_{j=1}^{4}\frac{ \phi_{j}(s)\phi_{j}(t)}{2^{j-1}}, \qquad K_{2}(s,t) = s\wedge t,\qquad  K_{3}(s,t) = 10st e^{-10|s-t|^{2}}
\end{equation*}
where $\phi_{1}(t) = 1$, $\phi_{2}(t) = \sqrt{3}(2t-1)$, $\phi_{3}(t) = \sqrt{5}(6t^{2} - 6t + 1)$ and $\phi_{4}(t) = \sqrt{7}(20t^{3} - 30t^{2} + 12t - 1)$. The first covariance is both finite-rank and analytic, the second is infinite-rank and non-analytic and the third is infinite-rank and analytic. For the first and the third covariance, every restriction to a serrated domain admits a unique completion (due to analyticity), which is not the case for the second covariance (by Lemma \ref{lem:secondex}). Define the domains $\Omega_{j}$ as follows:
\begin{align*}
	\Omega_{1} &= [0,3/5]^{2} \cup [2/5,1]^{2}\\
	\Omega_{2} &= \Omega_{1} \cup [1/5,4/5]^{2}\\ 
	\Omega_{3} &= \Omega_{2} \cup [1/10,7/10]^{2} \cup [3/10,9/10]^{2} \\
	\Omega_{4} &= \Omega_{3} \cup [1/20,13/20]^{2} \cup [3/20,15/20]^{2} \cup [5/20,17/20]^{2} \cup [7/20,19/20]^{2} \\
	\Omega_{5} &= \Omega_{4} \cup [1/40,25/40]^{2} \cup [3/40,27/40]^{2} \cup [5/40,29/40]^{2} \cup [7/40,31/40]^{2} \\
	&~~~~~~~~~~\cup [9/40,33/40]^{2} \cup [11/40,35/40]^{2} \cup [13/40,37/40]^{2} \cup [15/40,39/40]^{2}
\end{align*}
The number of intervals $m$ for the domains is 2, 3, 5, 9 and 17, respectively. The variable domain simulations of \cite{delaigle2020} roughly correspond to an implicit choice of $m = 17$ (i.e. $\Omega_5$).\\

The computations have been implemented in the \texttt{R} programming language \cite{Rcore2019} with the exception of those involving the estimator proposed in \cite{delaigle2020} which was implemented in MATLAB. The implementation of our estimator in R can be found in the covcomp package \cite{waghmare2022}. 

\subsection{General Simulation Study}
For the covariances $K = K_{1}, K_{2}$ and $K_{3}$, we study the performance of our Estimator \eqref{estimator} for the domains $\Omega_{2}$ $(m = 3)$ and $\Omega_{4}$ $(m = 9)$, and the number of fragments $n = 100$ and $500$, for two different sampling regimes: 
\begin{enumerate} [(a)]
	\item \textit{Regularly Observed Fragments.} We simulate $n$ fragments corresponding to $\Omega$ over a regular grid of size 100 over the unit interval and estimate the covariance over $\Omega$ using the pairwise empirical covariance estimator given by: for $s,t$ such that $n(s,t) = \# \{ j : s,t \in U_{j} \} \geq 10$,
	\begin{equation} \label{eq:emp_cov_est}
		\hat{K}_{\Omega}(s,t) = \frac{1}{n(s,t)}\sum_{j: s,t \in U_{j}}X_{j}(s) X_{j}(t)
	\end{equation}
	where $U_{j} \subset I$ denotes the support of the fragment $X_{j}$. 
	\item \textit{Sparsely Observed Fragments.} We generate $n$ fragments as before in (a) but retain only $\sim 6$ points for every fragment chosen randomly and discard the rest. We estimate the covariance over $\Omega$ by locally linear kernel smoothing. For $K_{1}$ and $K_{2}$, this is achieved using the fdapace package \cite{fdapace} in \texttt{R} under the default parameters. For $K_{2}$, the same method is unsuitable due to non-differentiability at the diagonal, and so we use the reflected triangle estimator proposed in \cite{neda2021} instead.
\end{enumerate}

Using the estimate of $K_{\Omega}$, we compute the completion using the method described in Section \ref{sec:estimation}. We do this 100 times for every combination of covariance, domain and number of samples. We calculate the median and mean absolute deviation for the error in the form of integrated squared error in estimating $K$ over the observed region $\Omega$ and its complement $\Omega^{c}$ to which it is extended using the completion procedure. The results are summarized in Table \ref{tab:simulation_study}. Note that for high sample size combinations in the sparse observation case of $K = K_{2}$, our computational resources proved to be inadequate for using the available implementation of the reflected triangle estimator to complete the computation. For such cases, the results provided are for $n = 200$ and have been marked with an asterisk.  

The choice of truncation parameter can be done using a scree plot or the fraction of variance explained (FVE) criterion given by 
\begin{equation*}
	N_{p} = \min \{r \geq 1: \sum_{j=1}^{r} \hat{\lambda}_{j}(\mathbf{J}_{p}) > 0.95 \cdot \mathrm{tr}(\mathbf{J}_{p}) \}.
\end{equation*}
Here, we choose the truncation parameters manually to illustrate how the nature of the covariance affects the choice of the truncation parameter. For a finite rank covariance, the truncation parameter should be close to but not exceeding the rank. Therefore, for $K_{1}$, we choose $N_{p} = 4$. For infinite rank covariances exhibiting fast eigenvalue decay, such as $K_{2}$ and $K_{3}$, small values of $N_{p}$ such as $2$ or $3$ work well. Accordingly, we choose $N_{p} = 2$ for them. This choice also seems to work slightly better in practice.

The results are summarised in Table \ref{tab:simulation_study}. Naturally, the error in the estimation of $K_{\Omega}$ and $K_{\star}$ tends to decrease as $N$ increases in all cases. The error in estimating $K_{\Omega}$ tends to increase as $m$ increase, even relative to the norm of $K_{\Omega}$. The same applies to the error in estimating $K_{\star}$ in the case of regular observations, however for sparse observations there does not seem to be a clear relationship.

\begin{table}[h]
	\caption{\label{tab:simulation_study} Results of General Simulation Study}
	\subcaption*{Median $\pm$  Mean Absolute Deviation of of Integrated Squared Errors.}
	\resizebox{\columnwidth}{!}{
	\begin{tabular}{||c c c|| c | c || c | c || c c ||} 
		\hline
		\multicolumn{3}{||c||}{Parameters} &\multicolumn{2}{|c||}{Regular Observations} &\multicolumn{2}{|c||}{Sparse Observations} &\multicolumn{2}{|c||}{Squared Norms}\Tstrut \\[2px]
		\hline
		$K$ & $m$ & $N$ & $\int_{\Omega} | \hat{K}_{\Omega} - K_{\Omega} |^{2}$ & $\int_{\Omega^{c}} | \hat{K}_{\star} - K |^{2}$ & $\int_{\Omega} | \hat{K}_{\Omega} - K_{\Omega} |^{2}$ & $\int_{\Omega^{c}} | \hat{K}_{\star} - K |^{2}$ & $\int_{\Omega} |K|^{2}$& $\int_{\Omega^{c}} |K|^{2}$\Tstrut\\ [1ex] 
		\hline\hline
		\multirow{4}{*}{$K_{1}$}	&3	&100	&0.0901	$\pm$0.0604	&0.0318	$\pm$0.0325 &0.1395	$\pm$0.0823	&0.1489	$\pm$0.1630	&1.3080	&0.0381 \Tstrut \\
			&3	&500	&0.0152	$\pm$0.0099	&0.0100	$\pm$0.0112 &0.0397	$\pm$0.0182	&0.0930	$\pm$0.1136	&1.3080	&0.0381 \\
			&9	&100	&0.1781	$\pm$0.1159	&0.1980	$\pm$0.1947 &0.1729	$\pm$0.0893	&0.0749	$\pm$0.0713	&1.3225	&0.0236 \\
			&9	&500	&0.0301	$\pm$0.0197	&0.0482	$\pm$0.0465 &0.0567	$\pm$0.0310	&0.0664	$\pm$0.0633 &1.3225	&0.0236 \\
		\hline
		\multirow{4}{*}{$K_{2}$}	&3	&100	&0.0058	$\pm$0.0052	&0.0007	$\pm$0.0006 &0.0046	$\pm$0.0043	&0.0010	$\pm$0.0011	&0.1573	&0.0094 \Tstrut \\
		   &3  &500    &0.0016	$\pm$0.0013	&0.0002	$\pm$0.0001 &0.0034 $\pm$0.0028*&0.0005 $\pm$0.0005*&0.1573 &0.0094 \\
			&9	&100	&0.0128	$\pm$0.0078	&0.0013	$\pm$0.0008 &0.0068	$\pm$0.0067	&0.0009	$\pm$0.0010	&0.1614	&0.0053 \\
		   &9  &500    &0.0025	$\pm$0.0017	&0.0003	$\pm$0.0002	&0.0037 $\pm$0.0035*&0.0005 $\pm$0.0006*&0.1614 &0.0053 \\	
		\hline
		\multirow{4}{*}{$K_{3}$}	&3	&100	&0.3657	$\pm$0.2879	&0.0231	$\pm$0.0230 &0.6279	$\pm$0.3913	&0.0420	$\pm$0.0391	&5.7930	&0.0021 \Tstrut \\
			&3	&500	&0.0758	$\pm$0.0604	&0.0198	$\pm$0.0119 &0.1594	$\pm$0.0943	&0.0278	$\pm$0.0220	&5.7930	&0.0021 \\
			&9	&100	&0.6543	$\pm$0.4339	&0.0463	$\pm$0.0329 &0.7780	$\pm$0.6108	&0.0569	$\pm$0.0536	&5.7949	&0.0001 \\
			&9	&500	&0.1082	$\pm$0.0764	&0.0231	$\pm$0.0098 &0.2417	$\pm$0.1508	&0.0410	$\pm$0.0266	&5.7949	&0.0001 \\[0.5ex] 		\hline
	\end{tabular}}
	\subcaption*{\footnotesize \qquad\qquad\qquad\qquad\qquad\qquad\qquad\qquad\qquad\qquad\qquad\qquad\qquad\qquad\qquad\qquad*values computed for $n = 200$.}
\end{table}
\subsection{Estimator Performance versus $m$}
We now turn to studying the dependence of the error of estimating $K_{\star}$ on the number of intervals $m$ corresponding to $\Omega$. To this end, we generate $n = 100$ fragments of the covariances $K = K_{1}, K_{2}$ and $K_{3}$ corresponding to the domains $\Omega \in \{ \Omega_{j} \}_{j=1}^{5}$ over a grid of length 100, estimate the partial covariance on a domain using the pairwise empirical covariance estimator as defined by Equation \ref{eq:emp_cov_est} and apply the completion algorithm. 
We then compute the ratio of relative errors (RRE) defined as the ratio of the median ISE in estimating $K_{\Omega^{c}}$ to that of $K_{\Omega}$ both relative to the norms of the respective quantities they are estimating. In other words,
\begin{equation}\label{eq:RRE}	
	\mbox{RRE} = \frac{\int_{\Omega^{c}} | \hat{K}_{\star} - K |^{2} / \int_{\Omega^{c}} | K |^{2}}{\int_{\Omega} | \hat{K}_{\Omega} - K_{\Omega} |^{2} / \int_{\Omega} | K |^{2}}.
\end{equation}

The results have been summarized as boxplots in Figure \ref{fig:boxplot}. We observe that the median RRE does not vary much in response the number of intervals $m$ for any of the covariance scenarios as we move from smaller values such as $m = 2,3$ to larger values such as $m = 17$. The most noticeable effect appears to be in the finite rank case, where the interquantile range of the RRE increases with $m$, even if the median is relatively stable. Importantly, we observe that the increase in error (across scenarios) is nowhere so large so as to affect the utility of the estimation procedure and the empirical performance appears more optimistic than what predicted by Theorem \ref{thm:rate_convergence}. 
\begin{figure}[h]
	\centering
	\begin{center}
		\resizebox{\textwidth}{!}{
			\includegraphics[trim=0 9 0 20,clip, scale=0.7]{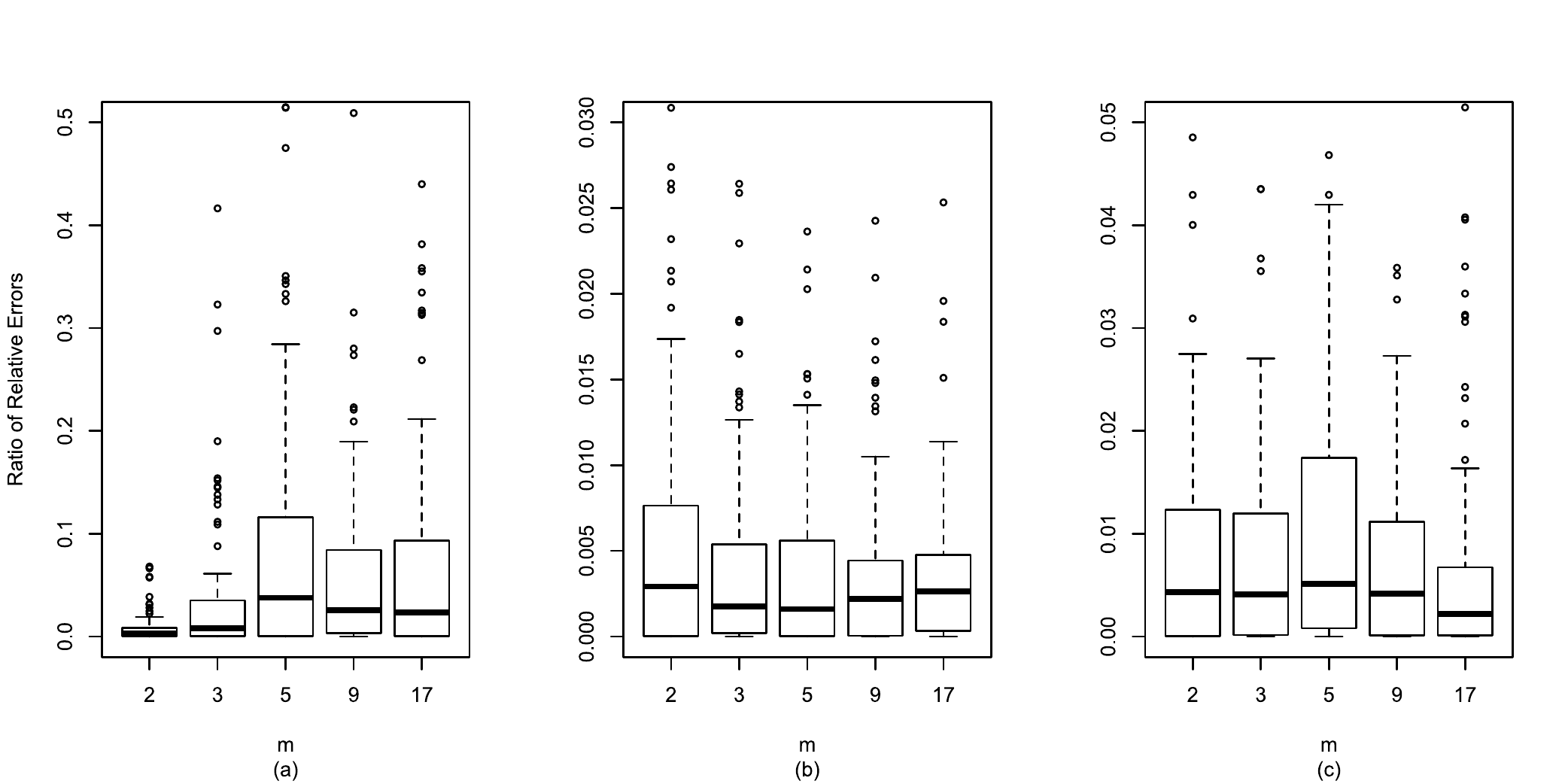}}
		\caption{Boxplot of Ratio of Relative Errors (RRE) vs. the number of intervals $m$ for $K_{1}$ (left), $K_{2}$ (middle) and $K_{3}$ (right).} 
		\label{fig:boxplot}
	\end{center}
\end{figure}

\subsection{Comparative Simulations}
In order to benchmark the performance of our estimator $\hat{K}_{\star}$, we compare to that of the estimator $\hat{K}_{p}$ proposed in \cite{delaigle2020}. 
 For different choices of $K$, $m$ and $n$ we generate fragments on a regular grid of size $50$ on the unit interval. We then estimate the covariance on $\Omega$ using the locally linear kernel smoothing and then apply the completion procedure. We do this 100 times and calculate the median and mean absolute deviation of the integrated squared error. We do the same for the estimator $\hat{K}_{p}$. The results are summarized by Table \ref{tab:compara_study}. As can be expected, neither estimator dominates, and performance varies according to scenario. The scenarios involving $K_{1}$ and $K_{3}$ feature covariances that are analytic and exactly or effectively low rank. As expected, $\hat{K}_{p}$ has better performance here, since these two settings admit unique extension and their infinite smoothness combined with their low (effective) rank is ideally suited for truncated series extrapolation. That being said, the performance of $\hat{K}_{\star}$ remains competitive, with errors of similar magnitude in these two scenarios. On the other hand, $\hat{K}_{\star}$ outperforms $\hat{K}_{p}$ by an order of magnitude in scenario $K_2$, which is a low regularity scenario without a unique completion. One would summarise that $\hat{K}_{\star}$ behaves like a ``robust'' estimator: competitively in ``easy'' scenarios, but substantially better otherwise. Another overarching observation (in line with intuition and theoretical results) is that the performance of $\hat{K}_{\star}$ is tied to the performance of the estimator of $K_{\Omega}$ -- in some cases (see e.g. scenario $K_3$), the larger errors relative to $\hat{K}_{p}$ might have more to do with the quality of estimation on $\Omega$, than the with completion procedure itself.


\begin{center}
\begin{table}
		\caption{\label{tab:compara_study} Results of Comparative Simulation Study}
	\subcaption*{Median $\pm$  Mean Absolute Deviation of of Integrated Squared Errors.}
	\resizebox{0.80\columnwidth}{!}{
		\begin{tabular}{||c | c c || c | c || c | c ||} 
			\hline
			\multicolumn{3}{||c||}{Parameters} &\multicolumn{2}{|c||}{Our estimator $\hat{K}_{\star}$} &\multicolumn{2}{|c||}{The estimator $\hat{K}_{p}$} \Tstrut\\[2px]
			\hline
			\hline
			$K$ & $m$ & $N$ & $\int_{\Omega} | \hat{K}_{\Omega} - K_{\Omega} |^{2}$ & $\int_{\Omega^{c}} | \hat{K}_{\star} - K |^{2}$ & $\int_{\Omega} | \hat{K}_{p} - K_{\Omega} |^{2}$& $\int_{\Omega^{c}} | \hat{K}_{p} - K |^{2}$ \Tstrut\\ [1ex] 
			\hline\hline
			\multirow{4}{*}{$K_{1}$} 
			 &3 &100 &0.0779 $\pm$0.0541    &0.0984 $\pm$0.1080 	&{0.0614 $\pm$0.0402}	&{0.0742	$\pm$0.0407}  \Tstrut\\
			 &3 &500 &0.0157 $\pm$0.0103    &0.0288 $\pm$0.0295	&{0.0146 $\pm$0.0086}	&{0.0145	$\pm$0.0099}  \\
			 &9 &100 &0.1250 $\pm$0.0773    &0.1400 $\pm$0.1450	&{0.0982	$\pm$0.0479}	&{0.1115	$\pm$0.0721}  \\ 
			 &9 &500 &0.0279 $\pm$0.0204    &0.0622 $\pm$0.0760	&{0.0225	$\pm$0.0123}	&{0.0224	$\pm$0.0143}  \\
			\hline
			\multirow{4}{*}{$K_{2}$}	         &3	           &100	&0.0058	$\pm$0.0053	&0.0007	$\pm$0.0005	&0.0049	$\pm$0.0041	&0.0055	$\pm$0.0049	\Tstrut \\
				         &3	           &500	&0.0016	$\pm$0.0014	&0.0001	$\pm$0.0001	&0.0009	$\pm$0.0006	&0.0010	$\pm$0.0008	\\
				         &9	           &100	&0.0129	$\pm$0.0073	&0.0012	$\pm$0.0007	&0.0081	$\pm$0.0065	&0.0059	$\pm$0.0052 \\
				         &9	           &500	&0.0029	$\pm$0.0020 &0.0003	$\pm$0.0002	&0.0020	$\pm$0.0017	&0.0012	$\pm$0.0009	\\
			\hline
			\multirow{4}{*}{$K_{3}$}	     
			&3	   &100	  &0.0232 $\pm$0.0202 &0.0057 $\pm$0.0059	&{0.0023	$\pm$0.0018}	&{0.0025	$\pm$0.0022}	\Tstrut \\
			&3     &500 &0.0053 $\pm$0.0039 &0.0009 $\pm$0.0008 	&{0.0006	$\pm$0.0004}	&{0.0006	$\pm$0.0004}	\\
			&9	   &100	 &0.0340 $\pm$0.0254 &0.0064 $\pm$0.0054 	&{0.0050	$\pm$0.0035}	&{0.0044	$\pm$0.0034}	\\
			&9	   &500&0.0058 $\pm$0.0037 &0.0013 $\pm$0.0010   &{0.0010	$\pm$0.0009}	&{0.0008	$\pm$0.0005}	\\\hline	
	\end{tabular}}
\end{table}
\end{center}

\section{Illustrative Data Analysis}\label{sec:illustration}
Following \cite{delaigle2020}, we apply our method to the spine bone mineral density (BMD) data described in \cite{bachrach1999}. We consider measurements of 117 females taken between the ages of 9.5 and 21 years. The measurements for every subject are taken over a short period of time, comprising in each case an interval far shorter than the age-range interval. Hence, the measurements on each subject constitute independent sparsely observed fragments, see Figure \ref{fig:scatter} (left) and yield information only on a partial covariance. Nevertheless, if we wish to conduct further analyses such as classification, regression, prediction, or even dimension reduction, we need access to a complete covariance.

\begin{figure}[h]
\includegraphics[scale=0.62]{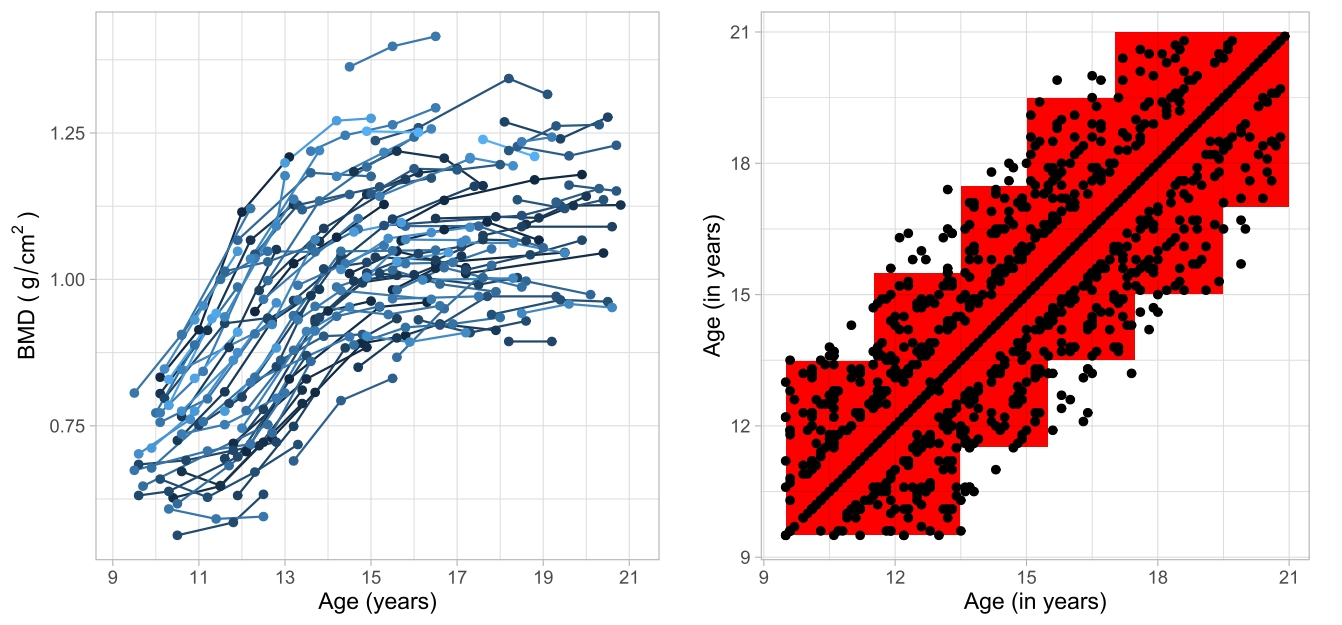}
\caption{(left) Sparsely observed spine BMD curves for 117 females (right) Scatter plot of pairs of ages for which simultaneous observations are available.}
\label{fig:scatter}
\end{figure}

We plot all those pairs of ages for which we have measurements on the same subject, see Figure \ref{fig:scatter} (right). Based on the plot, we infer that the covariance can be estimated reasonably well over the serrated domain $\Omega$ (colored in red) corresponding to the intervals $[9.5,13.5]$, $[11.5,15.5]$, $[13.5,17.5]$, $[15,19.5]$ and $[17,21]$. We then estimate the covariance on $\Omega$ using locally linear kernel smoothing through the fdapace package and use the completion algorithm to estimate the covariance over the entire region, see Figure \ref{fig:cov}.
\begin{figure}[h]
	\includegraphics[clip, trim=0.1cm 0.1cm 2.9cm 0.1cm, width=0.60\textwidth]{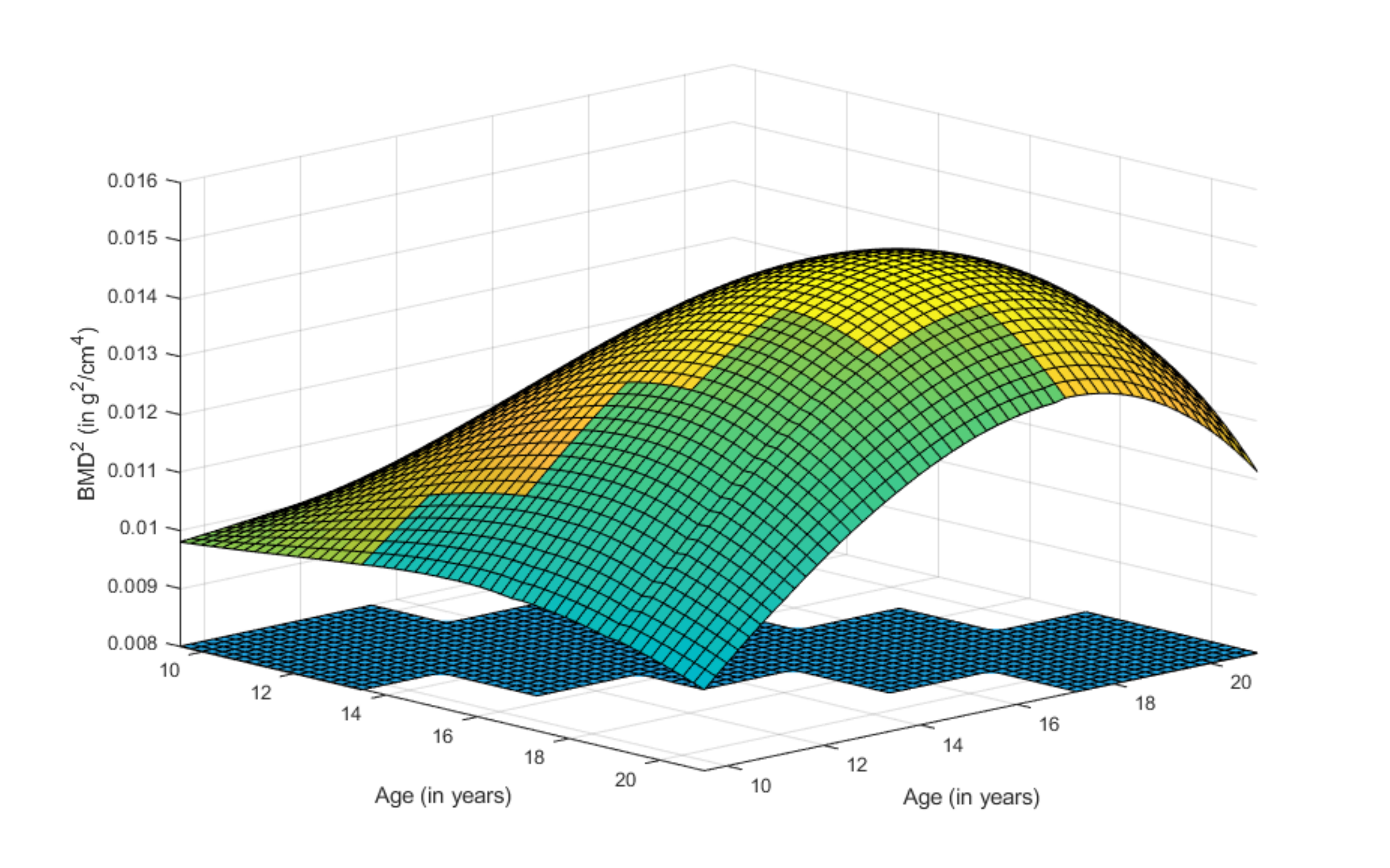}	
	\caption{Completed covariance of the BMD data.}
	\label{fig:cov}
\end{figure}

\section{Proofs of Formal Statements}\label{sec:proofs} 
We will arrange our proofs into subsections that parallel the corresponding sections of the paper. We shall make extensive use of the projection theorem as well as certain isometries between Hilbert spaces, such as the Lo\`{e}ve isometry.  For tidiness, we introduce some shorthand notation for the restrictions $K_{\Omega}(t, \cdot)$ or $K(t, \cdot)$: for $u\in I$ and $J \subset I$, we denote by $k_{u,J}$ the function $k_{u, J}: J \to \mathbb{R}$ given by $k_{u,J}(v) = K(u,v)$ or $K_{\Omega}(u,v)$ according to the context. Similarly, we denote by $k^{\star}_{u,J}$ the function $k^{\star}_{u, J}: J \to \mathbb{R}$ given by $k^{\star}_{u,J}(v) = K_{\star}(u,v)$.

Moreover, for every covariance $K$ we have $k_{u,J} \in \mathscr{H}(K_{J})$ for every $u \in I$ and $J \subset I$. This is known as the restriction theorem \cite[Corollary 5.8]{paulsen2016}. For $f,g \in \mathscr{H}(K_{J})$ we shall denote the norm $\|f\|_{\mathscr{H}(K_{J})}$ and the inner product $\langle f,g\rangle_{\mathscr{H}(K_{J})}$ in $\mathscr{H}(K_{J})$ simply as $\|f\|$ and $\langle f,g\rangle$, since the Hilbert space can always be inferred form the domain of the involved function. 

\subsection{The Canonical Completion}
The completion $K_{\star}$ as defined in Equation (\ref{2-serrated-completion}) is well-defined because the restrictions of $K_{\Omega}(s, \cdot)$ and $K_{\Omega}(t, \cdot)$ to $I_{1}\cap I_{2}$ belong to the RKHS $\mathscr{H}(K_{I_{1}\cap I_{2}})$ thanks to the restriction theorem \cite[Corollary 5.8]{paulsen2016}. We shall however make use of a stronger result in order to prove our theorems. 

Let $H_{J}$ denote the closed subspace spanned by $\{k_{x}: x\in J\}$ in $\mathscr{H}$ and let $\Pi_{J}$ denote the projection from $\mathscr{H}$ to $H_{J}$. Then the subspace $H_{J}$ is isomorphic to the reproducing kernel Hilbert space $\mathscr{H}(K_{J})$.
\begin{theorem} \label{subspace_isometry}
	There exists an isometry $\rho : H_{J} \to \mathscr{H}(K_{J})$ such that its adjoint $\rho^{\ast}$ satisfies $\rho^{\ast}g|_{J} = g$ for $g \in H_{J}$.
\end{theorem}
\begin{proof}
	Define a linear map $\sigma_{0}: \mathrm{Span} \{ k_{x,J}: x\in J\} \to H_{J}$ by $\sigma_{0}(k_{x,J}) = k_{x}$ for $x \in J$. For $f = \sum_{j=1}^{n} c_{j}k^{J}_{x_{j}}$, we have
	\begin{equation*}
		\| \sigma_{0}(f) \|^{2} = \sum_{i,j=1}^{n} c_{i}c_{j} K(x_{i},x_{j}) = \| f\|^{2}
	\end{equation*}
	Therefore, if $f = 0$, then $\sigma_{0}(f) = 0$. It follows that the map $\sigma_{0}$ is well-defined, injective and continuous. Moreover, it maps $\mathrm{Span} \{ k_{x,J}: x\in E\}$ onto $\mathrm{Span} \{ k_{x,J}: x\in J\}$.
	
	Extending $\sigma_{0}$ by continuity from $\mathrm{Span}\{ k_{x,J}: x\in J\}$ to $\mathscr{H}(K_{J})$ gives $\sigma: \mathscr{H}(K_{J}) \to H_{J}$ such that $\sigma(f)=\sigma_{0}(f)$ for $f \in \mathrm{Span}\{ k_{x,J}: x\in J\}$. Additionally, $\| \sigma(f) \| = \|f\|$ for every $f \in \mathscr{H}(K_{J})$ and therefore $\sigma$ is also well-defined, injective and continuous. 
	
	To show that $\sigma$ is surjective, pick any $g \in H_{J}$. Then there exists a sequence $\{g_{j}\}_{j=1}^{\infty} \subset \mathrm{Span}\{ k_{x}: x\in J\}$ such that $g_{j} \to g$ in $\mathscr{H}$. Let $f_{j} \in \mathrm{Span}\{ k_{x,J}: x\in J\}$ be such that $\sigma f_{j} = g_{j}$ for $j \geq 1$. Since $\{g_{j}\}_{j=1}^{\infty}$ is Cauchy, so is $\{f_{j}\}_{j=1}^{\infty}$ and therefore it converges to some $f \in \mathscr{H}(K_{J})$ such that $\sigma f = g$.
	
	Now, notice that for every $x\in J$,  
	\begin{equation*}
		\sigma^{\ast}g(x) = \langle \sigma^{\ast}g, k_{x,J}\rangle = \langle g, \sigma k_{x,J}\rangle = \langle g, k_{x} \rangle = g(x)
	\end{equation*}
	Define $\rho = \sigma^{\ast}$ and the conclusion follows.
\end{proof}
We shall see that the above isometry enables a sort of infinite-dimensional matrix algebra with the partial covariance in order to recover its unknown values. We shall refer to it as the \textit{subspace isometry}, in contrast to another isometry we often make use of, which is the \textit{Lo\`{e}ve isometry}.

\begin{remark}
	\label{proj_is_res}
	Notice that for $f,g \in \mathscr{H}(K)$ and  $\Pi_{J}: \mathscr{H}(K) \to \mathscr{H}(K)$  the closed subspace spanned by $\{k_{u}: u \in J\}$, we have
	\begin{equation*}
		\begin{split}
			\langle \Pi_{J}f, \Pi_{J}g \rangle_{\mathscr{H}(K)} 
			&= \langle \rho\Pi_{J}f, \rho\Pi_{J}g \rangle_{\mathscr{H}(K_{J})} \\
			&= \langle \Pi_{J}f|_{J}, \Pi_{J}g|_{J} \rangle_{\mathscr{H}(K_{J})} \\
			&= \langle f|_{J}, g|_{J} \rangle_{\mathscr{H}(K_{J})}
		\end{split}
	\end{equation*} 
	since for $t\in J$, $\Pi_{J}f(t) = \langle \Pi_{J}f, k_{t}\rangle = \langle f, \Pi_{J}k_{t}\rangle = \langle f, k_{t}\rangle = f(t)$. Thus, projection boils down to restriction. 
\end{remark}

\begin{proof}[Proof of Theorem \ref{2-serrated-theorem}]

	It suffices for us to construct a Hilbert space $\mathscr{H}$ with a set of vectors $\{ \varphi_{x} \}_{x\in I} \subset \mathscr{H}$ such that $K_{\star}(s,t) = \langle \varphi_{s}, \varphi_{t} \rangle$ for every $s,t\in I$. Accordingly, we let $\mathscr{H} = \mathscr{H}(K_{I_{1}}) \oplus \mathscr{H}(K_{I_{2}})$, the direct sum of the reproducing kernel Hilbert spaces of $K_{I_{1}}$ and the space $K_{I_{2}}$.\\ 
	
	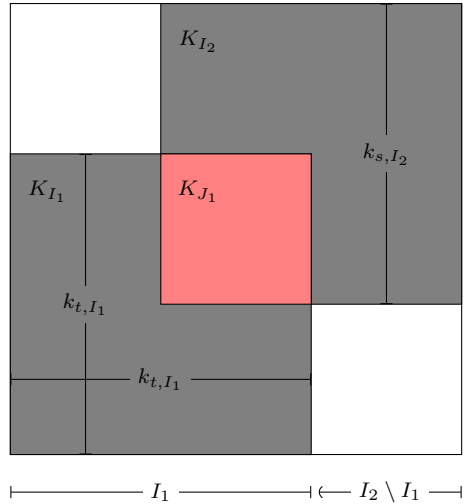
\begin{figure}[h]\label{fig:part_cov2}
		\centering
		\begin{center}
			\begin{tikzpicture}
				\draw [shift={(0.5,0.5)}] (0,0) -- (0,6) -- (6,6) -- (6,0) -- cycle;
				
				\draw [shift={(2.5,2.5)}, fill=gray] (0,0) -- (0,4) -- (4,4) -- (4,0) -- cycle;
				\draw [shift={(0.5,0.5)}, fill=gray] (0,0) -- (0,4) -- (4,4) -- (4,0) -- cycle;
				\draw [shift={(2.5,2.5)}, fill=red!50] (0,0) -- (0,2) -- (2,2) -- (2,0) -- cycle;
				
				\draw [shift={(0.5,0.5)}] (0.5, 3.5) node {$K_{I_{1}}$};
				\draw [shift={(2.5,2.5)}] (0.5, 3.5) node {$K_{I_{2}}$};
				\draw [shift={(2.5,0.5)}] (0.5, 3.5) node {$K_{J_{1}}$};
				
				\draw [|-|,shift = {(0.5,0.5)}] (1,0) -- (1,4) node [midway,fill=gray] {$k_{t,I_{1}}$};
				\draw [|-|,shift = {(0.5,0.5)}] (0,1) -- (4,1) node [midway,fill=gray] {$k_{t,I_{1}}$};
				\draw [|-|,shift = {(4,2.5)}] (1.5,0) -- (1.5,4) node [midway,fill=gray] {$k_{s,I_{2}}$};
				
				\draw [{Parenthesis}-|,shift = {(0.5,0)}] (4.1,0) -- (6,0) node [midway,fill=white] {$I_{2} \setminus I_{1}$};
				\draw [|-|,shift = {(0.5,0)}] (0,0) -- (4,0) node [midway,fill=white] {$I_{1}$};			
			\end{tikzpicture}
			\caption{The Partial Covariance $K_{\Omega}$} 
		\end{center}
	\end{figure}
	
	Let $H_{1}$ denote the closed subspace in $\mathscr{H}(K_{I_{1}})$ spanned by $\{ k_{t, I_{1}} : t\in J_{1}\}$ and similarly, let $H_{2}$ denote the closed subspace in $\mathscr{H}(K_{I_{2}})$ generated by $\{ k_{t, I_{2}} : t\in J_{1}\}$.
	By Theorem \ref{subspace_isometry}, both $H_{1}$ and $H_{2}$ are isomorphic to $\mathscr{H}(K_{J_{1}})$ with the restrictions $\rho_{1}: H_{1} \to \mathscr{H}(K_{J_{1}})$ and $\rho_{2}: H_{2} \to \mathscr{H}(K_{J_{1}})$ given by $\rho_{1}f = f|_{J_{1}}$ and $\rho_{2}g = g|_{J_{1}}$, serving as isometries. It follows that $H_{1}$ and $H_{2}$ are isomorphic, with the isometry $\rho_{1}^{\ast}\rho_{2} : H_{2} \to H_{1}$. Also, let $\Pi_{J_{1}}: \mathscr{H}(K_{I_{2}}) \to \mathscr{H}(K_{I_{2}})$ denote the projection to $H_{2}$. \\

	Define $\varphi_{t}$ as follows,
	\begin{equation*}
		\varphi_{t} = \begin{cases}
			k_{t, I_{1}} \oplus 0 & \mbox{ if } t \in I_{1} \\
			\rho_{1}^{\ast}\rho_{2}\Pi_{J_{1}}k_{t, I_{2}} \oplus \left[ k_{t, I_{2}} - \Pi_{J_{1}}k_{t, I_{2}} \right] & \mbox{ if } t \in I_{2} \setminus I_{1} \\
		\end{cases}
	\end{equation*}
	All that remains now, is for us to verify that $\langle \varphi_{s}, \varphi_{t} \rangle $ is indeed equal to $K_{\star}(s,t)$ for every $s,t \in I$. We shall do this on a case-by-case basis as follows:
	\begin{enumerate}[\hspace{0.25cm} \bfseries {Case} 1. ] 
		\item If both $s$ and $t \in I_{1}$, then $\langle \varphi_{s}, \varphi_{t} \rangle = \langle k_{s, I_{1}} , k_{t, I_{1}}  \rangle + 0 = K_{I_{1}}(s,t) = K_{\star}(s,t)$.
		\item If $s \in I_{1}$ and $t \in I_{2} \setminus I_{1}$, then by the projection theorem and Theorem \ref{subspace_isometry} we get
		\begin{equation*}
			\begin{split}
				\langle \varphi_{s}, \varphi_{t} \rangle 
				= \langle k_{s, I_{1}}, \rho_{1}^{\ast}\rho_{2}\Pi_{J_{1}}k_{t, I_{2}} \rangle + 0 
				= \langle \rho_{1}k_{s, I_{1}}, \rho_{2}\Pi_{J_{1}}k_{t, I_{2}} \rangle 
				= \langle k_{s, J_{1}}, k_{t, J_{1}} \rangle
			\end{split}
		\end{equation*}	
		If $s \in J_{1}$, then $\langle k_{s, J_{1}}, k_{t, J_{1}} \rangle = K_{J_{1}}(s,t) = K_{\star}(s,t)$. On the other hand, if $s \in I_{1}\setminus J$, $\langle k_{s, J_{1}}, k_{t, J_{1}} \rangle = K_{\star}(s,t)$ by definition.
		\item Covered by Case 2 by symmetry.
		\item If both $s$ and $t \in I_{2} \setminus I_{1}$, then
		\begin{equation*}
			\langle \varphi_{s}, \varphi_{t} \rangle
			= \langle \rho_{1}^{\ast}\rho_{2} \Pi_{J_{1}}k_{s, I_{2}}, \rho_{1}^{\ast}\rho_{2} \Pi_{J_{1}}k_{t, I_{2}} \rangle + \langle k_{s, I_{2}} - \Pi_{J_{1}}k_{s, I_{2}}, k_{t, I_{2}} - \Pi_{J_{1}}k_{t, I_{2}} \rangle \\
		\end{equation*}
		By Theorem \ref{subspace_isometry}, $\langle \rho_{1}^{\ast}\rho_{2} \Pi_{J_{1}}k_{s, I_{2}}, \rho_{1}^{\ast}\rho_{2} \Pi_{J_{1}}k_{t, I_{2}} \rangle = \langle \Pi_{J_{1}}k_{s, I_{2}}, \Pi_{J_{1}}k_{t, I_{2}} \rangle = \langle k_{s, J_{1}}, k_{t, J_{1}} \rangle$. And using the projection theorem,
		\begin{equation*}
			\begin{split}
				\langle k_{s, I_{2}} - \Pi_{J_{1}}k_{s, I_{2}}&, k_{t, I_{2}} - \Pi_{J_{1}}k_{t, I_{2}} \rangle \\
				&= \langle k_{s, I_{2}} - \Pi_{J_{1}}k_{s, I_{2}}, k_{t, I_{2}} \rangle - \langle k_{s, I_{2}} - \Pi_{J_{1}}k_{s, I_{2}},  \Pi_{J_{1}}k_{t, I_{2}} \rangle\\
				&= \langle k_{s, I_{2}}, k_{t, I_{2}} \rangle - \langle \Pi_{J_{1}}k_{s, I_{2}}, k_{t, I_{2}} \rangle - 0 \\
				&= \langle k_{s, I_{2}}, k_{t, I_{2}} \rangle - \langle \Pi_{J_{1}}k_{s, I_{2}}, \Pi_{J_{1}}k_{t, I_{2}} \rangle \\
				&= K_{I_{2}}(s,t) - \langle k_{s, J_{1}}, k_{t, J_{1}} \rangle\\
			\end{split}
		\end{equation*}
		Thus, $	\langle \varphi_{s}, \varphi_{t} \rangle = K_{I_{2}}(s,t) = K_{\star}(s,t)$. 
	\end{enumerate}
	Now that we have established that $K_{\star}(s,t) = \langle \varphi_{s}, \varphi_{t} \rangle$ is a covariance extension of $K_{\Omega}$, we need only verify that $K(s,t) = \langle k_{s, J_{1}}, k_{t, J_{1}} \rangle$ for $(s,t) \in \Omega^{c}$. For $s \in I_{2} \setminus I_{1}$ and $t \in I_{1} \setminus I_{2} \subset I_{1}$
	\begin{equation*}
		\begin{split}
			\langle \varphi_{s}, \varphi_{t} \rangle 
			&= \langle \rho_{1}^{\ast}\rho_{2}\Pi_{J_{1}}k_{s, I_{2}} \oplus \left[ k_{s, I_{2}} - \Pi_{J_{1}}k_{s, I_{2}} \right], k_{t, I_{1}} \oplus 0  \rangle\\   
			&= \langle \rho_{1}^{\ast}\rho_{2}\Pi_{J_{1}}k_{s, I_{2}}, k_{t, I_{1}}\rangle\\  
			&= \langle \rho_{2}\Pi_{J_{1}}k_{s, I_{2}}, \rho_{1}k_{t, I_{1}}\rangle\\ 
			&= \langle k_{s, J_{1}}, k_{t, J_{1}} \rangle\\          
		\end{split}
	\end{equation*}
	This completes the proof.
\end{proof}

\begin{remark}
	\label{any-2-sets}
	Nothing in the proof above requires that $I_{1}$ and $I_{2}$ have to be intervals of the real line. In fact, the result holds true so long as $I_{1}$ and $I_{2}$ are any two sets with a non-empty intersection. 
\end{remark}

\begin{proof}[Proof of Theorem \ref{serrated-theorem}]	
	By Theorem \ref{2-serrated-theorem}, the result of a 2-serrated completion is a valid completion and therefore, the same should be two for successive 2-serrated completions. 
	
	Let $K_{\star}$ denote the completion obatined by using Algorithm \ref{algorithm}. In order to show that the resulting completion is independent of the order in which the completion is carried out, it suffices to show that for $s,t \in I$ separated by $J_{p}$ and $J_{q}$, we have that $\langle k_{s, J_{p}}, k_{t, J_{p}} \rangle = \langle k_{s, J_{q}}, k_{t, J_{q}}\rangle$.  
	
	We proceed by induction. The statement is vacuously true for the case $m=2$ by Theorem \ref{2-serrated-theorem}. We shall prove it for $m=3$ and on. For $m=3$, it suffices for us to show that for $(s,t) \in (I_{3} \setminus I_{2}) \times (I_{1} \setminus I_{2})$, 
	\begin{equation*}
		\langle k^{\star}_{s, J_{1}}, k^{\star}_{t, J_{1}} \rangle = \langle k^{\star}_{s, J_{2}}, k^{\star}_{t, J_{2}} \rangle
	\end{equation*}
	since only such $s$ and $t$ are separated by both $J_{1}$ and $J_{2}$.
	By definition, $k^{\star}_{s, J_{1}}(u) = \langle k^{\star}_{s, J_{2}}, k^{\star}_{u, J_{2}} \rangle$ for $u \in J_{1}$. $k^{\star}_{s, J_{1}}$ can be written in terms of $k^{\star}_{s, J_{2}}$ in a more concise way, as an image of a linear operator. 
	
	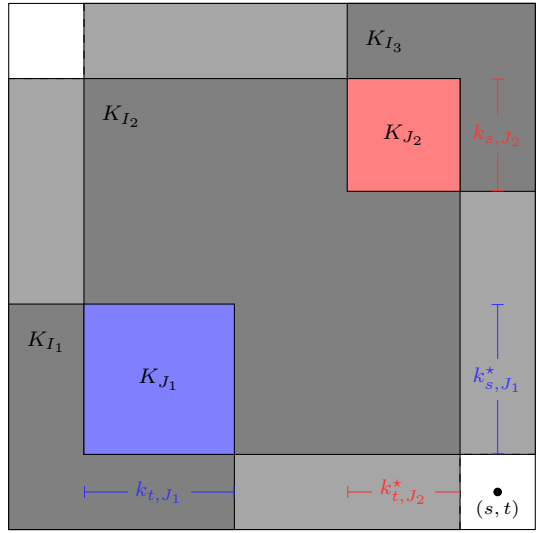
\begin{figure}[h]
			\begin{tikzpicture}
				\draw [shift={(0.5,0.5)}] (0,0) -- (0,7) -- (7,7) -- (7,0) -- cycle;
				\draw [shift={(0.5,0.5)}, fill=gray!70] (0,0) -- (6,0) -- (6,1) -- (7,1) -- (7,7) -- (1,7) -- (1,6) -- (0,6) -- cycle;
				\draw [shift={(0.5,0.5)}, fill=gray] (0,0) -- (0,3) -- (3,3) -- (3,0) -- cycle;			
				\draw [shift={(5,5)}, fill=gray] (0,0) -- (0,2.5) -- (2.5,2.5) -- (2.5,0) -- cycle;
				\draw [shift={(1.5,1.5)}, fill=gray] (0,0) -- (0,5) -- (5,5) -- (5,0) -- cycle;
				
				\draw [shift={(1.5,1.5)}, fill=blue!50] (0,0) -- (0,2) -- (2,2) -- (2,0) -- cycle;
				\draw [shift={(5,5)}, fill=red!50] (0,0) -- (0,1.5) -- (1.5,1.5) -- (1.5,0) -- cycle;
				\draw [shift={(0.5,0.5)},dashed] (6,0) -- (6,1) -- (7,1);
				\draw [shift={(0.5,0.5)},dashed] (0,6) -- (1,6) -- (1,7);
				
				\draw [|-|,shift = {(0.5,0.5)}, blue!80] (1,0.5) -- (3,0.5) node [midway,fill=gray] {$k_{t,J_{1}}$};
				\draw [|-|,shift = {(0.5,0.5)}, blue!80] (6.5,1) -- (6.5,3) node [midway,fill=gray!70] {$k^{\star}_{s,J_{1}}$};
				\draw [|-|,shift = {(5,0.5)}, red!80] (0,0.5) -- (1.5,0.5) node [midway,fill=gray!70] {$k^{\star}_{t,J_{2}}$};
				\draw [|-|,shift = {(0.5,0.5)}, red!80] (6.5,4.5) -- (6.5,6) node [midway,fill=gray] {$k_{s,J_{2}}$};
				
				\draw [shift={(0.5,0.5)},fill] (6.5,0.5) circle [radius=0.05];
				\draw [shift={(0.5,0.5)}] (6.5,0.5) node [below] {{\scriptsize $(s,t)$}};
				
				\draw [shift={(0.5,0.5)}] (0.5,2.5) node {$K_{I_{1}}$};			
				\draw [shift={(0.5,0.5)}] (5,6.5) node {$K_{I_{3}}$};
				\draw [shift={(0.5,0.5)}] (1.5,5.5) node {$K_{I_{2}}$};
				\draw [shift={(0.5,0.5)}] (2,2) node {$K_{J_{1}}$};
				\draw [shift={(0.5,0.5)}] (5.25,5.25) node {$K_{J_{2}}$};
			\end{tikzpicture}
			\caption{The Partial Covariance $K_{\Omega}$} 
			\label{fig:part_cov3}
	\end{figure}
	
	Let $\rho_{1}: \mathscr{H}(K_{I_{2}}) \to \mathscr{H}(K_{J_{1}})$ and $\rho_{2}: \mathscr{H}(K_{I_{2}}) \to \mathscr{H}(K_{J_{2}})$ denote the restrictions given by $\rho_{1}f = f|_{J_{1}}$ and $\rho_{2}f = f|_{J_{2}}$ respectively. Let $\Pi_{1}: \mathscr{H}(K_{I_{2}}) \to \mathscr{H}(K_{I_{2}})$ denote the projection to the closed subspace generated by $\{k_{u, I_{2}}: u \in J_{1}\}$. Now, for $u \in J_{1}$,
	\begin{equation*}
		\begin{split}
			\rho_{1} \Pi_{1} \rho_{2}^{\ast} k^{\star}_{s, J_{2}}(u)
			&= \langle \rho_{1} \Pi_{1} \rho_{2}^{\ast} k^{\star}_{s, J_{2}}, k_{u, J_{1}} \rangle \\
			&= \langle k^{\star}_{s, J_{2}}, \rho_{2} \Pi_{1}\rho_{1}^{\ast} k_{u, J_{1}} \rangle\\
			&= \langle k^{\star}_{s, J_{2}}, \rho_{2} \Pi_{1}k^{\star}_{u, I_{2}} \rangle\\
			&= \langle k^{\star}_{s, J_{2}}, \rho_{2} k^{\star}_{u, I_{2}} \rangle\\
			&= \langle k^{\star}_{s, J_{2}}, k^{\star}_{u, J_{2}} \rangle.\\
		\end{split}
	\end{equation*}
	because $\rho_{1}^{\ast} k_{u, J_{1}} = k_{u, I_{2}}$, $\rho_{2} k^{\star}_{u, I_{2}} = k^{\star}_{u, J_{2}}$, and $\Pi_{1} k^{\star}_{u, I_{2}} = k^{\star}_{u, I_{2}}$ as $u \in J_{1}$.  Therefore, $k^{\star}_{s, J_{1}} =  \rho_{1} \Pi_{1} \rho_{2}^{\ast} k^{\star}_{s, J_{2}}$. Using this representation,
	\begin{equation*}
		\begin{split}
			\langle k^{\star}_{s, J_{1}}, k^{\star}_{t, J_{1}} \rangle 
			&= \langle \rho_{1} \Pi_{1} \rho_{2}^{\ast} k^{\star}_{s, J_{2}}, k^{\star}_{t, J_{1}} \rangle \\
			&= \langle k^{\star}_{s, J_{2}}, \rho_{2} \Pi_{1} \rho_{1}^{\ast}k^{\star}_{t, J_{1}} \rangle \\
			&= \langle k^{\star}_{s, J_{2}}, k^{\star}_{t, J_{2}} \rangle. \\
		\end{split}
	\end{equation*}	
	Thus, $K_{\star}$ indeed satisfies the separation condition. Uniqueness follows by observing that the separation condition uniquely determines $K_{\star}$ given $K_{\star}|_{I_{3}\times I_{3}}$ and $K_{\star}|_{S_{1}^{2}\times S_{1}^{2}}$, which is in turn uniquely determined given $K_{\star}|_{I_{1} \times I_{1}}$ and $K_{\star}|_{I_{2} \times I_{2}}$.\\
	
	Now, assuming the statement for $m \leq q$, we consider the case $m = q+1$. Thus, $K_{\star}|_{[\cup_{j=1}^{k}I_{j}] \times [\cup_{j=1}^{k}I_{j}]}$ and $K_{\star}|_{[\cup_{j=2}^{k+1}I_{j}] \times [\cup_{j=2}^{k+1}I_{j}]}$ are uniquely determined. It suffices to verify the separating condition for the remaining part. So we need to show that for $(s,t) \in (I_{q+1} \setminus I_{q}) \times (I_{1} \setminus I_{2})$, that,
	\begin{equation*}
		\langle k^{\star}_{s, J_{1}}, k^{\star}_{t, J_{1}} \rangle 
		= \langle k^{\star}_{s, J_{2}}, k^{\star}_{t, J_{2}} \rangle 
		= \cdots
		= \langle k^{\star}_{s, J_{q}}, k^{\star}_{t, J_{q}} \rangle.
	\end{equation*}
	Pick $1 \leq p < q$, and let $I_{1}^{\prime} = [\cup_{j=1}^{p}I_{j}]$, $I_{2}^{\prime} = [\cup_{j=p+1}^{q}I_{j}]$, and $I_{3}^{\prime} = I_{q+1}$. Then $\Omega^{\prime} = \cup_{j=1}^{3} I^{\prime}_{j} \times I^{\prime}_{j}$ is a serrated domain of three intervals and we are back to the case when $m=3$. This implies that, $\langle k^{\star}_{s, J^{\prime}_{1}}, k^{\star}_{t, J^{\prime}_{1}} \rangle = \langle k^{\star}_{s, J^{\prime}_{2}}, k^{\star}_{t, J^{\prime}_{2}} \rangle$ for $J^{\prime}_{1} = I_{1}^{\prime} \cap I_{2}^{\prime} = J_{p}$ and $J^{\prime}_{2} = I_{2}^{\prime} \cap I_{3}^{\prime} = J_{q}$. Since $p$ was chosen arbitrarily, it follows that, $\langle k^{\star}_{s, J_{p}}, k^{\star}_{t, J_{p}} \rangle = \langle k^{\star}_{s, J_{q}}, k^{\star}_{t, J_{q}} \rangle$ for $1 \leq p < q$. Uniqueness follows the same way as in the case $m=3$. Hence proved.
\end{proof}

\begin{lemma} \label{lemma:fund3}
	Let $J \subset I$ be a separator of $\Omega$ containing $J_{p}$ for some $p$. Then $K$ as defined above, satisfies
	\begin{equation}
		\langle k_{s, J_{p}}, k_{t, J_{p}} \rangle_{\mathscr{H}(K_{J_{p}})} = \langle k_{s, J}, k_{t, J} \rangle_{\mathscr{H}(K_{J})}
	\end{equation}
	for every $s,t \in I$ separated by $J$.
\end{lemma}
\begin{proof}	
	By Theorem \ref{serrated-theorem}, $K_{\star}(s,t) = \langle k^{\star}_{s, J_{p}}, k^{\star}_{t, J_{p}} \rangle$ and by Remark \ref{proj_is_res}, $\langle k^{\star}_{s, J_{p}}, k^{\star}_{t, J_{p}} \rangle = \langle \Pi_{p}k^{\star}_{s, J}, \Pi_{p}k^{\star}_{t, J} \rangle$, where $\Pi_{p}: \mathscr{H}(K_{\star}|_{J \times J}) \to \mathscr{H}(K_{\star}|_{J \times J})$ denotes the projection to the closed subspace spanned by $\{ k^{\star}_{u, J}: u \in J_{p}\}$. Thus, all we need to show is for $s,t \in I$ separated by $J \subset I$,
	\begin{equation*}
		\begin{split}
			\langle k^{\star}_{s, J} - \Pi_{p}k^{\star}_{s, J}, k^{\star}_{t, J} - \Pi_{p}k^{\star}_{t, J} \rangle 
			&= \langle k^{\star}_{s, J}, k^{\star}_{t, J} \rangle  - \langle \Pi_{p}k^{\star}_{s, J}, \Pi_{p}k^{\star}_{t, J} \rangle\\
			&= \langle k^{\star}_{s, J}, k^{\star}_{t, J} \rangle  - \langle k^{\star}_{s, J_{p}}, k^{\star}_{t, J_{p}} \rangle\\
			&= 0.
		\end{split}
	\end{equation*}
	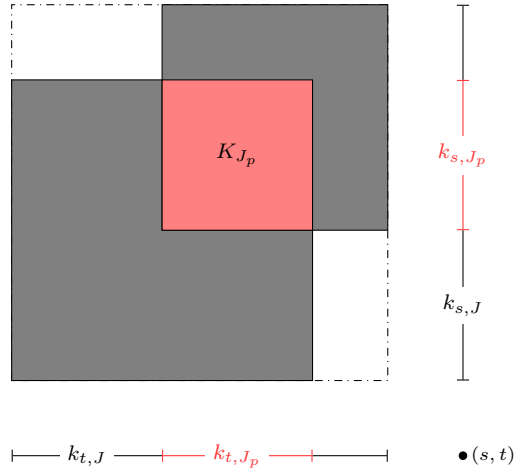
\begin{figure}[h]\label{fig:part_cov4}
		\centering
		\begin{center}
			\begin{tikzpicture}
				\draw [shift={(0.5,0.5)}, dashdotted] (0,0) -- (0,5) -- (5,5) -- (5,0) -- cycle;
				\draw [shift={(0.5,0.5)}, fill=gray] (0,0) -- (0,4) -- (4,4) -- (4,0) -- cycle;
				\draw [shift={(2.5,2.5)}, fill=gray] (0,0) -- (0,3) -- (3,3) -- (3,0) -- cycle;
				\draw [shift={(2.5,2.5)}, fill=red!50] (0,0) -- (0,2) -- (2,2) -- (2,0) -- cycle;

				\draw [|-,shift = {(0.5,-0.5)}] (0,0) -- (2,0) node [midway,fill=white] {$k_{t,J}$};
				\draw [-|,shift = {(0.5,-0.5)}] (4,0) -- (5,0);
				\draw [|-|,shift = {(0.5,-0.5)}, red!80] (2,0) -- (4,0) node [midway,fill=white] {$k_{t,J_{p}}$};
				
				\draw [-|,shift = {(6.5,-0.5)}] (0,5) -- (0,6);
				\draw [|-,shift = {(6.5,-0.5)}] (0,1) -- (0,3) node [midway,fill=white] {$k_{s,J}$};
				\draw [|-|,shift = {(6.5,-0.5)}, red!80] (0,3) -- (0,5) node [midway,fill=white] {$k_{s,J_{p}}$};
				
				\draw [shift={(0.5,-0.5)},fill] (6,0) circle [radius=0.05];
				\draw [shift={(0.5,-0.5)}] (6,0) node [right] {{\scriptsize $(s,t)$}};
				
				\draw [shift={(0.5,0.5)}] (3,3) node {$K_{J_{p}}$};
				
			\end{tikzpicture}
			\caption{The Partial Covariance $K_{J}$} 
		\end{center}
	\end{figure}
	Recall that $J_{p} = [a_{p+1}, b_{p}]$. Define $J_{-} = \{u \in J: u < v \mbox{ for some } v\in J_{p}\}$ and $J_{+} = \{u \in J: u > v \mbox{ for some } v\in J_{p}\}$. Thus, $J = J_{-} \cup J_{p} \cup J_{+}$. Notice that for $u \in J \setminus J_{+}$,
	\begin{equation*}
		\begin{split}
			\langle k^{\star}_{s, J} - \Pi_{p}k^{\star}_{s, J}, k^{\star}_{u, J} - \Pi_{p}k^{\star}_{u, J} \rangle 
			&= \langle k^{\star}_{s, J} - \Pi_{p}k^{\star}_{s, J}, k^{\star}_{u, J} \rangle - \langle k^{\star}_{s, J} - \Pi_{p}k^{\star}_{s, J}, \Pi_{p}k^{\star}_{u, J} \rangle \\
			&=  K_{\star}(s,u) - \langle \Pi_{p}k^{\star}_{s, J}, k^{\star}_{u, J} \rangle - 0\\
			&= \langle k^{\star}_{s, J_{p}}, k^{\star}_{u, J_{p}}\rangle - \langle \Pi_{p}k^{\star}_{s, J}, k^{\star}_{u, J} \rangle \\
			&= \langle \Pi_{p}k^{\star}_{s, J}, \Pi_{p}k^{\star}_{u, J}\rangle - \langle \Pi_{p}k^{\star}_{s, J}, k^{\star}_{u, J} \rangle \\
			&= 	\langle \Pi_{p}k^{\star}_{s, J}, \Pi_{p}k^{\star}_{u, J} - k^{\star}_{u, J}\rangle = 0\\
		\end{split}
	\end{equation*}
	Therefore, $k^{\star}_{s, J} - \Pi_{p}k^{\star}_{s, J}$ belongs to the closed subspace spanned by $\{ k^{\star}_{u, J} - \Pi_{p}k^{\star}_{u, J} : u \in J_{+}\}$. Similarly, it can be shown that $k^{\star}_{t, J} - \Pi_{p}k^{\star}_{t, J}$ belongs to the closed subspace spanned by $\{ k^{\star}_{u, J} - \Pi_{p}k^{\star}_{u, J} : u \in J_{-}\}$. If we are able to show that these subspaces themselves are mutually orthogonal, we would be done. Arguing as before, for $u \in J_{-}$ and $v \in J_{+}$,
	\begin{equation*}
		\begin{split}
			\langle k^{\star}_{u, J} - \Pi_{p}k^{\star}_{u, J}, k^{\star}_{v, J} - \Pi_{p}k^{\star}_{v, J} \rangle 
			&= \langle k^{\star}_{u, J} - \Pi_{p}k^{\star}_{u, J}, k^{\star}_{v, J} \rangle\\
			&=  K_{\star}(u,v) - \langle \Pi_{p}k^{\star}_{u, J}, k^{\star}_{v, J} \rangle\\
			&= \langle k^{\star}_{u, J_{p}}, k^{\star}_{v, J_{p}}\rangle - \langle \Pi_{p}k^{\star}_{u, J}, k^{\star}_{v, J} \rangle \\
			&= \langle \Pi_{p}k^{\star}_{u, J}, \Pi_{p}k^{\star}_{v, J}\rangle - \langle \Pi_{p}k^{\star}_{u, J}, k^{\star}_{v, J} \rangle \\
			&= 	\langle \Pi_{p}k^{\star}_{u, J}, \Pi_{p}k^{\star}_{v, J} - k^{\star}_{v, J}\rangle = 0\\
		\end{split}
	\end{equation*}
	The conclusion follows.
\end{proof}

Let $J \subset I$ separating $s,t \in I$ be such that $(s,t) \in \Omega^{c}$. Then for some $1 \leq p, q < m$, $J_{p} \subset I$ separates $s,t \in I$, since $(s,t) \in D_{p} \times S_{p}$ from some $1 \leq p < m$ and $J_{q} \subset J$ separates $s,t \in I$. By Lemma \ref{lemma:fund3} and Theorem \ref{serrated-theorem}, 
\begin{equation*}
	\begin{split}
		\langle k_{s, J}, k_{t, J} \rangle_{\mathscr{H}(K_{J})} &= \langle k_{s, J_{q}}, k_{t, J_{q}} \rangle_{\mathscr{H}(K_{J_{q}})} \\
		\langle k_{s, J_{q}}, k_{t, J_{q}} \rangle_{\mathscr{H}(K_{J_{q}})} &= \langle k_{s, J_{p}}, k_{t, J_{p}} \rangle_{\mathscr{H}(K_{J_{p}})}
	\end{split}
\end{equation*} 
And again by Theorem \ref{serrated-theorem}, $K_{\star}(s,t) = \langle k_{s, J_{p}}, k_{t, J_{p}} \rangle_{\mathscr{H}(K_{J_{p}})}$. We have thus shown the following:
\begin{theorem} \label{thm:fund}
	If $K_{\Omega}$ is a partial covariance on a serrated domain $
	\Omega$, then $K_{\Omega}$ has a unique covariance completion $K_{\star}$ to $I$ which possesses the separation property: for every $s,t \in I$ separated by $J \subset I$, 
	\begin{equation*}
		K_{\star}(s,t) = \langle k^{\star}_{s, J}, k^{\star}_{t, J} \rangle_{\mathscr{H}(K_{J})}
	\end{equation*}
	where $k^{\star}_{u, J}: J \to \mathbb{R}$ is given by $k^{\star}_{u, J}(v) = K_{\star}(u,v)$ for $v \in J$. Furthermore, $K_{\star}$ can be recursively computed using Algorithm \ref{algorithm}.
\end{theorem}

\subsection{Canonicity and Graphical models}
\begin{proof}[Proof of Theorem \ref{thm:canonical-graphical}]
	Simply use Theorem \ref{thm:fund} in conjunction with Theorem \ref{thm:canonical-graphical2}.
\end{proof}

\begin{proof}[Proof of Theorem \ref{thm:canonical-graphical2}]
	The process $X$ is said to form a graphical model with $([0,1], \Omega)$ precisely when for every $s,t \in I$ separated by $J \subset I$, we have
	\begin{equation*}
		\mathrm{Cov}(X_{s}, X_{t}| X_{J}) \equiv \mathbb{E}\left[ (X_{s} - \mathbb{E}\left[X_{s}|X_{J}\right]) (X_{t} - \mathbb{E}\left[X_{t}|X_{J}\right]) | X_{J}\right] = 0\quad \mbox{ a.s.}
	\end{equation*}
	which is equivalent to saying that $\mathbb{E}\left[ X_{s}X_{t}|X_{J} \right] = \mathbb{E}\left[X_{s}|X_{J}\right] \mathbb{E}\left[X_{t}|X_{J}\right]$ almost surely. According to \cite{loeve1963}, for a Gaussian process $X$, the conditional expectation $\mathbb{E}\left[X_{t}|X_{J}\right]$ is same as the projection $\Pi(X_{t}|X_{J})$ as described in Section \ref{sec:unique-completion}. Because the mean of the Gaussian process is zero, we can write the above equation as
	\begin{equation*}
		K(s,t) = \Pi(X_{s}|X_{J}) \Pi(X_{t}|X_{J})
	\end{equation*}
	By the Lo\`{e}ve isometry, 
	\begin{equation*}
		K(s,t) = \langle \Pi_{J}k_{s}, \Pi_{J}k_{t} \rangle
	\end{equation*}
	which reduces to
	\begin{equation*}
		K(s,t) = \langle k_{s,J}, \Pi_{J}k_{t,J} \rangle.
	\end{equation*}
	by the subspace isometry from Theorem \ref{subspace_isometry}. Thus, $K(s,t) = \langle K(s, \cdot), K(t, \cdot)\rangle_{\mathscr{H}(K_{J})}$ and the conclusion follows. 
\end{proof}

\subsection{Necessary and Sufficient Conditions for Unique Completion}
Naturally, we begin by dealing with the 2-serrated case. To this end we prove another existence result which captures how much the completion of a partial covariance on a 2-serrated domain can vary at a given point.
\begin{lemma} 
	Let $K_{\Omega}$ be a partial covariance on a serrated domain $\Omega$ of two intervals, $s \in I_{1}\setminus J_{1}$, $t \in I_{2} \setminus J_{1}$ and $\alpha \in \mathbb{R}$. There exists a covariance extension $K$ of $K_{\Omega}$ such that $$K(s,t) = \alpha + \langle k_{s, J_{1}} , k_{t, J_{1}} \rangle$$if any only if $$|\alpha| \leq \sqrt{K_{I_{1}}/K_{J_{1}}(s,s) \cdot K_{I_{2}}/K_{J_{1}}(t,t)}.$$
	\label{lem:secondex}
\end{lemma}

\begin{proof}
	We begin by getting rid of the part of the covariance that is due to $J_{1}$. Let $J_{-} = I_{1}\setminus J_{1}$, $J_{+} = I_{2}\setminus J_{1}$, $J^{c} = J_{-} \cup J_{+}$, $\Omega_{0} = [J_{-}  \times J_{-} ] \cup [J_{+} \times J_{+}]$ and define  $L_{\Omega_{0}}: \Omega_{0} \to \mathbb{R}$ as $L_{\Omega_{0}}(s,t) = K_{\Omega}(s,t) - \langle k_{s, J_{1}}, k_{t, J_{1}} \rangle$. This is similar to taking a Schur complement with respect to $J_{1}$. Strictly speaking, $L_{\Omega_{0}}$ is not a partial covariance, but it possesses the necessary structure of one and hence we can talk of its extension, which would be a covariance $L_{0}$ on $J^{c}$ such that $L_{0}|_{\Omega_{0}} = L_{\Omega_{0}}$.\\
	
	Notice that $K_{\Omega}$ has an extension if and only if $L_{\Omega_{0}}$ does. Indeed, if $K$ is an extension of $K_{\Omega}$ then $K/K_{J_{1}}$ is an extension of $L_{\Omega_{0}}$. Conversely, if $L_{0}$ is an extension of $L_{\Omega_{0}}$, then  $L: I \times I \to \mathbb{R}$ given by $L|_{J^{c} \times J^{c}} = L_{0}$ and $0$ otherwise, is a covariance, and so is $K: I \times I \to \mathbb{R}$ given by $K(s,t) = L(s,t) + \langle k_{s, J_{1}}, k_{t, J_{1}} \rangle$ for $s,t \in I$.  Thus, there is a clear one-one correspondence between the extensions $K$ of $K_{\Omega}$ and the ``extensions'' $L_{0}$ as defined above. Let $s \in I_{1}\setminus J_{1}$ and $t \in I_{2} \setminus J_{1}$.\\ 
	
	($\implies$) If $K$ is a covariance extension of $K_{\Omega}$, then $\alpha = K(s,t) - \langle k_{s, J_{1}} , k_{t, J_{1}} \rangle = L(s,t)$. Since, $L$ is a covariance, 
	\begin{equation*}
		|\alpha| = |L(s,t)| \leq \sqrt{L(s,s)\cdot L(t,t)} = \sqrt{K_{I_{1}}/K_{J_{1}}(s,s) \cdot K_{I_{2}}/K_{J_{1}}(t,t)}
	\end{equation*}
	
	($\impliedby$) Let $\alpha \in \mathbb{R}$ with the given property. It suffices to show that there is an extension $L_{0}$ of $L_{\Omega_{0}}$ such that $L_{0}(s,t) = \alpha$. To do this, we shall rearrange the points of $J_{-} \cup J_{+}$ so that the region over which $L_{0}$ is known, resembles a serrated domain.
	
	Let $I_{-} = J_{-} \setminus \{u\}$ and $I_{+} = J_{+} \setminus \{v\}$. Now consider the sets $J_{-} = I_{-} \cup \{u\}$, $\{u,v\}$ and $\{v\} \cup I_{+} = J_{+}$. These exhibit an overlapping pattern resembling that of the intervals of a serrated domain, although they are not intervals themselves -- see Figure \ref{fig:one-point}. In light of Remark \ref{any-2-sets}, by applying the completion procedure in Equation \ref{2-serrated-completion} twice or equivalently using Algorithm \ref{algorithm} we can complete the partial covariance on this ``serrated-type-domain". This permits us to conclude that there exists a covariance $L$ such that $L(u,v) = \alpha$ for every $\alpha$ as described above and the conclusion for $K$ follows from the correspondence between the two.
	
	\begin{figure}[h]
		\centering
		\begin{center}
			\begin{tikzpicture}
				\draw [shift={(0.5,0.5)}, fill = red!50, dotted] (0,0) -- (0,2) -- (2,2) -- (2,0) -- cycle;
				\draw [shift={(0.5,0.5)}, fill = red!50, dotted] (4,4) -- (4,5.5) -- (5.5,5.5) -- (5.5,4) -- cycle;
				\draw [shift={(0.5,0.5)}, dotted] (4,0) -- (5.5,0) -- (5.5,2) -- (4,2) -- cycle;
				\draw [shift={(0.5,0.5)}, dotted] (0,4) -- (0,5.5) -- (2,5.5) -- (2,4) -- cycle;
				\draw [shift={(0.5,0.5)}, dotted] (4.5,0) -- (4.5,2);
				\draw [shift={(0.5,0.5)}, dotted] (0,4.5) -- (2,4.5);
				\draw [shift={(0.5,0.5)}, dotted] (4,1) -- (5.5,1);
				\draw [shift={(0.5,0.5)}, dotted] (1,4) -- (1,5.5);
				\filldraw [shift={(0.5,0.5)}, fill=red!50] (1, 1) circle (2pt);
				\filldraw [shift={(0.5,0.5)}, fill=red!50] (4.5, 4.5) circle (2pt);			
				\filldraw [shift={(0.5,0.5)}] (4.5, 1) circle (2pt);
				\filldraw [shift={(0.5,0.5)}] (1, 4.5) circle (2pt);
				
				\draw [shift={(0.5,0.5)}] (0,2) -- (0,0) -- (2,0); 
				\draw [shift={(0.5,0.5)}] (4,5.5) -- (5.5,5.5) -- (5.5,4);			
				\draw [shift={(0.5,0.5)}] (4,0) -- (5.5,0) -- (5.5,2);
				\draw [shift={(0.5,0.5)}] (0,4) -- (0,5.5) -- (2,5.5);
				
				\draw [shift={(0.5,0.5)}] (6, 2.5) node {$\longrightarrow$};
				
				\draw [dotted, shift={(7.5,0.5)}, fill = red!50] (0,0) -- (0,2) -- (2,2) -- (2,0) -- cycle;
				\draw [dotted, shift={(7.5,0.5)}, fill = red!50] (4,4) -- (4,5.5) -- (5.5,5.5) -- (5.5,4) -- cycle;
				\draw [dotted, shift={(7.5,0.5)}] (4,0) -- (4,2) -- (5.5,2) -- (5.5,0) -- cycle;
				\draw [dotted, shift={(7.5,0.5)}] (0,4) -- (2,4) -- (2,5.5) -- (0,5.5) -- cycle;
				\draw [shift={(7.5,0.5)}] (0,2) -- (0,0) -- (2,0); 
				\draw [shift={(7.5,0.5)}] (4,5.5) -- (5.5,5.5) -- (5.5,4);			
				\draw [shift={(7.5,0.5)}] (4,0) -- (5.5,0) -- (5.5,2);
				\draw [shift={(7.5,0.5)}] (0,4) -- (0,5.5) -- (2,5.5);
				
				\filldraw [shift={(7.5,0.5)}, fill=red!50] (2.5, 2.5) circle (2pt);
				\filldraw [shift={(7.5,0.5)}, fill=red!50] (3.5, 3.5) circle (2pt);
				\filldraw [shift={(7.5,0.5)}] (3.5, 2.5) circle (2pt);			
				\filldraw [shift={(7.5,0.5)}] (2.5, 3.5) circle (2pt);
				
				\draw [shift={(7.5,0.5)}, thick, red!50] (2.5,0) -- (2.5,2);
				\draw [shift={(7.5,0.5)}, thick, red!50] (0,2.5) -- (2,2.5);
				\draw [shift={(7.5,0.5)}, thick, red!50] (3.5,4) -- (3.5,5.5);
				\draw [shift={(7.5,0.5)}, thick, red!50] (4,3.5) -- (5.5,3.5);
				
				\draw [shift={(7.5,0.5)}, dashed] (-0.2,-0.2) -- (2.7, -0.2) -- (2.7, 2.3) -- (3.7, 2.3) -- (3.7, 3.3) -- (5.7, 3.3) -- (5.7, 5.7) -- (3.3,5.7) -- (3.3,3.7) -- (2.3, 3.7) -- (2.3, 2.7) -- (-0.2, 2.7) -- cycle;
				
				\draw [shift={(7.5,0.5)}, dotted] (3.5,0) -- (3.5,2);
				\draw [shift={(7.5,0.5)}, dotted] (0,3.5) -- (2,3.5);
				\draw [shift={(7.5,0.5)}, dotted] (4,2.5) -- (5.5, 2.5);
				\draw [shift={(7.5,0.5)}, dotted] (2.5,4) -- (2.5,5.5);

				\draw [shift={(7.5,0.5)}, thick, white] (1, 0) -- (1,2);	
				\draw [shift={(7.5,0.5)}, thick, white] (0, 1) -- (2,1);
				\draw [shift={(7.5,0.5)}, thick, white] (4.5,4) -- (4.5,5.5);	
				\draw [shift={(7.5,0.5)}, thick, white] (4,4.5) -- (5.5,4.5);	
				
				
				\draw [shift={(7.5,0.5)}] (3.5,2.4) node [below right] {{\scriptsize $(t,s)$}};
				\draw [shift={(7.5,0.5)}] (2.4,3.5) node [above left] {{\scriptsize $(s,t)$}};
				\draw [shift={(7.5,0.5)}] (3.5,3.5) node [above right] {{\scriptsize $(t,t)$}};
				\draw [shift={(7.5,0.5)}] (2.5,2.5) node [below left] {{\scriptsize $(s,s)$}};
				
				\draw [shift={(0.5,0.5)}] (0.5,0.5) node {$K_{J_{-}}$};	
				\draw [shift={(0.5,0.5)}] (5,5) node {$K_{J_{+}}$};
				
				\draw [shift={(7.5,0.5)}] (0.5,0.5) node {$K_{I_{-}}$};	
				\draw [shift={(7.5,0.5)}] (5,5) node {$K_{I_{+}}$};		
			\end{tikzpicture}
			\caption{Rearrangement of $J^{c} = J_{-} \cup J_{+}$} 
			\label{fig:one-point}
		\end{center}		
	\end{figure}
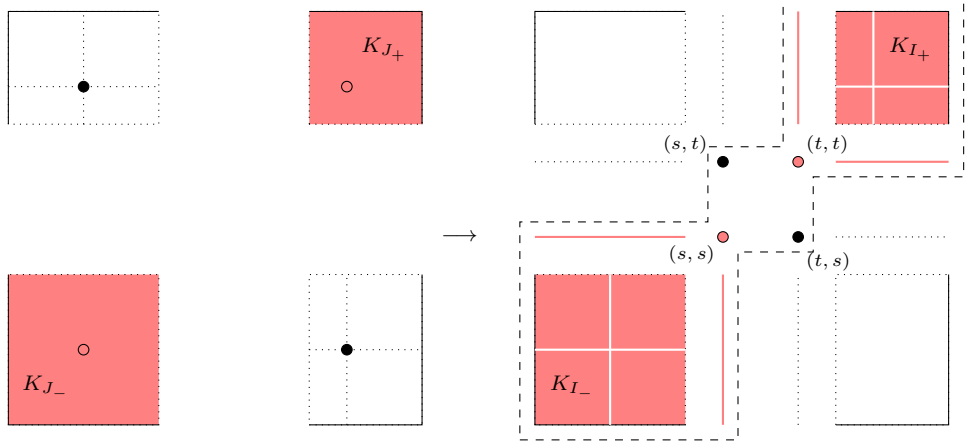
\end{proof}
The following corollary is immediate.
\begin{corollary}
	\label{unique-2-case}
	Let $K_{\Omega}$ be a partial covariance on a 2-serrated domain $\Omega$. Then $K_{\Omega}$ has a unique extension if and only if $K_{I_{1}} / K_{J_{1}} = 0$ or $K_{I_{2}} / K_{J_{1}} = 0$.
\end{corollary}
To extend this result to all serrated domains by induction we need to understand the effect uniqueness has on Schur complements. 	
\begin{lemma}\label{lem:unique_conseq}
	Let $K_{\Omega}$ be a partial covariance on a 2-serrated domain. If $K_{\Omega}$ has a unique extension, then $K_{I_{1}}/K_{J_{1}} = K_{\star}/K_{I_{2}}$ and $K_{I_{2}}/K_{J_{1}} = K_{\star}/K_{I_{1}}$.
\end{lemma}
\begin{proof} 
	For $s,t \in I_{1} \setminus J_{1} = I \setminus I_{2}$,
	\begin{equation*}
		\begin{split}
			K_{I_{1}}/K_{J_{1}}(s,t) - K_{\star}/K_{I_{2}}(s,t)
			&=  K_{I_{1}}(s,t) - \langle k_{s, J_{1}}, k_{t, J_{1}}\rangle - K_{\star}(s,t) + \langle k_{s, I_{2}}, k_{t, I_{2}}\rangle\\
			&= \langle k_{s, I_{2}}, k_{t, I_{2}}\rangle - \langle k_{s, J_{1}}, k_{t, J_{1}}\rangle\\
			&= \langle k_{s, I_{2}}, k_{t, I_{2}}\rangle - \langle k_{s, I_{2}}, \Pi_{J_{1}} k_{t, I_{2}}\rangle\\
			&= \langle k_{s, I_{2}}, k_{t, I_{2}} - \Pi_{J_{1}} k_{t, I_{2}}\rangle
		\end{split}
	\end{equation*}
	where $\Pi_{J_{1}}: \mathscr{H}(K_{I_{2}}) \to \mathscr{H}(K_{I_{2}})$	denotes the projection to the closed subspace spanned by $\{k_{u, I_{2}}: u \in J_{1}\}$. To see why this term vanishes we reason as follows.

	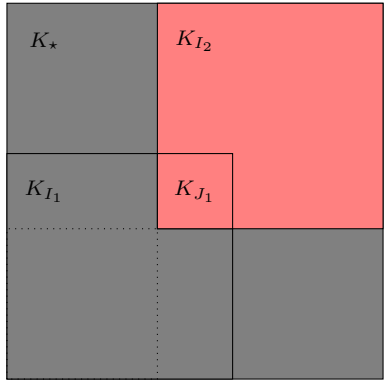
\begin{figure}[h]
		\centering
		\begin{center}
			\begin{tikzpicture}							
				\draw [shift={(0.5,0.5)}, fill=gray] (0,0) -- (0,5) -- (5,5) -- (5,0) -- cycle;
				\draw [shift={(0.5,0.5)}, fill=gray] (0,0) -- (0,3) -- (3,3) -- (3,0) -- cycle;
				\draw [shift={(0.5,0.5)}, fill=red!50] (2,2) -- (2,5) -- (5,5) -- (5,2) -- cycle;
				\draw [shift={(0.5,0.5)}, fill = red!50] (2,2) -- (2,3) -- (3,3) -- (3,2) -- cycle;
				
				\draw [shift={(0.5,0.5)}, dotted] (0,0) -- (0,2) -- (2,2) -- (2,0) -- cycle;
				
				\draw [shift={(0.5,0.5)}](2.5, 2.5) node [fill = red!50] {$K_{J_{1}}$};
				\draw [shift={(0.5,0.5)}](0.5, 2.5) node [fill = gray] {$K_{I_{1}}$};
				\draw [shift={(0.5,0.5)}](2.5, 4.5) node [fill = red!50] {$K_{I_{2}}$};
				\draw [shift={(0.5,0.5)}](0.5, 4.5) node {$K_{\star}$};
			\end{tikzpicture}
			\caption{The covariances $K_{J_{1}}$, $K_{I_{1}}$, $K_{I_{2}}$ and $K_{\star}$} 
		\end{center}
	\end{figure}
	Observe that for $u \in I_{2}$, 
	\begin{equation*}
		\begin{split}
			\langle k_{u, I_{2}}, k_{t, I_{2}} - \Pi_{J_{1}} k_{t, I_{2}}\rangle 
			&= \langle k_{u, I_{2}}, k_{t, I_{2}}\rangle - \langle k_{u, I_{2}}, \Pi_{J_{1}} k_{t, I_{2}}\rangle \\
			&= k_{t, I_{2}}(u) - \langle \Pi_{J_{1}} k_{u, I_{2}}, \Pi_{J_{1}} k_{t, I_{2}}\rangle\\
			&= K_{\star}(u,t) - \langle k_{u, J_{1}}, k_{t, J_{1}}\rangle\\
			&= 0.
		\end{split}
	\end{equation*}
	
	Therefore, $k_{t, I_{2}} - \Pi_{J_{1}} k_{t, I_{2}} = 0$ and the conclusion follows. Similarly, we can show that $K_{I_{2}}/K_{J_{1}} = K_{\star}/K_{I_{1}}$.
\end{proof}
In other words, uniqueness causes certain Schur complements to reduce to ``smaller" Schur complements. We are now ready to prove Theorem \ref{uniqueness}.\\ 

\begin{figure}[h]
	\centering
	\begin{center}
		\begin{tikzpicture}
			\draw [shift={(0.5,0.5)}, fill=gray!50] (0,0) -- (0,6) -- (1,6) -- (1,7) -- (7,7) -- (7,1) -- (6,1) -- (6,0) -- (0,0) -- cycle;
			\draw [shift={(1.5,1.5)}, dotted, fill=red!60] (0,0) -- (0,5) -- (5,5) -- (5,0) -- cycle;
			
			\draw [shift={(0.5,0.5)}] (0,0) -- (0,7) -- (7,7) -- (7,0) -- cycle;
			\draw [shift={(4.5,4.5)}, fill=gray!70] (0,0) -- (0,3) -- (3,3) -- (3,0) -- cycle;
			\draw [shift={(3.5,3.5)}, fill=gray!70] (0,0) -- (0,3) -- (3,3) -- (3,0) -- cycle;			
			\draw [shift={(0.5,0.5)}, fill=gray!70] (0,0) -- (0,3) -- (3,3) -- (3,0) -- cycle;
			
			\draw [shift={(2.5,2.5)}, fill=gray!70] (0,0) -- (0,3) -- (3,3) -- (3,0) -- cycle;
			\draw [shift={(1.5,1.5)}, fill=gray!70] (0,0) -- (0,3) -- (3,3) -- (3,0) -- cycle;
			\draw [shift={(1.5,1.5)}, fill=red!60, dotted] (0,0) -- (0,2) -- (2,2) -- (2,0) -- cycle;			
			\draw [shift={(0.5,0.5)}, dotted] (0,0) -- (0,1) -- (1,1) -- (1,0) -- cycle;
			\draw [shift={(6.5,6.5)}, dotted] (0,0) -- (0,1) -- (1,1) -- (1,0) -- cycle;
			
			\draw [shift={(0.5, 0.5)}] (0.5, 2.5) node {$K_{I_{1}}$};
			\draw [shift={(1.5, 1.5)}] (0.5, 2.5) node {$K_{I_{2}}$};
			\draw [shift={(1.5, 0.5)}] (0.5, 2.5) node {$K_{J_{1}}$};
			\draw [shift={(3.5, 3.5)}] (0.5, 2.5) node {$K_{I_{m}}$};
			\draw [shift={(4.5, 4.5)}] (0.5, 2.5) node {$K_{I_{m+1}}$};
			
			\draw [shift={(0.5, 0.5)}] (0.5, 5.5) node {$K_{\bar{I}_{1}}$};
			\draw [shift={(1.5, 1.5)}] (0.5, 5.5) node {$K_{\bar{I}_{2}}$};
			\draw [shift={(1.5, 0.5)}] (0.5, 5.5) node {$K_{\bar{I}_{1} \cap \bar{I}_{2}}$};
			\draw [shift={(0.5, 0.5)}] (0.5, 6.5) node {$K_{\star}$};			
		\end{tikzpicture}
		\caption{The covariance extensions $K_{\bar{I}_{1}}$, $K_{\bar{I}_{2}}$, $K_{\bar{I}_{1} \cap \bar{I}_{2}}$ and $K_{\star}$} 
	\end{center}
\end{figure}
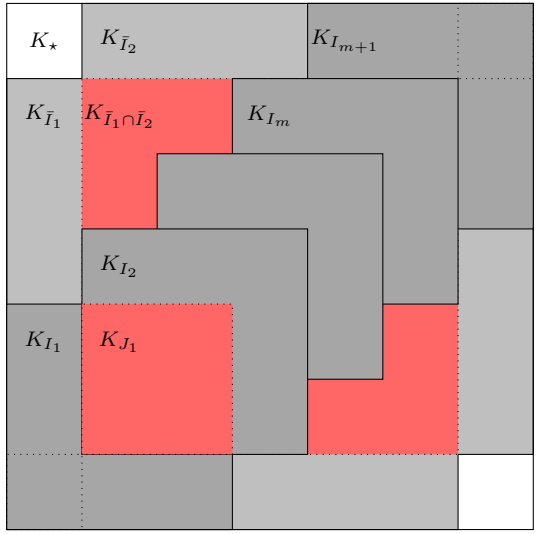
\begin{proof}[Proof of Theorem \ref{uniqueness}] 
	We shall use induction on $m$. The base case $m=2$ follows from Corollary \ref{unique-2-case}. Assume that the result holds for some $m \geq 2$ and consider a partial covariance $K_{\Omega}$ on an $(m+1)$-serrated domain $\Omega$. Let 
	\begin{align*}
		\bar{I}_{1} &= \cup_{j=1}^{m} I_{j}, 
		& \Omega_{1} &= \cup_{j=1}^{m}I_{j} \times I_{j} \subset \bar{I}_{1}\times \bar{I}_{1}, 
		& K_{\Omega_{1}} &= K_{\Omega}|_{\Omega_{1}},\\
		\bar{I}_{2} &= \cup_{j=2}^{m+1} I_{j},
		& \Omega_{2} &= \cup_{j=2}^{m+1}I_{j} \times I_{j} \subset \bar{I}_{2}\times \bar{I}_{2},
		& K_{\Omega_{2}} &= K_{\Omega}|_{\Omega_{2}},
	\end{align*}
	and $\bar{\Omega} = (\bar{I}_{1} \times \bar{I}_{1}) \cup (\bar{I}_{2} \times \bar{I}_{2})$.
	
	If $K_{\Omega}$ admits a unique extension then so do the partial covariances $K_{\Omega_{1}}$ and $K_{\Omega_{2}}$, for otherwise using Theorem \ref{2-serrated-theorem}, one can complete two distinct completions of $K_{\Omega_{1}}$ and $K_{\Omega_{2}}$ to get two distinct completions of $K_{\Omega}$. By the induction hypothesis, there exist $r_{1} \in \{ 1, \dots, m \}$ and $r_{2} \in \{2, \dots, m+1\}$ such that $K_{I_{p}}/K_{J_{p}} = 0$ for $1 \leq p < r_{1} \vee r_{2}$ and $K_{I_{q+1}}/K_{J_{q}} = 0$ for $r_{1} \wedge r_{2} \leq q < m+1$. Pick any $r$ such that $r_{1} \wedge r_{2} \leq r \leq r_{1} \vee r_{2}$. Then $K_{I_{p}}/K_{J_{p}} = 0$ for $1 \leq p < r$ and $K_{I_{q+1}}/K_{J_{q}} = 0$ for $r \leq q < m+1$.\\ 
	
	To show the converse, assume that $K_{I_{p}}/K_{J_{p}} = 0$ for $1 \leq p < r$ and $K_{I_{q+1}}/K_{J_{q}} = 0$ for $r \leq q < m+1$ for some $r \in \{1, \dots, m+1\}$. Then $K_{I_{p}}/K_{J_{p}} = 0$ for $1 \leq p < r \wedge m$ and $K_{I_{q+1}}/K_{J_{q}} = 0$ for $r \wedge m \leq q < m$ and $K_{I_{p}}/K_{J_{p}} = 0$ for $2 \leq p < r \vee 2$, and that $K_{I_{q+1}}/K_{J_{q}} = 0$ for $r \vee 2 \leq q < m+1$. By the induction hypothesis, it follows that $K_{\Omega_{1}}$ and $K_{\Omega_{2}}$ both admit unique completions, say $K_{\bar{I}_{1}}$ and $K_{\bar{I}_{2}}$ respectively. 
	
	Due to uniqueness, $K_{\bar{I}_{1}}(s,t) = K_{\bar{I}_{2}}(s,t)$ for $s,t \in \bar{I}_{1} \cap \bar{I}_{2}$, so together they form a partial covariance $K_{\bar{\Omega}}$ on $\bar{\Omega}$ given by $K_{\Omega^{\prime}}(s,t) = K_{\bar{I}_{1}}(s,t)$ if $(s,t) \in \Omega_{1}$ and $K_{\bar{I}_{2}}(s,t)$ if $(s,t) \in \Omega_{2}$. To prove that $K_{\Omega}$ admits a unique completion, it suffices to show that $K_{\bar{\Omega}}$ admits a unique completion. Since $\bar{\Omega}$ is a $2$-serrated domain, we can use the base case and this reduces to showing that $K_{\bar{I}_{1}}/K_{\bar{I}_{1} \cap \bar{I}_{2}} = 0$ or $K_{\bar{I}_{2}}/K_{\bar{I}_{1} \cap \bar{I}_{2}} = 0$. By applying Lemma \ref{lem:unique_conseq} to the 2-serrated domains 
	\begin{equation*}
		(I_{1} \times I_{1}) \cup [(\bar{I}_{1} \cap \bar{I}_{2}) \times (\bar{I}_{1} \cap \bar{I}_{2})] \mbox{ and } [(\bar{I}_{1} \cap \bar{I}_{2}) \times (\bar{I}_{1} \cap \bar{I}_{2})] \cup (I_{m+1} \times I_{m+1})
	\end{equation*}
	we get that $K_{\bar{I}_{1}}/K_{\bar{I}_{1} \cap \bar{I}_{2}} = K_{I_{1}}/K_{J_{1}}$ and $K_{\bar{I}_{2}}/K_{\bar{I}_{1} \cap \bar{I}_{2}} = K_{I_{m+1}}/K_{J_{m}}$, at least one of which has to be zero by our assumption. The conclusion follows.
\end{proof}

\subsection{Characterisation of All Completions}

\begin{proof}[Proof of Theorem \ref{2serrated-characteriation}]
	Let $K$ be a completion of $K_{\Omega}$ and $\Pi_{J}: \mathscr{H}(K) \to \mathscr{H}(K)$ denote the projection to the closed subspace spanned by $\{k_{u}: u \in J\}$. Then 
	\begin{align*}
		L(s,t) 
		= K(s,t) - \big\langle K_{\Omega}(s,\cdot), K_{\Omega}(\cdot, t) \big\rangle_{\mathscr{H}(K_{I_1\cap I_2})}
		&= \langle k_{s}, k_{t} \rangle - \langle k_{s,J}, k_{t,J} \rangle\\
		&= \langle k_{s}, k_{t} \rangle - \langle \Pi_{J}k_{s}, \Pi_{J}k_{t} \rangle\\
		&= \langle k_{s} - \Pi_{J}k_{s}, k_{t} - \Pi_{J}k_{t} \rangle
	\end{align*}
	is a completion of $L_{\Omega'}$. The converse is obvious. 
\end{proof}

\begin{remark}
	\label{L_is_C}
	Notice that if $K_{\Omega}$ is continuous, then so is $K$ and so is the term $\big\langle K_{\Omega}(s,\cdot), K_{\Omega}(\cdot, t) \big\rangle_{\mathscr{H}(K_{I_1\cap I_2})}$ as a function of $s$ and $t$. This is because the mapping $t \to k_{t}$ is continuous since
	\begin{equation*}
		\left\| k_{t+h} - k_{t} \right\|^{2} = K(t+h, t+h) -2K(t,h) + K(t,t) \to 0 
	\end{equation*}
	as $h \to 0$ and the same would apply to the mapping $t \to \Pi_{J}k_{t}$. It follows that $L$ is continuous.
\end{remark}

\begin{proof}[Proof of Lemma \ref{lemma:operK}]
	For $f \in L^{2}(I)$ we have
	\begin{equation*}
		\begin{split}
			\langle f, \mathbf{K}f \rangle_{2} &= \sum_{j=1}^{m} \langle f|_{I_{j}},\mathbf{K}_{j} f|_{I_{j}} \rangle_{2}  
			+ \sum_{p=1}^{m-1} \left[ 2\langle f|_{S_{p}},\mathbf{R}_{p} f|_{D_{p}} \rangle_{2} - \langle f|_{J_{p}},\mathbf{J}_{p} f|_{J_{p}} \rangle_{2} \right]. \\
		\end{split}
	\end{equation*}
	Let $g \in L^{2}(I)$. Then we can write
	\begin{equation*}
		\begin{split}
			\langle g, \mathbf{K}f \rangle_{2} 
			&= \tfrac{1}{4}\left[ \langle f+g, \mathbf{K}(f+g) \rangle_{2} - \langle f-g, \mathbf{K}(f-g) \rangle_{2} \right] \\
			&= \sum_{j=1}^{m} \langle g|_{I_{j}},\mathbf{K}_{j} f|_{I_{j}} \rangle_{2}  
			+ \sum_{p=1}^{m-1} \left[\langle g|_{S_{p}},\mathbf{R}_{p} f|_{D_{p}} \rangle_{2} + \langle \mathbf{R}_{p}g|_{D_{p}} ,f|_{S_{p}} \rangle_{2} - \langle g|_{J_{p}},\mathbf{J}_{p} f|_{J_{p}} \rangle_{2}\right] \\
		\end{split}
	\end{equation*}
	Thus,
	\begin{equation*}
		\mathbf{K}f(t) = \sum_{j: t \in I_{j}} \mathbf{K}_{j} f|_{I_{j}}(t)
		+ \sum_{p: t \in S_{p}} \mathbf{R}_{p} f|_{D_{p}}(t) 
		+ \sum_{p: t \in D_{p}} \mathbf{R}_{p}^{\ast}f|_{S_{p}}(t)
		- \sum_{p: t \in J_{p}} \mathbf{J}_{p} f|_{J_{p}}(t) \mathrm{~a.e.}
	\end{equation*}
\end{proof}

\begin{proof}[Proof of Theorem \ref{general-characterisation}]
	We use induction. Consider the base case $m=2$. Using Theorem \ref{2serrated-characteriation}, we know that the integral kernel $K_{R_{1}}$ of $\mathbf{R_{1}}$ at some point $(s,t)$ is given by the contribution due to the canonical completion which is $\langle k_{s,J_{1}}, k_{t, J_{1}}\rangle$ plus the perturbation. By \cite[Theorem 11.18]{paulsen2016}, we can write the first term as
	\begin{equation*}
		\langle k^{\Omega}_{s, J_{1}}, k^{\star}_{t, J_{1}} \rangle
		= \langle \mathbf{J}_{1}^{-1/2}k^{\Omega}_{s, J_{1}}, \mathbf{J}_{1}^{-1/2}k^{\star}_{t, J_{1}} \rangle_{L^{2}(J_{1})} \\
	\end{equation*}
	and therefore corresponding integral operator is given by
	\begin{equation}\label{eq:kernel_to_operator}
		\left[ \mathbf{J}_{1}^{-1/2} \mathbf{S}^{\ast}_{1} \right]^{\ast}
		\left[\mathbf{J}_{1}^{-1/2} \mathbf{D}_{1}\right]. 
	\end{equation}
	Due to \cite[Theorem 2]{baker1973} as mentioned before, the second term has to be of the form 
	\begin{equation*}
		\mathbf{U}_{1}^{1/2} \Psi_{1} \mathbf{V}_{1}^{1/2}
	\end{equation*}
	for some bounded linear map $\Psi_{1}: L^{2}(D_{1}) \to L^{2}(S_{1})$ with $\left\| \Psi_{1} \right\| \leq 1$ where
	\begin{equation}
		\mathbf{U}_{1} = \mathbf{K}_{I_{1}} -  \left[\mathbf{J}_{1}^{-1/2}\mathbf{S}_{1}^{\ast}\right]^{\ast}\left[\mathbf{J}_{1}^{-1/2}\mathbf{S}_{1}^{\ast}\right],\qquad
		\mathbf{V}_{1} = \mathbf{K}_{I_{2}} - \left[\mathbf{J}_{1}^{-1/2}\mathbf{D}_{1}^{\ast}\right]^{\ast}\left[\mathbf{J}_{1}^{-1/2}\mathbf{D}_{1}^{\ast}\right]
	\end{equation}
	are simply integral operators corresponding to the Schur complements $K_{I_{1}}/K_{J_{1}}$ and $K_{I_{2}}/K_{J_{1}}$ found using the technique  in Equation (\ref{eq:kernel_to_operator}).
	
	In the base case $m=2$, we have from Theorem \ref{2serrated-characteriation} that
	\begin{equation}\label{eq:1ara}
		\mathbf{R}_{1} 
		= \left[ \mathbf{J}_{1}^{-1/2} \mathbf{S}^{\ast}_{1} \right]^{\ast}
		\left[\mathbf{J}_{1}^{-1/2} \mathbf{D}_{1}\right] 
		+ \mathbf{U}_{1}^{1/2} \Psi_{1} \mathbf{V}_{1}^{1/2}
	\end{equation}

	Now for the induction case, assume that $K$ is known over the region $(\cup_{j=1}^{p} I_{j}) \times (\cup_{j=1}^{p} I_{j})$. Consider the 2-serrated domain given by $$\left[(\cup_{j=1}^{p} I_{j}) \times (\cup_{j=1}^{p} I_{j})\right] \cup (I_{p+1} \times I_{p+1}).$$ Then $(\cup_{j=1}^{p} I_{j}) \cap I_{p+1} = J_{p}$, $(\cup_{j=1}^{p} I_{j}) \setminus I_{p+1} = S_{p}$ and $I_{p+1} \setminus (\cup_{j=1}^{p} I_{j})  = D_{p}$. Repeating the above reasoning gives
	\begin{equation*}
		\mathbf{R}_{p} = \left[ \mathbf{J}_{p}^{-1/2} \mathbf{S}^{\ast}_{p} \right]^{\ast}\left[\mathbf{J}_{p}^{-1/2} \mathbf{D}_{p}\right] + \mathbf{U}_{p}^{1/2} \Psi_{p} \mathbf{V}_{p}^{1/2}
	\end{equation*}
	for some bounded linear map $\Psi_{p}: L^{2}(D_{p}) \to L^{2}(S_{p})$ with $\left\| \Psi_{p} \right\| \leq 1$ where
	$$
	\mathbf{U}_{p} = \mathbf{K}_{S_{p}} -  \left[\mathbf{J}_{p}^{-1/2}\mathbf{S}_{p}^{\ast}\right]^{\ast}\left[\mathbf{J}_{p}^{-1/2}\mathbf{S}_{p}^{\ast}\right],\qquad
	\mathbf{V}_{p} = \mathbf{K}_{D_{p}} - \left[\mathbf{J}_{p}^{-1/2}\mathbf{D}_{p}^{\ast}\right]^{\ast}\left[\mathbf{J}_{p}^{-1/2}\mathbf{D}_{p}^{\ast}\right]
	$$
	The proof is this complete.
\end{proof}

\subsection{Estimation of the Canonical Completion} Solving Equation (\ref{eq:system_linear}) involves an interesting complication. Since $K_{\Omega}$ is specified inexactly, both the \textit{operator}, which is $\mathbf{J}^{1/2}_{p}$, as well as the \textit{data}, in the form of the operators $\mathbf{D}_{p}$ and $\mathbf{S}_{p}$, are inexactly specified.
In essence, the problem is to estimate an operator
\begin{equation*}
	\mathbf{W} = \sum_{j=1}^{\infty} \frac{\mathbf{U} e_{j} \otimes \mathbf{V} e_{j}}{\lambda_{j}}
\end{equation*}
where $\{(\lambda_{j}, e_{j})\}_{j=1}^{\infty}$ are the eigenpairs of $\mathbf{T}$, from the estimates $\hat{\mathbf{U}}$, $\hat{\mathbf{V}}$ and $\hat{\mathbf{T}}$ which converge to the operators $\mathbf{U}$, $\mathbf{V}$ and $\mathbf{T}$ almost surely or in $L^{2}$. A natural candidate is the estimator
\begin{equation*}
	\hat{\mathbf{W}} = \sum_{j=1}^{N} \frac{\hat{\mathbf{U}} \hat{e}_{j} \otimes \hat{\mathbf{V}} \hat{e}_{j}}{\hat{\lambda}_{j}}
\end{equation*}
where $N$ serves as the truncation or regularization parameter. We shall work out how fast $N$ can grow as the estimates $\hat{\mathbf{U}}$, $\hat{\mathbf{V}}$ of $\hat{\mathbf{T}}$ converge to $\mathbf{U}$, $\mathbf{V}$ and $\mathbf{T}$ for $\hat{\mathbf{W}}$ to converge to $\mathbf{W}$. We have the following estimate:
\begin{lemma}\label{lem:estimate}
	Let $\alpha_{j}$ be the eigenvalue gap given by
	\begin{equation*}
		\alpha_{j} = 
		\begin{cases}
			(\lambda_{1} - \lambda_{2}) / 2\sqrt{2} &j=1 \\
			\Big[(\lambda_{j-1} - \lambda_{j}) \wedge (\lambda_{j} - \lambda_{j+1})\Big] / 2\sqrt{2} &j>1.
		\end{cases}
	\end{equation*}
	If $\alpha_{j}$ is monotonically decreasing with $j$, then for every $N$ satisfying $\lambda_{N} > \| \mathbf{T} - \hat{\mathbf{T}} \|_{2}$, we have the bound
	\begin{equation}\label{eq:estimate}
		\begin{split}
			\| \mathbf{W} - \hat{\mathbf{W}} \|_{2}^{2} 
			&\preceq \frac{N}{\lambda_{N}^{2}} \left[\| \mathbf{U} - \hat{\mathbf{U}}\|_{2}^{2} + \| \mathbf{V} - \hat{\mathbf{V}}\|_{2}^{2}\right]\\
			&+ \frac{N}{\lambda_{N}^{2}\alpha_{N}^{2}} \| \mathbf{T} - \hat{\mathbf{T}} \|_{2}^{2} + \left\| \sum_{j=N+1}^{\infty} \frac{\mathbf{U}e_{j} \otimes \mathbf{V}e_{j}}{\lambda_{j}} \right\|_{2}^{2}. 
		\end{split}
	\end{equation}	
	Moreover, if $\alpha_{j}$ is not monotonically decreasing with $j$, the above bound still holds if we replace $\alpha_{N}$ with $\min_{j \leq N} \alpha_{j}$,
\end{lemma}
\begin{proof}
	We take a step wise approach. Define
	\begin{equation*}
		\hat{\mathbf{W}}_{1} = \sum_{j=1}^{N} \frac{\hat{\mathbf{U}}e_{j} \otimes \hat{\mathbf{V}}e_{j}}{\lambda_{j}}, 
		\hat{\mathbf{W}}_{2} = \sum_{j=1}^{N} \frac{\hat{\mathbf{U}}\hat{e}_{j} \otimes \hat{\mathbf{V}}\hat{e}_{j}}{\lambda_{j}} \mbox{ and }
		\hat{\mathbf{W}}_{3} = \sum_{j=1}^{N} \frac{\hat{\mathbf{U}}\hat{e}_{j} \otimes \hat{\mathbf{V}}\hat{e}_{j}}{\hat{\lambda}_{j}}.
	\end{equation*}
	Naturally,
	\begin{equation*}
		\| \mathbf{W} - \hat{\mathbf{W}}\|_{2}^{2} \leq 2\| \mathbf{W} - \hat{\mathbf{W}}_{1}\|_{2}^{2} + 2\|\hat{\mathbf{W}}_{1} - \hat{\mathbf{W}}_{2}\|_{2}^{2} + 2\|\hat{\mathbf{W}}_{2} - \hat{\mathbf{W}}_{3}\|_{2}^{2}
	\end{equation*} 
	We now proceed by working out an upper bound for every term individually:\\
	
	\textbf{Step 1.} Using the identity $\| x + y \|^{2} \leq 2\|x \|^{2} + 2 \|y \|^{2}$ we can write the first term as follows
	\begin{align*}
		\|\mathbf{W} - \hat{\mathbf{W}}_{1}\|_{2}^{2} 
		&\leq 2\left\| \sum_{j=1}^{N} \frac{\mathbf{U}e_{j} \otimes \mathbf{V}e_{j}}{\lambda_{j}} - \sum_{j=1}^{N} \frac{\hat{\mathbf{U}}e_{j} \otimes \hat{\mathbf{V}}e_{j}}{\lambda_{j}} \right\|_{2}^{2} + 2\left\| \sum_{j=N+1}^{\infty} \frac{\mathbf{U}e_{j} \otimes \mathbf{V}e_{j}}{\lambda_{j}} \right\|_{2}^{2} \\
		&= 2\left\| \sum_{j=1}^{N} \frac{[\mathbf{U} - \hat{\mathbf{U}}]e_{j} \otimes \mathbf{V}e_{j} + \hat{\mathbf{U}}e_{j} \otimes [\mathbf{V} - \hat{\mathbf{V}}]e_{j}}{\lambda_{j}}  \right\|_{2}^{2} 
		+ 2\left\| \sum_{j=N+1}^{\infty} \frac{\mathbf{U}e_{j} \otimes \mathbf{V}e_{j}}{\lambda_{j}} \right\|_{2}^{2} \\
		&\leq 2\left\| \sum_{j=1}^{N} \frac{[\mathbf{U} - \hat{\mathbf{U}}]e_{j} \otimes \mathbf{V}e_{j}}{\lambda_{j}} \right\|_{2}^{2} + 2\left\| \sum_{j=1}^{N} \frac{\hat{\mathbf{U}}e_{j} \otimes [\mathbf{V} - \hat{\mathbf{V}}]e_{j}}{\lambda_{j}}\right\|_{2}^{2}\\
		&+ 2\left\| \sum_{j=N+1}^{\infty} \frac{\mathbf{U}e_{j} \otimes \mathbf{V}e_{j}}{\lambda_{j}} \right\|_{2}^{2} \\
		&\leq 2\left[\|\mathbf{V} \|_{2}^{2} \| \mathbf{U} - \hat{\mathbf{U}}\|_{2}^{2} +  \|\hat{\mathbf{U}} \|_{2}^{2} \| \mathbf{V} - \hat{\mathbf{V}}\|_{2}^{2}\right] \sum_{j=1}^{N} \frac{1}{\lambda_{j}^{2}} 
		+ 2\left\| \sum_{j=N+1}^{\infty} \frac{\mathbf{U}e_{j} \otimes \mathbf{V}e_{j}}{\lambda_{j}} \right\|_{2}^{2} \\ 
		&\leq 2\frac{N}{\lambda_{N}^{2}}\left[\|\mathbf{V} \|_{2}^{2} \| \mathbf{U} - \hat{\mathbf{U}}\|_{2}^{2} +  \|\hat{\mathbf{U}} \|_{2}^{2} \| \mathbf{V} - \hat{\mathbf{V}}\|_{2}^{2}\right] 
		+ 2\left\| \sum_{j=N+1}^{\infty} \frac{\mathbf{U}e_{j} \otimes \mathbf{V}e_{j}}{\lambda_{j}} \right\|_{2}^{2} \\ 
	\end{align*}
	
	\textbf{Step 2.} In the same way, we can write the second term as
	\begin{equation*}
		\begin{split}
			\|\hat{\mathbf{W}}_{1} - \hat{\mathbf{W}}_{2}\|_{2}^{2}
			&= \left\| \sum_{j=1}^{N} \frac{\hat{\mathbf{U}}e_{j} \otimes \hat{\mathbf{V}}e_{j}}{\lambda_{j}} - \sum_{j=1}^{N} \frac{\hat{\mathbf{U}}\hat{e}_{j} \otimes \hat{\mathbf{V}}\hat{e}_{j}}{\lambda_{j}} \right\|_{2}^{2}\\
			&\leq \left\| \sum_{j=1}^{N} \frac{\hat{\mathbf{U}}(e_{j} - \hat{e}_{j}) \otimes \hat{\mathbf{V}}e_{j}}{\lambda_{j}}\right\|_{2}^{2} + \left\| \sum_{j=1}^{N} \frac{\hat{\mathbf{U}}\hat{e}_{j} \otimes \hat{\mathbf{V}}(e_{j} - \hat{e}_{j})}{\lambda_{j}}\right\|_{2}^{2}\\
			&\leq 2\|\hat{\mathbf{U}}\|_{2}^{2}\|\hat{\mathbf{V}}\|_{2}^{2} \cdot \sum_{j=1}^{N} \frac{\|e_{j} - \hat{e}_{j}\|^{2} }{\lambda_{j}^{2}}\\
			&\preceq 2\|\hat{\mathbf{U}}\|_{2}^{2}\|\hat{\mathbf{V}}\|_{2}^{2} \|\mathbf{T} - \hat{\mathbf{T}}\|_{2}^{2} \cdot \sum_{j=1}^{N}\frac{1 }{\alpha_{j}^{2}\lambda_{j}^{2}} \\
			&\preceq \frac{N }{\alpha_{N}^{2}\lambda_{N}^{2}}\|\hat{\mathbf{U}}\|_{2}^{2}\|\hat{\mathbf{V}}\|_{2}^{2} \|\mathbf{T} - \hat{\mathbf{T}}\|_{2}^{2} 
		\end{split}
	\end{equation*}
	The third inequality is a consequence of the perturbation bound for eigenfunctions which states that the perturbation $\| e_{j} - \hat{e}_{j} \|$ can be controlled by the perturbation $\|\mathbf{T} - \hat{\mathbf{T}}\|$ of $\mathbf{T}$ divided by the eigenvalue gap $\alpha_{j}$. In the last inequality,  we use the assumption that the eigenvalue gap $\alpha_{j}$ decreases with $N$. If this is not true we can simply replace $\alpha_{N}$ above with $\min_{j \leq N} \alpha_{j}$.
	
	\textbf{Step 3.} And now the third term satisfies,
	\begin{align*}
		\|\hat{\mathbf{W}}_{3} - \hat{\mathbf{W}}_{2}\|_{2}^{2}
		&=\left\| \sum_{j=1}^{N} \frac{\hat{\mathbf{U}}\hat{e}_{j} \otimes \hat{\mathbf{V}}\hat{e}_{j}}{\lambda_{j}} - \sum_{j=1}^{N} \frac{\hat{\mathbf{U}}\hat{e}_{j} \otimes \hat{\mathbf{V}}\hat{e}_{j}}{\hat{\lambda}_{j}} \right\|_{2}^{2}\\
		&=  \left\| \sum_{j=1}^{N} \hat{\mathbf{U}}\hat{e}_{j} \otimes \hat{\mathbf{V}}\hat{e}_{j} \frac{\hat{\lambda}_{j} - \lambda_{j}}{\lambda_{j}\hat{\lambda_{j}}}\right\|_{2}^{2}\\
		&\leq  2\sum_{j=1}^{N} \|\hat{\mathbf{U}}\|_{2}^{2}\|\hat{\mathbf{V}}\|_{2}^{2} \left[\frac{\hat{\lambda}_{j} - \lambda_{j}}{\lambda_{j}\hat{\lambda_{j}}}\right]^{2}\\
		&\leq  2\frac{N}{\lambda_{N}^{2}\hat{\lambda}_{N}^{2}}\|\hat{\mathbf{U}}\|_{2}^{2}\|\hat{\mathbf{V}}\|_{2}^{2} \| \mathbf{T} - \hat{\mathbf{T}} \|_{2}^{2}
	\end{align*}
	Here, we used the perturbation bound for eigenvalues which is given by
	\begin{align*}
		|\hat{\lambda}_{j} - \lambda_{j}| \leq \| \mathbf{T} - \hat{\mathbf{T}} \|.
	\end{align*}
	Since, we choose $N$ such that $\| \mathbf{T} - \hat{\mathbf{T}} \| \leq \lambda_{N}$, we can bound the $\hat{\lambda}_{j}$ in the denominator using the fact that $\hat{\lambda}_{N} \geq \lambda_{N} - \| \mathbf{T} - \hat{\mathbf{T}} \| > 0$ and write,
	\begin{align*}
		\left\| \sum_{j=1}^{N} \frac{\hat{\mathbf{U}}\hat{e}_{j} \otimes \hat{\mathbf{V}}\hat{e}_{j}}{\lambda_{j}} - \sum_{j=1}^{N} \frac{\hat{\mathbf{U}}\hat{e}_{j} \otimes \hat{\mathbf{V}}\hat{e}_{j}}{\hat{\lambda}_{j}} \right\|_{2}^{2}
		&\leq 2\frac{N}{\lambda_{N}^{2}}\|\hat{\mathbf{U}}\|_{2}^{2}\|\hat{\mathbf{V}}\|_{2}^{2}  \frac{\| \mathbf{T} - \hat{\mathbf{T}} \|^{2}}{\left[ \lambda_{N} - \| \mathbf{T} - \hat{\mathbf{T}} \| \right]^{2}}\\
		&\preceq \frac{N}{\lambda_{N}^{4}}\|\hat{\mathbf{U}}\|_{2}^{2}\|\hat{\mathbf{V}}\|_{2}^{2} \| \mathbf{T} - \hat{\mathbf{T}} \|_{2}^{2}
	\end{align*}
	The last inequality follows from the fact that 
	\begin{align*}
		\frac{1}{\left[ \lambda_{N} - \| \mathbf{T} - \hat{\mathbf{T}} \| \right]} = \frac{1}{\lambda_{N}}\left[ 1 + \frac{\| \mathbf{T} - \hat{\mathbf{T}} \|}{\lambda_{N}} + \frac{\| \mathbf{T} - \hat{\mathbf{T}} \|^{2}}{\lambda_{N}^{2}} + \cdots\right] \preceq \frac{1}{\lambda_{N}}
	\end{align*}
	
	\textbf{Step 4.} Putting everything together, we get
	\begin{align*}
		\|\hat{\mathbf{W}}_{1} - \mathbf{W}\|_{2}^{2} &+ \|\hat{\mathbf{W}}_{1} - \hat{\mathbf{W}}_{2}\|_{2}^{2} + \|\hat{\mathbf{W}}_{3} - \hat{\mathbf{W}}_{2}\|_{2}^{2}\\
		&\preceq \frac{N}{\lambda_{N}^{2}}\left[\|\mathbf{V} \|_{2}^{2} \| \mathbf{U} - \hat{\mathbf{U}}\|_{2}^{2} +  \|\hat{\mathbf{U}} \|_{2}^{2} \| \mathbf{V} - \hat{\mathbf{V}}\|_{2}^{2}\right] 
		+ \frac{N }{\alpha_{N}^{2}\lambda_{N}^{2}}\|\hat{\mathbf{U}}\|_{2}^{2}\|\hat{\mathbf{V}}\|_{2}^{2} \|\mathbf{T} - \hat{\mathbf{T}}\|_{2}^{2}\\ 
		&+ \frac{N}{\lambda_{N}^{4}}\|\hat{\mathbf{U}}\|_{2}^{2}\|\hat{\mathbf{V}}\|_{2}^{2} \| \mathbf{T} - \hat{\mathbf{T}} \|_{2}^{2} + 2\left\| \sum_{j=N+1}^{\infty} \frac{\mathbf{U}e_{j} \otimes \mathbf{V}e_{j}}{\lambda_{j}} \right\|_{2}^{2}\\
		&\preceq \frac{N}{\lambda_{N}^{2}} \left[\| \mathbf{U} - \hat{\mathbf{U}}\|_{2}^{2} + \| \mathbf{V} - \hat{\mathbf{V}}\|_{2}^{2}\right] + \frac{N}{\lambda_{N}^{2}\alpha_{N}^{2}} \| \mathbf{T} - \hat{\mathbf{T}} \|_{2}^{2} + \left\| \sum_{j=N+1}^{\infty} \frac{\mathbf{U}e_{j} \otimes \mathbf{V}e_{j}}{\lambda_{j}} \right\|_{2}^{2} 
	\end{align*}
	since $\alpha_{j} < \lambda_{j}$. Hence proved.
\end{proof}

\begin{remark}
	The next result illustrates how the estimate from Lemma \ref{lem:estimate} can be used to derive consistency and rates of convergence for our estimator. We begin by considering the case when the tuning parameters $N_{p}$ are allowed to be random and then consider the case when they are required to be deterministic. Strictly speaking the results are independent and the impatient reader can skip them and go directly to Lemma \ref{lem:prob_rates1}, but we believe that they are helpful in understanding the the proof of the rate of convergence result.
\end{remark}
\begin{theorem}[Consistency]\label{thm:consistent}
	Let $K_{\Omega}$ be a continuous partial covariance on a serrated domain $\Omega$ of $m$ intervals, $\hat{K}_{\Omega} \in L^{2}(\Omega)$. If $\hat{K}_{\Omega} \to K_{\Omega}$ in $L^{2}(\Omega)$ and the regularization parameters $\mathbf{N} = (N_{p})_{p=1}^{m-1}$ are chosen such that for $1 \leq p < m$, $\delta_{p} = \| \hat{\mathbf{S}}_{p} - \mathbf{S}_{p} \|_{2} \vee \| \hat{\mathbf{D}}_{p} - \mathbf{D}_{p} \|_{2}$ and $\epsilon_{p} = \| \hat{\mathbf{J}}_{p} - \mathbf{J}_{p} \|_{2}$ we have 
	\begin{enumerate}
		\item $N_{p} \to \infty$,
		\item $\lambda_{p,N_{p}} > \epsilon_{p}$,
		\item $\frac{N_{p}}{\lambda_{p,k}^{2}} \delta_{p}^{2} \to 0$ and $\frac{N_{p}}{\lambda_{p,k}^{2}\alpha_{p,k}^{2}} \epsilon_{p}^{2} \to 0$ 
	\end{enumerate}
	as $\delta_{p}, \epsilon_{p} \to 0$ where $\lambda_{p,k}$ denotes the $k$th eigenvalue of $\mathbf{J}_{p}$ and $\alpha_{p,k}$ is given by 
	\begin{equation*}
		\alpha_{p,k} = \begin{cases}
			(\lambda_{p,1} - \lambda_{p,2})/ 2\sqrt{2} &k = 1\\
			\big[ (\lambda_{p,k-1} - \lambda_{p,k}) \wedge (\lambda_{p,k} - \lambda_{p,k+1}) \big] / 2\sqrt{2} &k > 1
		\end{cases}
	\end{equation*}
	then $\hat{K}_{\star} \to K_{\star}$ in $L^{2}(I \times I)$. 
\end{theorem}
\begin{proof}[Proof of Theorem \ref{thm:consistent}]
	We again proceed by induction on the number of intervals $m$. The claim is vacuously true for $m=1$. Assume that it holds for $m = q-1$ for some $q \geq 2$. We shall show that it holds for $m = q$. 
	
	Consider a partial covariance $K_{\Omega}$ on a serrated domain $\Omega$ of $q$ intervals: $I_{1}, \dots, I_{q}$. Let $I^{\prime} = \cup_{j=1}^{q-1} I_{j}$ and $\Omega^{\prime} = \cup_{j=1}^{q-1} I_{j} \times I_{j}$. Define $K_{\Omega^{\prime}} = K_{\Omega}|_{\Omega^{\prime}}$. Let $\epsilon = \int_{\Omega} [\hat{K}_{\Omega}(s,t) - K_{\Omega}(s,t)]^{2} ~dx dy$. We can decompose the error of $\hat{K}$ as follows:
	\begin{equation}\label{eq:decomp_error_1}
		\begin{split}
			\textstyle\int_{I \times I} \left[\hat{K}_{\star}(s,t) - K_{\star}(s,t)\right]^{2} dx ~dy 
			&=\textstyle \int_{I^{\prime} \times I^{\prime}} \left[\hat{K}_{\star}(s,t) - K_{\star}(s,t)\right]^{2} dx ~dy \\
			&+\textstyle \int_{A_{q}} \left[\hat{K}_{\Omega}(s,t) - K_{\Omega}(s,t)\right]^{2} dx ~dy\\
			&+ \textstyle 2\int_{R_{q}} \left[\hat{K}_{\star}(s,t) - K_{\star}(s,t)\right]^{2} dx ~dy
		\end{split}
	\end{equation}
	where $A_{q} = [I_{q} \times I_{q}] \setminus [J_{q-1} \times J_{q-1}]$. By construction, the estimator for the canonical extension of $K_{\Omega^{\prime}}$ is the restriction $\hat{K}|_{I^{\prime} \times I^{\prime}}$ of the estimator $\tilde{K}$ for the canonical extension of $K_{\Omega}$. Therefore, the first term in Equation \ref{eq:decomp_error_1} converges to zero as $\epsilon \to 0$, by the induction hypothesis. The same applies to the second term for more obvious reasons. It suffices to show that the third term
	\begin{equation*}
		\int_{R_{q}} \left[\hat{K}_{\star}(s,t) - K_{\star}(s,t)\right]^{2} dx dy = \| \hat{\mathbf{R}}_{q} - \mathbf{R}_{q} \|_{2}^{2}
	\end{equation*}
	converges to zero as $\epsilon \to 0$. Clearly, by Lemma \ref{lem:estimate},
	\begin{equation*}
		\begin{split}
			\| \hat{\mathbf{R}}_{q} - \mathbf{R}_{q}  \|_{2}^{2}
			&\preceq \frac{N_{q}}{\lambda_{q, N_{q}}^{2}} \left[\| \mathbf{S}_{q} - \hat{\mathbf{S}}_{q}\|_{2} \wedge \| \mathbf{D}_{q} - \hat{\mathbf{D}}_{q}\|_{2}\right]^{2} \\
			&+ \frac{N_{q}}{\lambda_{q,N_{q}}^{2}\alpha_{q, N_{q}}^{2}} \| \mathbf{J}_{q} - \hat{\mathbf{J}}_{q} \|_{2}^{2} + \left\| \sum_{j=N_{q}+1}^{\infty} \frac{\mathbf{S}_{q}e_{q,j} \otimes \mathbf{D}_{q}^{\ast}e_{q,j}}{\lambda_{q,j}} \right\|_{2}^{2}. 
		\end{split}
	\end{equation*}
	The first two terms converge to zero because $N_{q}$ has been chosen such that$\frac{N_{p}}{\lambda_{p,N_{p}}^{2}\tilde{\alpha}_{p,N_{p}}^{2}} \| \hat{K}_{\Omega} - K_{\Omega} \|_{L^{2}(\Omega)} ^{2} \to 0$ which means that $\frac{N_{q}}{\lambda_{q,k}^{2}} \delta_{q}^{2} \to 0$ and $\frac{N_{q}}{\lambda_{q,k}^{2}\alpha_{q,k}^{2}} \epsilon_{q}^{2} \to 0$ 
	and the last term converges to $0$ as $N_{q} \to \infty$. The conclusion follows.
\end{proof}

Notice that Equation \ref{eq:estimate} decomposes the error $\| \mathbf{W} - \hat{\mathbf{W}} \|_{2}^{2}$ into estimation and approximation terms as follows: 
\begin{equation*}
	\begin{split}
		E_{N} &= \frac{N}{\lambda_{N}^{2}} \left[\| \mathbf{U} - \hat{\mathbf{U}}\|_{2}^{2} + \| \mathbf{V} - \hat{\mathbf{V}}\|_{2}^{2}\right] + \frac{N}{\lambda_{N}^{2}\alpha_{N}^{2}} \| \mathbf{T} - \hat{\mathbf{T}} \|_{2}^{2}\\
		A_{N} &= \left\| \sum_{j=N+1}^{\infty} \frac{\mathbf{U}e_{j} \otimes \mathbf{V}e_{j}}{\lambda_{j}} \right\|_{2}^{2} 
	\end{split}
\end{equation*}
Notice that $A_{N}$ is a completely deterministic term which depends on $\mathbf{U}$, $\mathbf{V}$ and the spectral properties of $\mathbf{T}$. 
Furthermore, the error in $\mathbf{U}$ and $\mathbf{V}$ has much less weight than the error in $\mathbf{T}$. 


\begin{lemma}\label{lem:rate}
	Under the setting of Lemma \ref{lem:estimate}, if $\lambda_{N} \sim N^{-\alpha}$ and $A_{N} \sim N^{-\beta}$, the best error is achieved for $$N \sim \delta^{-2/2\alpha+\beta+1} \wedge \epsilon^{-2/4\alpha + \beta + 3}$$ and is given by 
	\begin{equation*}
		\| \mathbf{W} - \hat{\mathbf{W}} \|_{2} \preceq  N^{-\beta/2}
	\end{equation*}
\end{lemma}
\begin{proof}
	Since, $\lambda_{N} \sim N^{-\alpha}$, it follows that $\alpha_{N} \sim N^{-\alpha-1}$. Following the above discussion, we can write $$\| \mathbf{W} - \hat{\mathbf{W}} \|_{2}^{2} \preceq E_{N} + A_{N}$$ where 
	\begin{equation*}
		\begin{split}
			E_{N} &= N^{2\alpha + 1} \delta^{2} + N^{4 \alpha + 3}\epsilon^{2} \\
			A_{N} &= N^{-\beta}
		\end{split}
	\end{equation*} 
	To minimize the error or in other wordsm maximize the decay rate, the estimation term should decrease at least as fast as the approximation term as $N$ increases. Thus, $N^{2\alpha + 1} \delta^{2} \preceq N^{-\beta}$ and $N^{4 \alpha + 3}\epsilon^{2} \preceq N^{-\beta}$. It follows that for the sum $E_{N} + A_{N}$ to decay at the maximum rate we need that
	$$N \sim \delta^{-2/2\alpha+\beta+1} \wedge \epsilon^{-2/4\alpha + \beta + 3}.$$
	The total error then satisfies
	\begin{equation*}
		\| \mathbf{W} - \hat{\mathbf{W}} \|_{2}^{2} \preceq N^{-\beta}
	\end{equation*}
	and the conclusion follows.
\end{proof}

\begin{theorem}[Rate of Convergence]\label{thm:rate_convergence_rand}
	Let $K_{\Omega}$ be a partial covariance on a serrated domain $\Omega$ of $m$ intervals and $\hat{K}_{\Omega}$ be its estimate. Let $\hat{K}$ be defined as above. Assume that for every $1 \leq p < m$, we have $\lambda_{p,k} \sim k^{-\alpha}$ and  $A_{p,k}  \sim k^{-\beta}$. If the truncation parameters $\mathbf{N} = (N_{p})_{p=1}^{m-1}$ are chosen according to the rule
	$$N_{p} \sim \| \hat{K}_{\Omega} - K_{\Omega} \|_{L^{2}(\Omega)}^{-2\gamma_{p}/\beta}$$ 
	then 
	\begin{equation*}
		\| \hat{K}_{\star} - K_{\star} \|_{L^{2}(I\times I)} \preceq \| \hat{K}_{\Omega} - K_{\Omega} \|_{L^{2}(\Omega)}^{\gamma_{m-1}}
	\end{equation*}
	where $\gamma_{m-1} = \tfrac{\beta}{4\alpha + \beta + 3}\left[ \tfrac{\beta}{2\alpha+\beta+1}\right]^{m-2}$ for $m > 1$ and $1$ for $m = 1$.
\end{theorem}
\begin{proof}[Proof of Theorem \ref{thm:rate_convergence_rand}]
	We proceed by induction on the number of intervals $m$. The statement is vacuously true for the base case $m=1$. Assume that it holds for $m = q-1$ for some $q \geq 3$. We shall show that it must hold for $m = q$. 
	
	As before, consider a partial covariance $K_{\Omega}$ on a serrated domain $\Omega$ of $q$ intervals: $I_{1}, \dots, I_{q}$. Let $I^{\prime} = \cup_{j=1}^{q-1} I_{j}$ and $\Omega^{\prime} = \cup_{j=1}^{q-1} I_{j} \times I_{j}$. Define $K_{\Omega^{\prime}} = K_{\Omega}|_{\Omega^{\prime}}$. Let $\epsilon = \int_{\Omega} [\hat{K}_{\Omega}(s,t) - K_{\Omega}(s,t)]^{2} ~dx dy$. The error of $\hat{K}_{\star}$ can be decomposed as in Equation \ref{eq:decomp_error_1}. 
	
	By construction, the estimator for the canonical extension of $K_{\Omega^{\prime}}$ is the restriction $\hat{K}_{\star}|_{I^{\prime} \times I^{\prime}}$ of the estimator $\hat{K}_{\star}$ for the canonical extension of $K_{\Omega}$. By the induction hypothesis,
	\begin{equation*}
		\int_{I^{\prime} \times I^{\prime}} \left[\hat{K}_{\star}(s,t) - K_{\star}(s,t)\right]^{2} dx dy~ \lesssim \left[\int_{\Omega^{\prime}} \left[\hat{K}_{\Omega}(s,t) - K_{\Omega}(s,t)\right]^{2} dx dy\right]^{\gamma_{m-2}} \lesssim \epsilon^{\gamma_{m-2}}.
	\end{equation*}
	Thus the first term in Equation (\ref{eq:decomp_error_1}) can be bounded by a power of $\epsilon$. The second term is obviously less than $\epsilon$. We now turn our attention to the third term, 
	\begin{equation*}
		\int_{R_{q}} \left[\hat{K}_{\star}(s,t) - K_{\star}(s,t)\right]^{2} dx dy~ = \| \hat{\mathbf{R}}_{q} - \mathbf{R}_{q} \|_{2}^{2}
	\end{equation*}
	Using Lemma \ref{lem:rate}, we have
	\begin{equation*}
		N_{q} \sim \left[ \epsilon^{\gamma_{m-2}} \right]^{-2/2\alpha+\beta+1} \wedge \epsilon^{-2/4\alpha+\beta+1} = \left[ \epsilon^{\gamma_{m-2}} \right]^{-2/2\alpha+\beta+1}
	\end{equation*}
	and,
	\begin{equation*}
		\| \hat{\mathbf{R}}_{q} - \mathbf{R}_{q} \|_{2}^{2} \preceq N_{q}^{-\beta} \sim \left[ \epsilon^{\gamma_{m-2}} \right]^{\beta/2\alpha+\beta+1}  = \epsilon^{\gamma_{m-1}}
	\end{equation*}	
	From Equation \ref{eq:decomp_error_1},  
	\begin{equation}
		\begin{split}
			\textstyle\int_{I \times I} \left[\hat{K}_{\star}(s,t) - K_{\star}(s,t)\right]^{2} dx ~dy 
			&\preceq \epsilon^{\gamma_{m-2}} + \epsilon + 2\epsilon^{\gamma_{m-1}} \sim \epsilon^{\gamma_{m-1}}
		\end{split}
	\end{equation}	
	and the proof is complete.
\end{proof}

We now present results which treat the tuning parameter as a deterministic quantity.
\subsection{Consistency and Rates of Convergence with Nonrandom Tuning Parameter}
Recall the setting of Lemma \ref{lem:estimate} and let $\epsilon_{n} = \| \mathbf{T} - \hat{\mathbf{T}}_{n} \|_{2}$ and $\delta_{n} = \| \mathbf{U} - \hat{\mathbf{U}}_{n} \|_{2} \vee \| \mathbf{V} - \hat{\mathbf{V}}_{n} \|_{2}$.
\begin{lemma}
	\label{lem:prob_rates1}
	Assume that $\lambda_{N} \sim N^{-\alpha}$, $A_{N} \sim N^{-\beta}$, $\epsilon_{n} = O_{\mathbb{P}}(n^{-\zeta})$ and $\delta_{n} = O_{\mathbb{P}}(n^{-\zeta^{\prime}})$. 
	\begin{enumerate}
		\item If the tuning parameter scales as $N \sim n^{x}$ where $$0 < x < \frac{\zeta}{2\alpha+2/3} \wedge \frac{\zeta^{\prime}}{\alpha+1/2},$$ then $\| \mathbf{W} - \hat{\mathbf{W}}_{n} \|_{2} \to 0$ in probability as $n \to \infty$.
		\item For every $\varepsilon > 0$,
		\begin{equation*}
			\| \mathbf{W} - \hat{\mathbf{W}}_{n} \|_{2} = O_{\mathbb{P}}(1/n^{y_{\ast}-\varepsilon})
		\end{equation*}
		where
		\begin{equation*}
			y_{\ast} = \frac{\beta \zeta}{\beta + 2\alpha + 3/2} \wedge \frac{ \beta \zeta^{\prime}}{\beta + \alpha + 1/2}
		\end{equation*}
		so long as the tuning parameter satisfies $N \sim n^{y_{\ast}/ \beta}$.
	\end{enumerate}
	
\end{lemma}
\begin{proof}
	Let $\{\eta_{n}\}_{n=1}^{\infty}$ be a sequence with $\eta_{n} > 0$. Then,
	\begin{equation*}
		\begin{split}
			&\mathbb{P} \{ \| \mathbf{W} - \hat{\mathbf{W}}_{n} \|_{2}^{2} > \eta_{n} \} \\
			&= 	\mathbb{P} \{ \| \mathbf{W} - \hat{\mathbf{W}}_{n} \|_{2}^{2} > \eta_{n} ~|~ \lambda_{N} > \| \hat{\mathbf{T}}_{n} - \mathbf{T} \| \} \cdot \mathbb{P} \{ \lambda_{N} > \| \hat{\mathbf{T}}_{n} - \mathbf{T} \| \} \\
			&\qquad\qquad+ \mathbb{P} \{ \| \mathbf{W} - \hat{\mathbf{W}}_{n} \|_{2}^{2} > \eta_{n} ~|~ \lambda_{N} \leq \| \hat{\mathbf{T}}_{n} - \mathbf{T} \| \} \cdot \mathbb{P} \{ \lambda_{N} \leq \| \hat{\mathbf{T}}_{n} - \mathbf{T} \| \} \\
			&\leq 	\mathbb{P} \{ \| \mathbf{W} - \hat{\mathbf{W}}_{n} \|_{2}^{2} > \eta_{n} ~|~ \lambda_{N} > \| \hat{\mathbf{T}}_{n} - \mathbf{T} \| \} \cdot 1 \\
			&\qquad\qquad+ 1 \cdot \mathbb{P} \{ \lambda_{N} \leq \| \hat{\mathbf{T}}_{n} - \mathbf{T} \| \} \\
			&\leq \textstyle\mathbb{P} \Big\{ \frac{N}{\lambda_{N}^{2}} \delta_{n}^{2} + \frac{N}{\lambda_{N}^{2}\alpha_{N}^{2}} \epsilon_{n}^{2} + \Big\| \sum_{j=N+1}^{\infty} \frac{\mathbf{U}e_{j} \otimes \mathbf{V}e_{j}}{\lambda_{j}} \Big\|_{2}^{2}  > C\eta_{n} ~\Big|~ \lambda_{N} > \epsilon_{n} \Big\} + \mathbb{P} \{ \epsilon_{n} \geq \lambda_{N}  \} \\
			&\leq \frac{\mathbb{P} \left\lbrace \frac{N}{\lambda_{N}^{2}} \delta_{n}^{2} + \frac{N}{\lambda_{N}^{2}\alpha_{N}^{2}} \epsilon_{n}^{2} + \left\| \sum_{j=N+1}^{\infty} \frac{\mathbf{U}e_{j} \otimes \mathbf{V}e_{j}}{\lambda_{j}} \right\|_{2}^{2}  > C\eta_{n} \right\rbrace}{\mathbb{P}\{\lambda_{N} > \epsilon_{n}\}} + \mathbb{P} \{ \epsilon_{n} \geq \lambda_{N}  \} \\
			&\leq \frac{\mathbb{P} \left\lbrace \frac{N}{\lambda_{N}^{2}} \delta_{n}^{2} + \frac{N}{\lambda_{N}^{2}\alpha_{N}^{2}} \epsilon_{n}^{2} + \left\| \sum_{j=N+1}^{\infty} \frac{\mathbf{U}e_{j} \otimes \mathbf{V}e_{j}}{\lambda_{j}} \right\|_{2}^{2}  > C\eta_{n} \right\rbrace}{1 - \mathbb{P}\{\epsilon_{n} \geq \lambda_{N} \}} + \mathbb{P} \{ \epsilon_{n} \geq \lambda_{N}  \} \\
		\end{split}
	\end{equation*}
	It suffices for us to show that $\mathbb{P} \{ \epsilon_{n} \geq \lambda_{N}  \} \to 0$ and
	\begin{equation*}
		\textstyle \mathbb{P} \left\lbrace \frac{N}{\lambda_{N}^{2}} \delta_{n}^{2} + \frac{N}{\lambda_{N}^{2}\alpha_{N}^{2}} \epsilon_{n}^{2} + \left\| \sum_{j=N+1}^{\infty} \frac{\mathbf{U}e_{j} \otimes \mathbf{V}e_{j}}{\lambda_{j}} \right\|_{2}^{2}  > C\eta_{n} \right\rbrace \to 0
	\end{equation*} 
	as $n \to \infty$. Furthermore, since $\lambda_{N} \sim N^{-\alpha}$, we have $\alpha_{N} \sim N^{-\alpha- 1}$ and we are given that
	\begin{equation*}
		\textstyle	\left\| \sum_{j=N+1}^{\infty} \frac{\mathbf{U}e_{j} \otimes \mathbf{V}e_{j}}{\lambda_{j}} \right\|_{2} \sim N^{-\beta}.
	\end{equation*}
	Now, $$\mathbb{P} \{ \epsilon_{n} \geq \lambda_{N}  \} \leq \mathbb{P} \{ n^{\zeta}\epsilon_{n} \geq n^{\zeta}N^{-\alpha}  \}$$ and
	\begin{equation*}
		\begin{split}
			&\textstyle \mathbb{P} \left\lbrace \frac{N}{\lambda_{N}^{2}} \delta_{n}^{2} + \frac{N}{\lambda_{N}^{2}\alpha_{N}^{2}} \epsilon_{n}^{2} + \left\| \sum_{j=N+1}^{\infty} \frac{\mathbf{U}e_{j} \otimes \mathbf{V}e_{j}}{\lambda_{j}} \right\|_{2}^{2}  > C\eta_{n} \right\rbrace \\
			&\leq \textstyle
			\mathbb{P} \left\lbrace N^{2\alpha+1} \delta_{n}^{2}  > C^{\prime}\eta_{n} \right\rbrace + \mathbb{P}\left\lbrace N^{4\alpha+3} \epsilon_{n}^{2} > C^{\prime}\eta_{n} \right\rbrace + \mathbb{P}\left\lbrace N^{-2\beta} > C^{\prime}\eta_{n} \right\rbrace\\
			&\leq 
			\mathbb{P} \left\lbrace  n^{2\zeta^{\prime}}\delta_{n}^{2}  > C^{\prime}\frac{n^{2\zeta^{\prime}}\eta_{n}}{N^{2\alpha+1}} \right\rbrace + \mathbb{P}\left\lbrace n^{2\zeta} \epsilon_{n}^{2} > C^{\prime}\frac{n^{2\zeta}\eta_{n}}{N^{4\alpha+3}} \right\rbrace + \mathbb{P}\left\lbrace 1 > C^{\prime}\frac{\eta_{n}}{N^{-2\beta}   } \right\rbrace\\
		\end{split}
	\end{equation*}
	where $C^{\prime} = C/3$.
	
	Let $\eta_{n} \sim n^{-2y}$ and $N \sim n^{x}$. We need to show that there exists $y, x > 0$  such that the following terms increase with $n$:
	\begin{alignat}{2}
		\nonumber
		\frac{n^{\zeta}}{N^{\alpha}} &\sim n^{\zeta - \alpha x}, 
		&\frac{n^{2\zeta^{\prime}}\eta_{n}}{N^{2\alpha+1}} &\sim n^{2'\zeta^{\prime}-2y - (2\alpha+1)x},\\ \nonumber
		\frac{n^{2\zeta}\eta_{n}}{N^{4\alpha+3}} &\sim n^{2\zeta-2y-(4\alpha+3)x}, 
		&\frac{\eta_{n}}{N^{-2\beta}} &\sim n^{-2y+2\beta x}.
	\end{alignat}
	It follows that:
	\begin{equation}
		\label{eq:paraeqs}
		\begin{split}
			\zeta - \alpha x &> 0 \\
			2\zeta^{\prime}-2y - (2\alpha+1)x &> 0 \\
			2\zeta-2y-(4\alpha+3)x &> 0 \\
			-2y+2\beta x &> 0 \\
		\end{split}
	\end{equation}
	\begin{figure}[h]
		\centering
		\begin{center}
			\begin{tikzpicture}		
				\draw [shift={(0.5,0.5)}, fill=gray, ->] (0,0) -- (0,5);
				\draw [shift={(0.5,0.5)}, fill=gray, ->] (0,0) -- (12,0);
				\draw [shift={(0.5,0.5)}, fill=yellow!30] (0,0) -- (6,0) -- (3,2) -- (3/2,5/2) -- cycle;
				
				\draw [shift={(0.5,0.5)}, color=red!50] (0,0) -- (3,5);
				\node [shift={(0.5,0.5)}, draw, rectangle, thick, fill = white, text = red!50, draw = red!50] at (3,4) (eq1) {\tiny $-2y + 2\beta x = 0$};
				
				\draw [shift={(0.5,0.5)}, color=blue!50] (0,3) node [left] {$\zeta^{\prime}$} -- (9,0) node [below] {$\zeta^{\prime}/(\alpha+1/2)$};
				
				\draw [shift={(0.5,0.5)}, color=green!50] (0,4) node [left] {$\zeta$} -- (6,0) node [below] {$\zeta/(2\alpha+3/2)$};				
				
				\draw [shift={(0.5,0.5)}, color=cyan] (11,0) node[below] {$\zeta/\alpha$} -- (11,5);
				
				\draw [shift={(0.5,0.5)}, fill=yellow!30, dashed] (0,0) -- (6,0) -- (3,2) -- (3/2,5/2) -- cycle;
				\node [shift={(0.5,0.5)}, draw, rectangle, fill = white, text = green!50, draw = green!50] at (4.5,0.5) (eq3) {{\tiny $2\zeta-2y - (4\alpha+3)x = 0$}};			
				\node [shift={(0.5,0.5)}, draw, rectangle, fill = white, text = cyan!50, draw = cyan!50] at (11,2.5) (eq2) {\tiny$\zeta- \alpha x = 0$};
				\node [shift={(0.5,0.5)}, draw, rectangle, fill = white, text = blue!50, draw = blue!50] at (6,1.5) (eq2) {\tiny$2\zeta^{\prime}-2y - (2\alpha+1)x = 0$};
				
				\draw[shift={(0.5,0.5)}] (3/2,5/2) circle (3pt);			
				\draw[shift={(0.5,0.5)}] (12/7,20/7) circle (3pt);
				
				
				
			\end{tikzpicture}
			\caption{The yellow region indicates the solutions $(x,y)$ of the inequalities (\ref{eq:paraeqs}).} 
			\label{fig:eqs_study}
		\end{center}
	\end{figure}
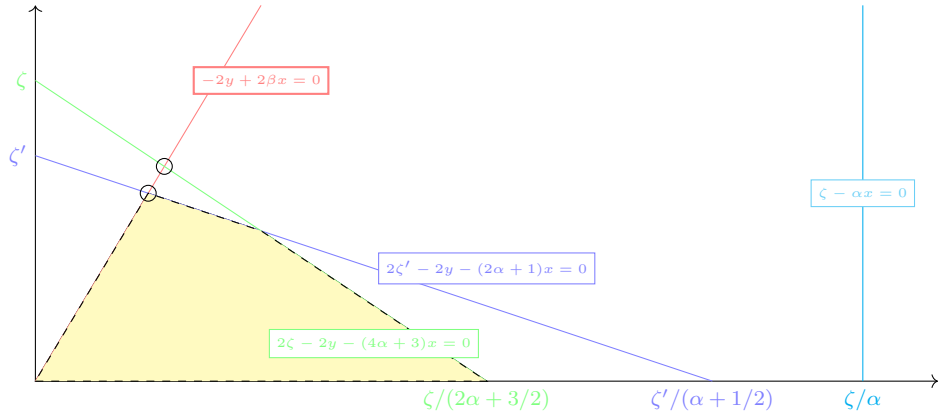	
	Any pair $(x,y)$ with $x,y > 0$ satisfying the inequalities (\ref{eq:paraeqs}), indicated by the yellow region in Figure \ref{fig:eqs_study}, corresponds to a consistent estimator for which the tuning parameter scales according to $N \sim n^{x}$ and the error decreases at least as fast as $n^{-y}$ in probability. It thus follows that so long as $x$ satisfies,
	\begin{equation*}
		x < \frac{\zeta}{2\alpha+2/3} \wedge \frac{\zeta^{\prime}}{\alpha+1/2}
	\end{equation*}
	we have $\| \mathbf{W} - \hat{\mathbf{W}}_{n} \| \to 0$ in probability as $n \to \infty$.
	
	Of course, nothing prevents us from choosing an $x$ for which we can have the highest possible value of $y$. From Figure \ref{fig:eqs_study}, it is clear that the supremum of $y$ is given by
	\begin{equation*}
		y_{\ast} = \frac{\beta \zeta}{\beta + 2\alpha + 3/2} \wedge \frac{ \beta \zeta^{\prime}}{\beta + \alpha + 1/2} 
	\end{equation*}
	depending on which line among the blue and green ones intersects with the red line first and thus at a higher value of $y$. The corresponding value of $x$ is given by $x_{\ast} = y_{\ast}/\beta$. Thus by choosing $x$ to be $x_{\ast}$, we have that $\| \mathbf{W} - \hat{\mathbf{W}}_{n} \|_{2} = O_{\mathbb{P}}(1/n^{y_{\ast}-\varepsilon})$ for every $\varepsilon > 0$.
\end{proof}

\begin{proof}[Proof of Theorem \ref{thm:rate_convergence}]
	In essence, we shall merely apply Lemma \ref{lem:prob_rates1} repeatedly. We proceed by induction on the number of intervals $m$. The statement is true for the base case $m=2$ due to Lemma \ref{lem:prob_rates1}. Assume that it holds for $m = q-1$ for some $q > 3$. We shall show that it must hold for $m = q$. 
	
	As before, consider a partial covariance $K_{\Omega}$ on a serrated domain $\Omega$ of $q$ intervals: $I_{1}, \dots, I_{q}$. Let $I^{\prime} = \cup_{j=1}^{q-1} I_{j}$ and $\Omega^{\prime} = \cup_{j=1}^{q-1} I_{j} \times I_{j}$. Define $K_{\Omega^{\prime}} = K_{\Omega}|_{\Omega^{\prime}}$. Let $\epsilon = \int_{\Omega} [\hat{K}_{\Omega}(s,t) - K_{\Omega}(s,t)]^{2} ~dx dy$. The error of $\hat{K}_{\star}$ can be decomposed as in Equation \ref{eq:decomp_error_1}. 
	
	By construction, the estimator for the canonical extension of $K_{\Omega^{\prime}}$ is the restriction $\hat{K}_{\star}|_{I^{\prime} \times I^{\prime}}$ of the estimator $\hat{K}_{\star}$ for the canonical extension of $K_{\Omega}$. By the induction hypothesis,
	\begin{equation*}
		\int_{\Omega^{\prime}} \left[\hat{K}_{\Omega}(s,t) - K_{\Omega}(s,t)\right]^{2} dx dy = O_{\mathbb{P}}(n^{-\zeta})
	\end{equation*}
	implies that for every $\varepsilon > 0$
	\begin{equation*}
		\int_{I^{\prime} \times I^{\prime}} \left[\hat{K}_{\star}(s,t) - K_{\star}(s,t)\right]^{2} dx dy = O_{\mathbb{P}}(1/n^{\zeta\gamma_{m-2}-\varepsilon}).
	\end{equation*}
	Thus the first term in Equation (\ref{eq:decomp_error_1}) can be bounded in probability by a power of $n^{-\zeta}$. This means that in the language of Lemma \ref{lem:prob_rates1}, $\delta_{n} = O_{\mathbb{P}}(1/n^{\zeta\gamma_{m-2} - \varepsilon})$ and thus $\zeta^{\prime} = \zeta\gamma_{m-2} - \varepsilon$. It is given that the second term is $O_{\mathbb{P}}(n^{-\zeta})$. We now turn our attention to the third term, 
	\begin{equation*}
		\int_{R_{q}} \left[\hat{K}_{\star}(s,t) - K_{\star}(s,t)\right]^{2} dx dy~ = \| \hat{\mathbf{R}}_{q} - \mathbf{R}_{q} \|_{2}^{2}
	\end{equation*}
	Using Lemma \ref{lem:prob_rates1}, if the tuning parameter satisfies $N_{q} \sim n^{\gamma_{m-1}/\beta}$, we have
	\begin{equation*}
		\begin{split}
			\| \hat{\mathbf{R}}_{q} - \mathbf{R}_{q} \|_{2} 
			&= O_{\mathbb{P}}\left(1/n^{\frac{\beta \zeta}{\beta + 2\alpha + 3/2} \wedge \frac{ \beta \zeta^{\prime}}{\beta + \alpha + 1/2}} \right) \\
			&= O_{\mathbb{P}}\left(1/n^{\frac{\beta \zeta}{\beta + 2\alpha + 3/2} \wedge \frac{ \beta (\zeta\gamma_{m-2} - \varepsilon)}{\beta + \alpha + 1/2}} \right) \\
			&= O_{\mathbb{P}}\left(1/n^{\zeta\gamma_{m-1} - \varepsilon^{\prime}} \right) \\
		\end{split}
	\end{equation*}
	where $\varepsilon^{\prime} > 0$ can be arbitrarily small. From Equation \ref{eq:decomp_error_1},  
	\begin{equation}
		\begin{split}
			\textstyle&\int_{I \times I} \left[\hat{K}_{\star}(s,t) - K_{\star}(s,t)\right]^{2} dx ~dy \\
			&= O_{\mathbb{P}}(1/n^{\zeta\gamma_{m-2} - \varepsilon}) + O_{\mathbb{P}}(1/n^{\zeta}) + O_{\mathbb{P}}(1/n^{\zeta\gamma_{m-1}-\varepsilon^{\prime}}) = O_{\mathbb{P}}(1/n^{\zeta\gamma_{m-1}-\varepsilon^{\prime}})
		\end{split}
	\end{equation}	
	and the proof is complete.
\end{proof}

\begin{remark}
	It is indeed possible to give a general consistency result like Lemma \ref{lem:prob_rates1}(i) but for the $m$-serrated domain with $m > 2$ using Lemma \ref{lem:prob_rates1} as before, however this proves to be a tedious exercise which doesn't tell us significantly more than what we already know from Figure \ref{fig:eqs_study} and Theorem \ref{thm:rate_convergence}. Hence, we shall skip it.
\end{remark}

\subsection{Beyond Serrated Domains}

\begin{proof}[Proof of Theorem \ref{thm:unique-nearly-serrated}]
	Let $K_{1}$ and $K_{2}$ be completions of $K_{\widetilde{\Omega}}$ and assume that $K_{\widetilde{\Omega}}|_{\Omega}$ admits a unique completion. Then $K_{1}$ and $K_{2}$ are completions of $K_{\widetilde{\Omega}}|_{\Omega}$, implying that $K_{1} = K_{2}$. 
\end{proof}

\begin{proof}[Proof of Theorem \ref{thm:unique-canonical-nearly}]
	Let $s,t \in \widetilde{\Omega}^{c}$ separated by $J \subset I$ in $(I, \widetilde{\Omega})$. Let $J_{-} = \{u \in I: u \leq v \mbox{ for some } v \in J\}$ and $J_{+} = \{u \in I: u \geq v \mbox{ for some } v \in J\}$. Define 
	\begin{equation*}
		\bar{\Omega} = (J_{-} \times J_{-}) \cup  (J_{+} \times J_{+})
	\end{equation*}
	
	If $K$ is the unique completion of $K_{\widetilde{\Omega}}$ then it is a unique completion of the partial covariance $K|_{\bar{\Omega}}$ on the serrated domain $\Omega$. The canonical completion of $K|_{\bar{\Omega}}$ would thus have to be same as the unique completion $K$ and therefore,
	\begin{equation*}
		K(s,t) = \langle k_{s, J}, k_{t, J} \rangle
	\end{equation*}
	It follows that $K$ is the canonical completion of $K_{\widetilde{\Omega}}$.
	
	\begin{figure}[h]
		\centering
		\begin{center}
			\begin{tikzpicture}
				\draw [shift={(0.5,0.5)}, fill=red!50] (0,0) -- (0,5) -- (5,5) -- (5,0) -- cycle;			
				\draw [shift={(0.5,0.5)}, fill=gray!70] (0,0) -- (2,0) .. controls (3,2.5) and (4,0.5) .. (5,3) -- (5,5) -- (3,5) .. controls (0.5,4) and (2.5,3) .. (0,2) -- cycle;
				\draw [shift={(1.5,1.5)}] (0,0) -- (0,2) -- (2,2) -- (2,0) -- cycle;
				
				\draw [|-,shift = {(0.5,0.5)}, blue!50] (3.5,1) -- (3.5,2.5);
				\draw [|-,shift = {(0.5,0.5)}, blue!50] (3.5,3) -- (3.5,1.5) node [midway,fill=gray!70] {$k_{s,J}$};
				
				\draw [|-,shift = {(0.5,0.5)}, red] (1,0.5) -- (2,0.5) node [midway,fill=gray!70] {$k_{t,J}$};
				\draw [-|,shift = {(0.5,0.5)}, red] (1.9,0.5) -- (3,0.5);
				
				\draw [shift={(0.5,0.5)}, densely dotted] (3,0) -- (3,1) -- (5,1);
				\draw [shift={(0.5,0.5)}, densely dotted] (0,3) -- (1,3) -- (1,5);
				
				\draw [shift={(0.5,0.5)}] (4,4) node {$K_{\widetilde{\Omega}}$};
				\draw [shift={(0.5,0.5)}] (2,2) node {$K_{J}$};
				\draw [shift={(0.5,0.5)}] (1.5,4.6) node {$K_{J_{+}}$};
				\draw [shift={(0.5,0.5)}] (0.5,2.6) node {$K_{J_{-}}$};
				\draw [shift={(0.5,0.5)}] (0.5,4.6) node {$K$};
				
				\draw [shift={(0.5,0.5)},fill] (3.5,0.5) circle [radius=0.05];
				\draw [shift={(0.5,0.5)}] (3.5,0.5) node [below] {{\scriptsize $(s,t)$}};
			\end{tikzpicture}
			\caption{The partial covariance $K_{\widetilde{\Omega}}$} 
			\label{fig:uniqcan}
		\end{center}
	\end{figure}
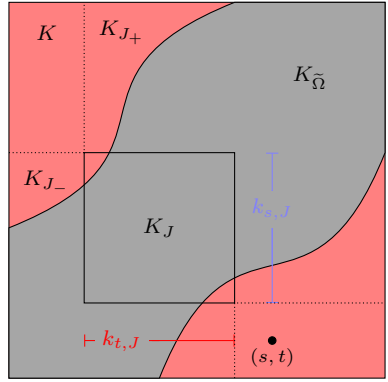
\end{proof}

\begin{proof}[Proof of Theorem \ref{thm:construction-nearly}]
	Let $\widetilde{\Omega} \subset \Omega$. Then, every pair $(s,t) \in \Omega^c$ separated by $J \subset I$ in $\Omega$ is also separated by $J \subset I$ in $\widetilde{\Omega}$. Therefore the canonical completion of $K_{\star}|_{\Omega}$ is equal to $K_{\star}(s,t)$. Since this is true for every $(s,t) \in \Omega^{c}$, it follows that the canonical completion of $K_{\star}|_{\Omega}$ is $K_{\star}$.
	
	For the converse, if the canonical completion of $K_{\star}|_{\Omega}$ is $K_{\star}$ then $K_{\star}$ is the completion of a partial covariance on the nearly serrated domain $\Omega$ which is a improper subset of $\Omega$. 
\end{proof}

\section{Additional Graphs}
Figure \ref{fig:covs} depicts plots of covariance completions in the case of regular and sparse observations for $K_{1}, K_{2}$ and $K_{3}$. 
\begin{figure}[h]
	\centering
	\includegraphics[clip, trim=0cm 6cm 0cm 0cm, width=0.95\textwidth]{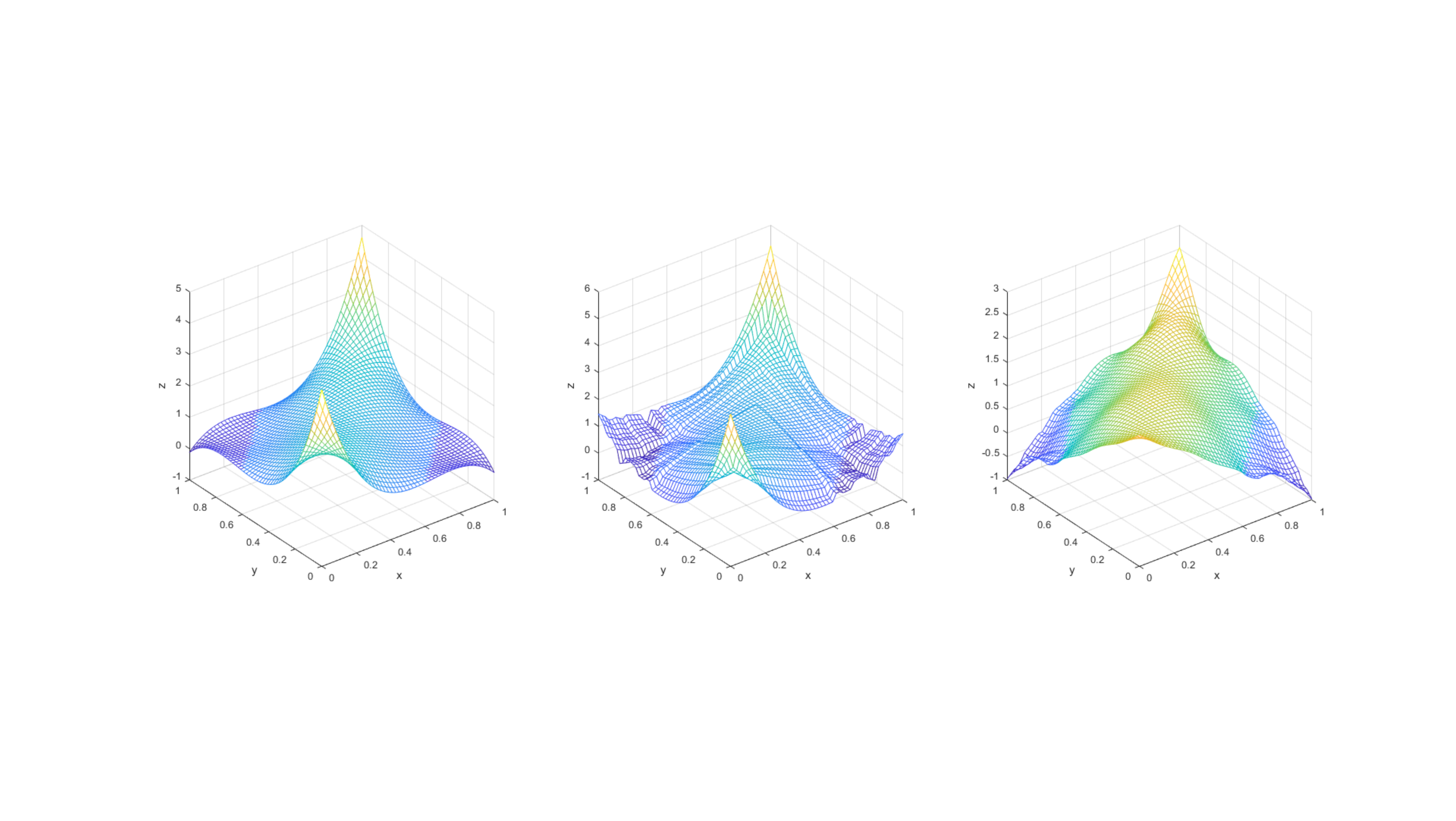}	
	\includegraphics[width=0.95\textwidth]{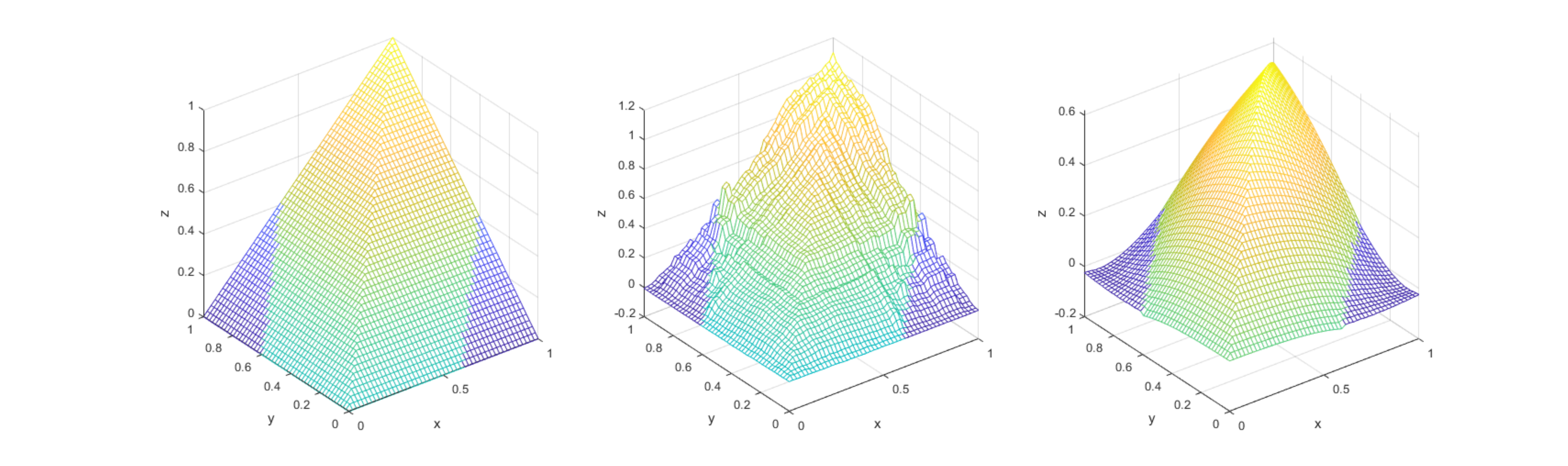}
	\includegraphics[width=0.95\textwidth]{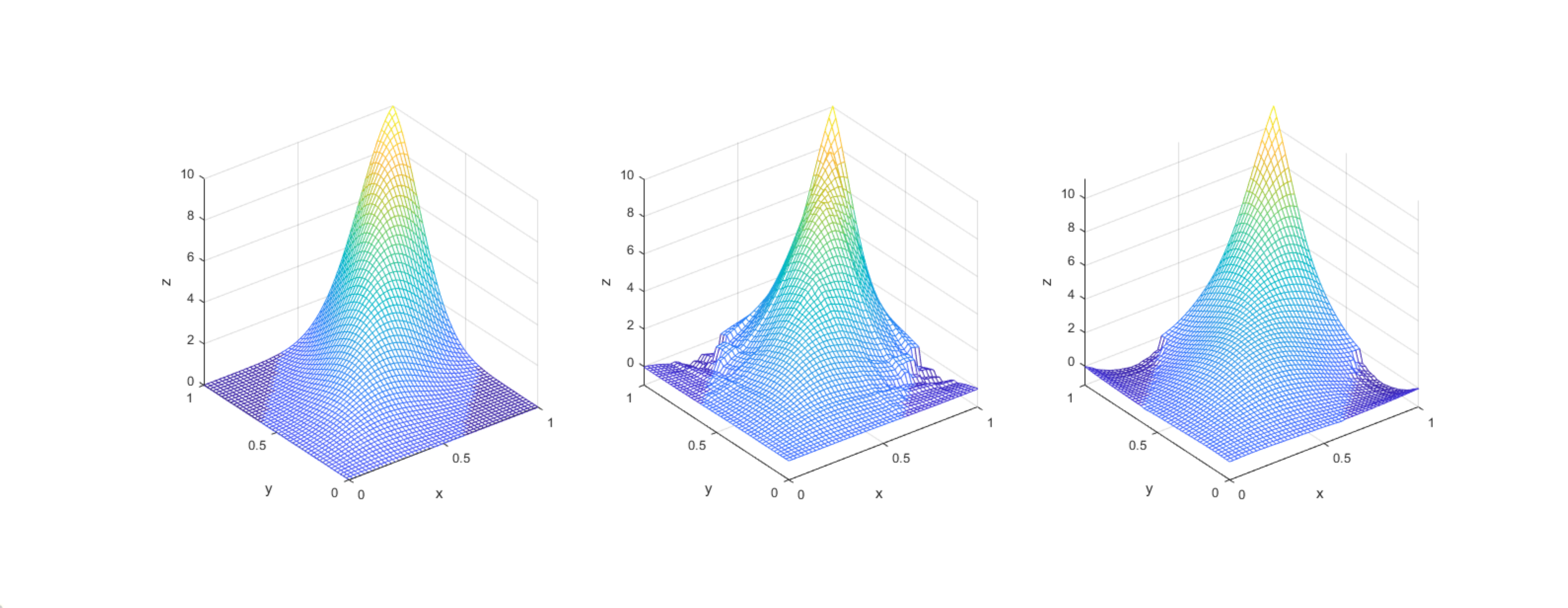}
	\caption{Covariance Completions of $K_{1}$ (top), $K_{2}$ (middle) and $K_{3}$ (bottom) for $m=9$ and $N = 300$. For every row, the plot on the left is the true covariance, the plot in the middle is the completion from regular observations using pairwise estimatand on the right is the completion for sparse observations.}
	\label{fig:covs}
\end{figure}

\end{document}